\documentclass[10pt,leqno,a4paper]{article}
\usepackage{color}
\usepackage{amsfonts}
\usepackage{amsmath, amsfonts, amsthm, amssymb, amscd}
\usepackage[mathcal]{euscript}
\usepackage{mathrsfs}
\usepackage{dsfont}
\usepackage{ulem}
\usepackage{bbm}
\usepackage{graphicx}
\usepackage{mathtools}
\usepackage{slashed}
\usepackage{lineno} 
\usepackage{verbatim}
\usepackage{enumitem}
\usepackage{appendix}
\usepackage{hyperref}
\usepackage{tikz-cd}

\newtheorem{theorem}{Theorem}[section]
\newtheorem{proposition}[theorem]{Proposition}
\newtheorem{lemma}[theorem]{Lemma}
\newtheorem{corollary}[theorem]{Corollary}
\newtheorem{definition}[theorem]{Definition}
\newtheorem{remark}[theorem]{Remark}

\newtheorem{claim}[theorem]{Claim}
\numberwithin{equation}{section}

\usepackage{multirow}
\usepackage{tabularx}
\usepackage{lscape}
\newcommand \Rbb {\mathbb{R}}

\newcommand \Xcal{\mathcal{X}}

\newcommand \del \partial

\newcommand {\vep}{\varepsilon}

\makeatletter
\def\hlinew#1{%
  \noalign{\ifnum0=`}\fi\hrule \@height #1 \futurelet
   \reserved@a\@xhline}
\makeatother

\title{Flat Standing Sphere Blow-up Solutions for the Nonlinear Heat Equation
}

\begin{document}
\maketitle

\centerline{Senhao Duan$^{(1)(2)}$}
%, Nejla Nouaili$^{(1)}$ and Hatem Zaag$^{(2)}$} 
\medskip
{\footnotesize
	\centerline{ $^{(1)}$ CEREMADE, Universit\'e Paris Dauphine, Paris Sciences et Lettres, France}
 \centerline{ $^{(2)}$Universit\'e Sorbonne Paris Nord,
			LAGA, CNRS (UMR 7539), F-93430, Villetaneuse, France.}
	}


\begin{abstract}
In this paper, we prove the existence of a singular standing sphere blow-up solution for the nonlinear heat equation with radial symmetry. This solution develops a finite-time singularity on a fixed-radius sphere and exhibits a flat blow-up profile. Our construction refines the method developed by Merle and Zaag \cite{MZJEMS24} which reduces the infinite-dimensional dynamics to a finite-dimensional problem in radial case. The solution satisfies explicit asymptotics near the singular ring and remains regular elsewhere. %%This work extends previous one-dimensional constructions to higher dimensions and provides new insight into the geometry of singularities in supercritical parabolic equations.
\end{abstract}
\tableofcontents

\section{Introduction}
In this paper, we consider the following nonlinear heat equation,
\begin{equation}\label{eq: NLH introduction}
	\left \{
	\begin{array}{lll}
		\partial_t U&=&\Delta U+|U|^{p-1}U,\\
		U(.,0)&=&U_0 \in L^{\infty}(\Rbb^d),
	\end{array}
	\right.
\end{equation}
\noindent where $U(t):x\in \Rbb^d \to U(x,t)\in \mathbb{R}$, $d \geq 2$ and $p>1$ with $(d-2)p\leq d+2$.

Equation \eqref{eq: NLH introduction}
 is considered as a model for many physical situations such as heat transfer, combustion theory, thermal explosion, etc. (see more in Kapila \cite{KSJAM80}, Kassoy and Poland \cite{KPsiam80}, Bebernes and Eberly \cite{BEbook89}).

The local Cauchy problem for equation \eqref{eq: NLH introduction} can be solved within ${L^\infty(\Rbb^d)}$. Furthermore, it can be shown that the solution $U(t)$ exists either in the interval $[0, +\infty)$ or within $[0, T)$ where $T < +\infty$.  In the latter case, $U$ undergoes a finite-time blow-up. This indicates that there exists a sequence $(x_n,t_n)\to (x_0,T)$ while $n \to \infty$ such that
$$\lim_{n \to \infty}|U(x_n,t_n)|=\infty.$$ 
We then call, in this case, $x_0$ the blow-up point and $T$ the blow-up time. 

While the nonlinear heat equation has been widely studied over the past six decades, many aspects of its blow-up dynamics, particularly those involving non-trivial geometries, remain open. Here, we focus on constructing solutions exhibiting finite-time blow-up along a spherical blow-up set, with a prescribed flat asymptotic profile. Consequently, our references will be limited to prior work within this scope.
 Interested readers may refer to Quitener and Souplet \cite{QSbook07} for a comprehensive sight of research on equation \eqref{eq: NLH introduction}.

In the Sobolev subcritical case ($(d-2)p < d+2$), pioneering works by Giga and Kohn \cite{GKcpam85} and Giga, Matsui, Sasayama \cite{GMSmmas04} yielded the first insight into the asymptotics of blow-up. Following \cite{GKcpam85}(\cite{GMSmmas04}), we introduce the following similarity variables:
 \[
W_{x_0}(Y, s)=(T-t)^{\frac{1}{p-1}} U(x, t) \text { where } Y=\frac{x-x_0}{\sqrt{T-t}} \text { and } s=-\log (T-t) .\]
 \begin{comment}
 By introducing the following self-similar variables:
 \begin{equation}
\begin{aligned}
y&=\frac{x-a}{\sqrt{T-t}},\\
s&=-\log(T-t),\\
W(y,s)&=(T-t)^{\frac{1}{p-1}}U
(x,t)\\
\end{aligned}
\end{equation}

  \end{comment}
   They established that, up to changing $U$ by $-U$,
   for each $K>0$, the following holds on a compact set $\{|Y|\leq 2K\sqrt{s}\}$: 
 
\begin{equation}\label{eq:GK decay of w-k}
W_{x_0}(Y,s)\rightarrow \kappa, \mbox{as $s\to\infty$ with $\kappa= (p-1)^{-\frac{1}{p-1}} $.}
\end{equation}
 For the simplicity of the exposition, we consider the two dimensional case.
In these coordiantes, Vel\'azquez \cite{VELtams93} (see also Filippas and Kohn \cite{FKcpam92} or Filippas and Liu \cite{FLaihn93}) observed that unless $W_{x_0}\equiv\kappa$, convergence \eqref{eq:GK decay of w-k} can be refined as following:
 \begin{equation}\label{eq:RDFWTK}
W_{x_0}(Y, s)-\kappa \sim Q_{x_0}(Y, s) \text { as } s \rightarrow \infty
\end{equation}
uniformly on a compact set, where $Q_{x_0}(Y,s)$ called ``blow-up'' profile,
decays in time and captures the correction of the convergence to constant $\kappa$.

Then, two typical forms of profiles $Q_{x_0}$ arise:
\begin{enumerate}
    \item ``$\frac{1}{s}$'' profile:
    \begin{equation}\label{eq: profile generic}
      Q_{x_0}(Y, s)=-\frac{\kappa}{4 p s} \sum_{i=1}^l h_2\left(Y_i\right),
    \end{equation}
    where $l=1$ or $2$, up to some rotation of coordinates.
    \item ``exponential $s$'' profile: 
    
    \begin{equation}\label{profile-generic-flat}
    Q_{x_0}(Y, s)=-e^{-\left(\frac{m}{2}-1\right) s} \sum_{j=0}^m C_{m, j} h_{m-j}\left(Y_1\right) h_j\left(Y_2\right)
    \end{equation}
    as $s \rightarrow \infty$, for some even integer $m=m(a) \geq 4$, where $Y=\left(Y_1, Y_2\right), h_j(\xi)$ is the rescaled Hermite polynomial defined by
    
    $$
    h_j(\xi)=\sum_{i=0}^{[j / 2]} \frac{j!}{i!(j-2 i)!}(-1)^i \xi^{j-2 i}
    $$
    \end{enumerate}
%%  An extensive literature is devoted to the blow-up profile for equation \eqref{eq: NLH introduction} see V\'elazquez \cite{VELcpde92}, \cite{VELtams93}, \cite{VELiumj93}, and Zaag \cite{ZAAaihp02}, \cite{ZAAcmp02} for partial results.

The existence of the solutions to \eqref{eq: NLH introduction} obeying \eqref{eq: profile generic} or \eqref{profile-generic-flat} is widely studied.
 In the one-dimensional case, Bricmont and Kupiainen \cite{BKcpam88} constructed a solution for \eqref{eq: NLH introduction} that behaves like \eqref{eq: profile generic} and \eqref{profile-generic-flat}. Recently, Duong et al. \cite{DNZArxiv2022} revisited the construction of the \textit{flat} profile using modulation theory.
 
 In the higher dimensional case, one may refer to Nguyen and Zaag \cite{NZ2017}, for constructed solutions showing a refinement of behavior \eqref{eq: profile generic} in 2 dimensions with $l=2$. In the supercritical case, we have an example of a single-point blow-up with a degenerate profile given by Merle, Raph\"ael, and Szeftel \cite{MRSIMRNI20}. More recently, Merle and Zaag \cite{MZJEMS24} have provided an example of a blow-up solution with a novel cross-shaped blow-up profile.

The methods used in \cite{BKcpam88} were enhanced after by Merle and Zaag \cite{MZdm97} using a more geometrical approach. Generally speaking, regarding the linearized equation around the profile, Merle and Zaag cope with this problem in two steps:
\begin{enumerate}
    \item Reduce the problem into a finite-dimensional one.
    \item Solve the finite-dimensional problem with a topological shooting argument. 
\end{enumerate}

This powerful method is then applied to many other different fields. Interested readers are invited to see \cite{MZjfa08}
and \cite{DNZMAMS20} for an application in the complex Ginzburg-Landau equation and a more direct way to accomplish the first step in \cite{MZdm97}. Besides, D\'avila, Del Pino, and Wei \cite{DDWIM20} applied this method to deal with the formation of the singularities for harmonic map flow.

Let us now consider the analysis of blow-up behaviors involving non-trivial geometries. We would like to mention Mahmoudi, Nouaili and Zaag \cite{MNZNon2016}, for the existence of a blow-up solution to the nonlinear heat equation on a circle;  Rapha\"el \cite{Rdm06}, for the existence and stability of solution blowing up on a sphere for the $L^2$-supercritical nonlinear Schr\"odinger equation and Rapha\"el and Szeftel \cite{RSCMP09} for a analogous result extended to higher dimensions. We also cite the work of Collot et al. \cite{CTNVJFA23} where the authors proved the existence and stability of a collapsing-ring blow-up solution to the Keller-Segel system in 3 dimensions and higher. Recently in \cite{DNZ24}, we constructed a standing sphere blow-up solution of generic type for \eqref{eq: NLH introduction}
 
In this work, our main contribution is to prove the existence of a radially symmetric solution to  \eqref{eq: NLH introduction} that develops finite-time blow-up on a fixed-radius sphere, with a flat asymptotic profile of the following type:
\[
U(x,t)\sim(T-t)^{-\frac{1}{p-1}}f\left(\frac{|x|-r_0}{(T-t)^{1/4}}\right),
\]
where
\[
			f(z)=\left (p-1+\frac{(p-1)^2}{\kappa}z^{4}\right )^{-\frac{1}{p-1}}.
\]
%%{In radial coordinates it can be described as follows:
%% \begin{equation}\label{eq: w assymptotic properties HV}
%%	w_1(y,s)-\kappa \sim -e^{-s} h_4(y) \;\text{as}\; y\to\infty,
%%\end{equation}
%%where $w_1, y ,s$ will be defined later in \eqref{eq: self-similar-variables}.
Our main result is stated as follows:
\begin{theorem}
\label{th: Main theorem}
(Existence of a singular standing solution for equation \eqref{eq: NLH introduction} with prescribed profile).
For any $r_0 > 0$, there exists $T > 0$ such that equation \eqref{eq: NLH introduction} admits a radially symmetric solution $U(x,t)$ defined on $\Rbb^d \times [0,T)$, satisfying:
\begin{enumerate}
\item The solution $U$ blows up in finite time $T$ on the sphere of radius $r_0$.
\item For every $C > 0$,
\begin{equation}
\label{eq:coninthm}
\sup_{\Lambda_C} \left|(T - t)^{\frac{1}{p-1}} U(x, t) - f\left (\frac{|x|-r_0}{(T-t)^{\frac{1}{4}}}\right )\right| \to 0
\quad \text{as } t \to T,
\end{equation}
where
$
		\Lambda_C := \left\{\, ||x|-r_0| \leq C(T-t)^{\frac{1}{4}} \,\right\}, \qquad
		f(z)=\left (p-1+\frac{(p-1)^2}{\kappa}z^{4}\right )^{-\frac{1}{p-1}}.
		$
\item For every $x \in \Rbb^d$ with $|x| \neq r_0$, we have $U(x,t)\to U(x,T)$ as $t\to T$, where near the singular sphere the limiting profile satisfies
\begin{equation}\label{eq: profile}
U(x,T) \sim u^*(|x|-r_0),
\qquad
u^*(\xi)=\left(\frac{(p-1)^2}{\kappa}\xi^{4}\right)^{-\frac{1}{p-1}},
\quad \xi \in \Rbb.
\end{equation}
\end{enumerate}
\end{theorem}
\begin{remark}
    Note that $f$ in \eqref{eq:coninthm} can be written more generally as \\ $\left(p-1+\frac{(p-1)^2}{\kappa}z^{m}\right)^{-\frac{1}{p-1}}$ with $m=2k$, and $k\geq 2$. In this paper, for simplicity of the exposition, we only focus on the case where $k=2$ but it can be also generalized to $k>2$. 
\end{remark}
We proceed in several sections to prove Theorem \ref{th: Main theorem}
%To establish this result, the paper is organized as follows:
\begin{itemize}
\item In Section \ref{formulation of the problem}, we present the problem setting.
\item In Section \ref{SOTP}, we outline the strategy of the proof of Theorem \ref{th: Main theorem}.
\item Section \ref{Dynamics of the equation} is devoted to the dynamics of the equation: we describe the general form of the initial data and sketch the proof of Theorem \ref{th: Main theorem}, omitting the technical details.
\item In Section \ref{Proof of the priori bound}, we establish a crucial a priori bound in self-similar variables.
\item Finally, Sections \ref{study of the projections} and \ref{proof of technical details} provide the technical ingredients required in the proofs of Section \ref{Dynamics of the equation} and Section \ref{Proof of the priori bound}.
\end{itemize}
\begin{comment}
    \begin{theorem}
\label{th: Main theorem}
(Existence of a singular standing solution for \eqref{nlh-introduction} with Prescribed Profile).
There exists \(T > 0\) such that Eq. \eqref{nlh-introduction} has a solution \(u(r, t)\) in \(S \times [0, T)\) such that:
\begin{enumerate}
    \item the solution \(u\) blows up in finite time \(T\)  on the sphere of raius \(r=r_{max}\);
    \item there holds that for all \(R > 0\),
    \[
    \sup_{\Lambda_R} \left|(T - t)^{\frac{1}{p-1}} u(r, t) - f\right| \to 0 \text{ as } t \to T, \quad (7)
    \]
    where \(\Lambda_R := \left\{ |r-r_{max}| \leq R \sqrt{(T - t)|\log(T - t)|}\right\}\) and where the function \(f=(p-1+b\frac{y^2}{s})\) and \(b=\frac{(p-1)^2}{4p}\) 
\end{enumerate}
\end{theorem}
\end{comment}
\section{Formulation of the problem}\label{formulation of the problem}
 To prove Theorem \ref{th: Main theorem} , it suffices to construct a solution blowing up on a sphere of radius $r_0=1$, since the general case of $r_0> 0$ follows by the scaling invariance of the equation:
$$
\lambda\mapsto U_{\lambda}(\xi,\tau)\equiv \lambda^{\frac{2}{p-1}}u(\lambda \xi,\lambda^2 \tau).
$$ 
Our goal is to construct a solution that blows up on the sphere $\{|x|=1\}$, with asymptotic behavior described by \eqref{eq:coninthm}. In what follows, we provide a detailed description of this challenge and outline our strategy for overcoming it.
\subsection{Strategy for bounding the ``radial singularity''}
We consider radially symmetric solutions and define for all $t \geq 0$ and $r\geq 0$: \begin{equation}\label{eq: transform between 1d to dd}
	\begin{aligned}
		u(r,t)&=U\left((r,0,\cdots,0),t\right),
	\end{aligned}
\end{equation}
 so that  \begin{equation}\label{eq: transform between dd to 1d}
 	\begin{aligned}
 	\forall x\in \Rbb^d,\quad	U(x,t)&=u(|x|,t).\\
 	\end{aligned}
 \end{equation}
Substituting in  \eqref{eq: NLH introduction}, the equation becomes a 1-dimensional one:
\begin{equation}\label{eq: NLH 1d radial coordinates}
\del_t u=\del^2_{r} u+\frac{d-1}{r}\del_r u+|u|^{p-1}u,
\end{equation}
which introduces a singularity at $r=0$ due to the $\frac{1}{r}$ term. As a matter of fact, by ``radial singularity'', we mean the term $\frac{d-1}{r}\del_r u$. In order to handle this term, we divide the space into two regions:
\begin{itemize}
\item Inner region: Near the origin, where the solution remains regular 
\item Outer region: Away from the origin, where blow-up occurs.
\end{itemize}
This approach is inspired by \cite{MNZNon2016} and \cite{DNZ24}, although we must carefully avoid introducing additional singularities that may complicate the global analysis.

Indeed, we introduce two sets of similarity variables corresponding to the different aims. Let $T>0$
 be the expected blow-up time. In the outer region near $\{r=1\}$,
 we define:
\begin{equation}\label{eq: self-similar-variables}
	w_a(y,s)=(T-t)^{\frac{1}{p-1}}u(r,t)\mbox{ with }y=\frac{r-a}{\sqrt{T-t}},\; s=-\log(T-t),
\end{equation}
where $a\geq0$, $s\geq-\log T$ and $y\geq-ae^{s/2}$. In the full coordinates, self-similar variables are given by:
\begin{equation}\label{eq: second group ssv}
W_{x_0}(Y,s)=(T-t)^{\frac{1}{p-1}}U(x,t),\; Y=\frac{x-x_0}{\sqrt{T-t}}\;,s=-\log(T-t),
\end{equation}
where $x\in \Rbb^d$, $s\geq-\log T$ and $Y\in \Rbb^d$.
The first transformation is tailored to constructing the blow-up solutions near the sphere, while the second is used to analyze the behavior globally, especially to prove the $L^\infty$ priori bound
\begin{equation}\label{eqabU}
\|(T-t)^{\frac{1}{p-1}}U\|\leq M.
\end{equation}
\medskip
We will separate the origin and the singular standing sphere by  introducing the following smooth nonnegative cut-off functions:

\begin{equation}\label{eq: def Chi}
    \chi= \left\{
    \begin{array}{rcl}
        0 &&  0\leq\xi \leq \frac{1}{8} ,\\
        && \\       
        1 &&  \xi\geq \frac{1}{4},\\
        \end{array}\right.
        \end{equation}
and
\begin{equation}
    \overline{\chi}= \left\{
    \begin{array}{rcl}
         1 &&  0 \leq\xi \leq\frac{3}{8},\\
        &&\\
       0 && \xi\geq \frac{3}{4}.\\
        \end{array}\right.
        \end{equation}
 Then we get two different region 
 \begin{itemize}
 \item \textbf{The inner region}: $\left\{|x|\leq\frac{3}{8}\vep_0\right\}$ for some $\vep_0$ to be fixed small enough later. Note that we expect no blowup in this region. In order to study this inner region, we define $\overline{U}(x,t)=\overline{\chi}\left(\frac{|x|}{\vep_0}\right)U(x,t)$, for $x\in\mathbb{R}^d$, where $U(x,t)$ is assumed to satisfy the following:
$$\del_t U = \Delta U+|U|^{p-1}U.$$ 
Then for all $x\in \mathbb{R}$, $\overline{U}$ satisfies the following equation:
\begin{equation}
\del_t \overline{U} = \Delta \overline{U} + |U|^{p-1}\overline{U}-2\nabla \overline{\chi}\nabla U - \Delta \overline{\chi} U.
\end{equation}
The function $\overline{U}$ will be controlled using classical parabolic estimates.

\item \textbf{The outer region} $\left\{|x|\geq \frac{\vep_0}{4}\right\}$: In this region, we consider the equation in radial coordinates given by \eqref{eq: NLH 1d radial coordinates}.
Then, $u$ introduced in \eqref{eq: transform between 1d to dd} satisfies the following equation: $\forall t\in [0,T)$, $\forall r\geq 0$
\begin{equation}\label{eq: U }
\del_t u=\del^2_{r} u+\frac{d-1}{r}\del_r u+|u|^{p-1}u.
\end{equation}
%While to construct a self-similar %solution blowups at $r_0=1$, we use %the self-similar transformation %\eqref{eq: self-similar-variables}. 

Using the self-similar transformation \eqref{eq: self-similar-variables} (with $a=1$), we get from \eqref{eq: U } that $w_{1}$  satisfies:
\begin{equation}\label{eq: W}
\del_s w_{1} =\del^2_{y}w_{1}-\frac{1}{2}y\del_{y}w_{1}+e^{-s/2}\frac{d-1}{ye^{-s/2}+1}\del_y w_{1}-\frac{w_{1}(y,s)}{p-1}+|w_{1}|^{p-1} w_{1},
\end{equation}
where $y\in[-e^{s/2},+\infty)$ and $s\in [-\log T,+\infty)$. 

As we should restrict ourselves to the outer region $\left\{y\in \mathbb{R}\left| y\geq e^{s/2}\left(\frac{\vep_0}{4}-1\right)\right. \right\}$ in this part, we introduce for all $s\geq0$ and $y\geq-e^{\frac{s}{2}}$:
\begin{equation}\label{eqdoftw}
\tilde{w} = w_{1}\cdot \chi(\frac{ye^{-s/2} + 1}{\vep_0}),
\end{equation}
where $\chi$ is defined by \eqref{eq: def Chi}. Then from \eqref{eq: W}, $\tilde{w}$ satisfies for all $s\geq0$ and $y>e^{-s/2}$:  
\begin{equation} \label{eq: w-equation}
    \del_s \tilde{w}= \del^2_y \tilde{w}-\frac{1}{2}y\del_y \tilde{w}-\frac{1}{p-1}\tilde{w} + |\tilde{w}|^{p-1}\tilde{w} + e^{-s/2}\frac{d-1}{ye^{-s/2}+1}\del_y \tilde{w}+ F(y,s),
\end{equation}
where $F(y,s) $ is defined as follows:

\begin{equation}	
\begin{split}
F(y,s)&=
w_1\del_s \chi - 2\del_y \chi\del_y w_1- w_1\del_y^2\chi\\
&+\frac{1}{2}yw_1\del_y\chi - \frac{d-1}{y+e^{s/2}}w_1\del_y\chi+ |w_1|^{p-1}w_1(\chi-\chi^{p}).
\end{split}
\end{equation}
Extending $\tilde{w}$ and $F$ by setting $\tilde{w}=F(y,s)=0$, for all $s\geq s_0$ and $y\leq e^{-s/2}$, we see that $\tilde{w}$ is a solution of the following modification to \eqref{eq: w-equation}, for all $s\geq 0$ and $y\in\Rbb$:
\begin{equation} \label{eq:MEFWT}
\begin{aligned}
    \del_s \tilde{w}= &\del^2_y \tilde{w}-\frac{1}{2}y\del_y \tilde{w}-\frac{1}{p-1}\tilde{w} + |\tilde{w}|^{p-1}\tilde{w} \\
    &+ e^{-s/2}\frac{d-1}{ye^{-s/2}+1}\chi\left(\frac{ye^{-s/2}+1}{2\vep_0}\right)\del_y \tilde{w}+ F(y,s),
    \end{aligned}
\end{equation}
To see that, we recall by definition \eqref{eq: def Chi} that $\tilde{w}\equiv F(y,s)\equiv0$ in the transition region $y\in (e^{-s/2},(\frac{\vep_0}{8}-1)e^{-s/2})$. Note that in particular, the coefficient of $\del_y\tilde{w}$ is no longer singular, again by definition \eqref{eq: def Chi} of $\chi$.
\end{itemize}

\subsection{Expectations for the profile}
In this section, we heuristically give a guess on the assypmtotic behavior of \eqref{eq:RDFWTK} type satisfied by $\tilde{w}$ in one space dimensional case.
We separate this problem in two cases, $y$ belongs to a compact sets and $y$ in a larger zone, and discuss them separately in the following two subsection.
\subsubsection{Profile for $y$ in compact sets}
Remark that equation \eqref{eq:MEFWT} is in fact a perturbation of
\begin{equation}\label{eq:ESIBKHV}
    \del_s \tilde{w}= \del^2_y \tilde{w}-\frac{1}{2}y\del_y \tilde{w}-\frac{1}{p-1}\tilde{w} + |\tilde{w}|^{p-1}\tilde{w}. 
\end{equation}
Therefore, we expect that solutions to \eqref{eq:MEFWT} have parallel assymptotic behavior to solutions to \eqref{eq:ESIBKHV}. Hence, studies \cite{BKcpam88} and \cite{HVaihn93} on \eqref{eq:ESIBKHV} implies that there are two possible behavior on some compact sets for solutions to \eqref{eq:MEFWT}:
\begin{itemize}
    \item \textbf{Case 1 (generic):} $\tilde{w}=\kappa-\frac{\kappa}{4ps}h_2(y)+O(\frac{1}{s})$,
    \item \textbf{Case 2 (flat):} $\tilde{w}=\kappa- e^{(1-\frac{m}{2})s}h_m(y)+O(e^{-\frac{ms}{2}}),$ when $m\geq 4$ is an even number.
\end{itemize}
The existence of solutions to \eqref{eq:MEFWT} obeying behavior described in Case 1 is already studied in the previous work \cite{DNZ24}. In this paper, we will try to construct a solution in Case 2 with $m=4$. That means we will focus on the case where $\tilde{w}$ obeying:
\begin{equation}\label{eq:WTAB}
    \tilde{w}=\kappa-e^{-s}h_4(y)+O(e^{-2s}).
\end{equation}
\subsubsection{Profile in larger region}
If $|y|\leq C$, then \eqref{eq:WTAB}, we should only get a convergence to constant $\kappa$ as $s\to \infty$, other than a shape of such kind of solutions.\\
How do we get a shape? From \eqref{eq:WTAB}, we think that the reduced variable $z = e^{-s/4}y$ may play a role. 
Indeed, by \cite{HVaihn93}, we suppose that if $\tilde{w}$ satisfies \eqref{eq:WTAB}, then,
\begin{equation}\label{eq:CTTV}
\forall K>0,\quad \sup_{|y|\leq Ke^{s/4}}|\tilde{w}(y,s)-\tilde{\varphi}(y,s)|\to 0\quad \text{as } s\to0,
\end{equation}
where by\eqref{eq:coninthm},
\begin{equation}\label{eq:TVP}
    \tilde{\varphi}(y,s)=(p-1+e^{-s}y^4)^{-\frac{1}{p-1}}=f\left(\frac{|x|-r_0}{(T-t)^\frac{1}{4}}\right).
\end{equation}
Therefore, our basic idea will be to prove that equation \eqref{eq:MEFWT} admits a solution obeying \eqref{eq:WTAB} and \eqref{eq:CTTV}, since \eqref{eq:MEFWT} seems to be a small perturbation of \eqref{eq:ESIBKHV}.
\section{Strategy of the proof}
\label{SOTP}
From Section \ref{formulation of the problem}, we recall that our goal is to construct $U(x, t)$, a solution to equation \eqref{eq: NLH introduction} which is radially symmetric and defined for all $(x, t) \in \mathbb{R}^d \times[0, T)$ for some small enough $T>0$, such that in radial coordinates 

$$
w_{1}(y, s)-\kappa \sim-e^{-s} h_4\left(y\right)\text { as } s \rightarrow \infty,
$$

uniformly on compact sets, where $w_1(y, s)$ is the similarity variables' version defined in \eqref{eq: self-similar-variables}. Moreover, as explained in Section \ref{formulation of the problem}, in the inner region, the solution $U(x,t)$ will remain small, due to a parabolic estimation. Therefore, it is equivalent to say that $\tilde{w}$ defined in \eqref{eqdoftw} satisfies:
$$
\tilde{w}(y, s)-\kappa \sim-e^{-s} h_4\left(y\right)\text { as } s \rightarrow \infty.
$$
We use a classical way to construct this kind of solution. A natural way would be to linearize equation \eqref{eq:MEFWT} around the profile $\tilde{\varphi}$ given in \eqref{eq:TVP}. Unfortunately, this $\tilde{\varphi}$ does not match with \eqref{eq:WTAB}, in the sense that we don't have    
$$
\tilde{\varphi}=\kappa-e^{-s}h_4+O(e^{-2s}).
$$
In fact, as we will show below in Section \ref{COP}, it is possible to modify $\tilde{\varphi}$ into $\varphi$, so that \eqref{eq:CTTV} still holds and also \eqref{eq:WTAB}! For this, we will use an original method introduced by Merle and Zaag in \cite{MZJEMS24} to handle the degnerate profiles for \eqref{eq: NLH introduction} in higher dimensions. 

Then, by defining 
\begin{equation}\label{eq:Dq}
q(y, s)=\tilde{w}(y, s)-\varphi(y, s),
\end{equation}
 where $\varphi$ is the modified expansion of the profile. We further specify our goal by requiring that $q(y, s)$ is small in some sense that will shortly be given in Definition \ref{Def: shrinking set}

In order to achieve this goal, we need to write then understand the dynamics of the equation satisfied by $q(y, s)$ near 0 . This is done below in Section \ref{SDOq}.
Before that some preliminary estimation have to be done, as we explained. Here is the following plan of this section:
\begin{itemize}
    \item In Section \ref{POAB}, we show how we intend to prove the apriori bound \eqref{eqabU}.
    \item In Section \ref{COP}, we give the precise expression of the modified profile $\varphi$ and prove some good properties satisfied by $\varphi$.
    \item In Section \ref{SDOq} and \ref{control of potential, non linear term}, we show how we intend to analyze the dynamics of $q$ defined above in \eqref{eq:Dq}.
    \item Finally, in Section \ref{DOSS}, we give a bootstrap regime \eqref{Def: shrinking set} in which we successfully construct solutions admitting the behavior described in Theorem \ref{th: Main theorem}.
\end{itemize}
\subsection{Proof of the apriori bound \eqref{eqabU}}\label{POAB}
As we explained before, after linearizing the equation \eqref{eq: w-equation}, the classical approach consists of two steps (see for example \cite{MZdm97}):
\begin{enumerate}
    \item Reduce the problem into a finite-dimensional one.
    \item Solve the finite-dimensional problem with a topological argument. 
\end{enumerate}
In fact, these two steps are easily carried out if we have the following apriori bound
$$
\forall t\in[0,T),\;\; \|(T-t)^{\frac{1}{p-1}}U(x,t)\|_{L^{\infty}(\Rbb^d)} \leq M.
$$
 To derive this bound, there is a interesting method in Section 5 of \cite{MZJEMS24} that we prefer to use in the present work. In fact, as proved in Section 5 of \cite{MZJEMS24}, by definition \eqref{eq: second group ssv}, it is equivalent to show that
 $$
\forall x_0\in\Rbb,\; |W_{x_0}(0,s)|\leq M.
 $$
 However, for this particular method in \cite{MZJEMS24}, it is more convenient to use $W_{x_0}(Y,s)$ rather than $w_{r_0}(y,s)$. Therefore, assuming $|x_0|=r_0$,  the following is needed :
 \[
 W_{x_0}(Y,s)=w_{r_0}((|Ye^{-s/2}+x_0|-r_0)e^{s/2},s).
 \]
 (this directly follows by definition \eqref{eq: self-similar-variables} and \eqref{eq: second group ssv}, as we will show below in Claim \ref{CRBWaw}, with $r_0=1$). 
 Note that function $(|Y+x_0|e^{-s/2}-1)e^{s/2}$ is not $C^1$. It brings us more difficulties in derive a bound on $W_{x_0}(Y,s)=(T-t)^{-\frac{1}{p-1}}U(x,t)$. 
\iffalse      
According to the explanation of strategy above, we organize this paper as follows:
\begin{itemize}
\item In Section \ref{SOTP}, we give all the details 
in the settings explained in Section \ref{formulation of the problem}.
\item Then, in Section \ref{Dynamics of the equation} we give the proof of Theorem \ref{th: Main theorem} without technical details and solve the finite dimension problem. 
\item In Section \ref{Proof of the priori bound}, we prove a crucial $L^{\infty}$ bound on the solution in similarity variables.
\item Finally, in Section \ref{study of the projections} and Section \ref{proof of technical details} we conclude by giving the proofs of propositions cited in Section \ref{Dynamics of the equation}.

\end{itemize}
\fi
\medskip 
\subsection{Determination of a modified version of the profile}\label{COP}

 %As in \cite{MZcpam98} and \cite{BKcpam88}, normally, to construct a solution with certain profile $\varphi$, we will linearize the equation around it. Then the problem  comes, if we linearize the equation around the "flat" profile: 
%\[
%\varphi= \left(p-1+\frac{(p-1)^2}{\kappa}y^4e^{-s}\right)^{-\frac{1}{p-1}}
%\] 
%whose Taylor expansion is written as
%\[
%%\varphi= \kappa-e^{-s}y^{4}+O(e^{-2s}),
%\] 
%one may then remark the inconsistency with the prediction of the behavior \eqref{profile-generic-flat}. 
In this section, we will use a tricky linearization inspired in \cite{MZJEMS24}. We therefore introduce a polynomial perturbation $P$ for the profile $\tilde{\varphi}$:
\[
\varphi=\left(\frac{1+e^{-s}P(y)}{p-1+\frac{(p-1)^2}{\kappa}y^4e^{-s}}\right)^{\frac{1}{p-1}},
\]
where $P(y)=\frac{p-1}{\kappa}(y^4-h_4(y))$(see \eqref{eq: definition P(y)} below for a justification). This ensures that the leading order term for the expansion aligns precisely with \eqref{profile-generic-flat}.
The properties of $\varphi$ and the choice of $P$ will be rigorously justified later in Lemma \ref{lemma: good properties of phi}, showing that $\varphi$ remains positive, bounded, and decays appropriately as $|y|\to\infty$.

In proximity to the ring $\{|x|=1\}$, we define the perturbation $q$ is defined by \begin{equation}\label{eq: w decompose-in-terms-q-phi}
\tilde{w}=\varphi+q,
\end{equation} where \begin{equation}\label{eq: profile form}
\varphi=\left(\frac{1+e^{-s}P(y)}{p-1+\frac{(p-1)^2}{\kappa}y^4e^{-s}}\right)^{\frac{1}{p-1}}=\left(\frac{E}{D}\right)^{\frac{1}{p-1}}.
\end{equation} A polynomial $P(y)$ is selected such that $\varphi$ remains positive, such that \begin{equation}\label{eq:flat profile}
\varphi=\kappa-e^{-s}h_4+O(e^{-2s}) .
\end{equation} The expansion of equation \eqref{eq: profile form} yields 
\begin{equation}\label{profile expansion}
\varphi=\kappa + e^{-s}\left(\frac{\kappa}{p-1}P(y)-y^4\right)+O(e^{-2s}).
\end{equation} Through identification with \eqref{eq:flat profile}, it follows that
\begin{equation}\label{eq: definition P(y)}
P(y)=\frac{p-1}{\kappa}(y^4-h_4(y)).
\end{equation} We hereby assert the following:
\begin{lemma}\label{lemma: good properties of phi}
Consider $s \geq s_{2}$ for some large enough  $s_{2} \geq 0$. Then:
\begin{itemize}
    \item[(i)] $1+e^{-s}P(y) \geq \frac{1}{2}$, hence $\varphi(y,s)$ is well defined and positive.
    \item[(ii)] $\|\varphi(\cdot,s)\|_{L^{\infty}} \leq \kappa + C_0e^{-\frac{s}{3}}$, for some $C_0 > 0$.
    \item[(iii)] $\|\partial_y\varphi(\cdot,s)\|_{L^{\infty}} \leq C_0e^{-\frac{s}{4}}$.
    \item[(iv)] $\varphi(y,s) \to 0$ as $|y| \to \infty$.
\end{itemize}
\end{lemma}
\begin{proof}
Consider $s\geq 0$ to be taken large enough.\\
(i) It is enough to show that $P(y)$ is bounded from below. Since 
\begin{equation}
h_4(y)=y^4-12y^2+12,
\end{equation}
it follows from \eqref{eq: definition P(y)} that
\begin{equation} \label{eq: estimation P(y)}
P(y)=\frac{12(p-1)}{\kappa}(y^2-1)\geq -\frac{12(p-1)}{\kappa},
\end{equation}
which implies that by taking $s$ large enough, we have $1+e^{-s}P(y)\geq \frac{1}{2}$, \\ and hence $\varphi$ is well defined in \eqref{eq: profile form}.\\
(ii) Using \eqref{eq: estimation P(y)}, we see from \eqref{eq: profile form} that
\begin{equation}\label{eq: priori estimation on phi}
\left|\varphi^{p-1}\right|\leq C\frac{N_1+N_2}{p-1+\frac{(p-1)^2}{\kappa}y^4e^{-s}},
\end{equation}
where 
\begin{equation}
N_1=1+e^{-s},\quad N_2=e^{-s}|y|^2.
\end{equation}
Since $p-1+\frac{(p-1)^2}{\kappa}y^4e^{-s}\geq p-1$, it follows that
\begin{equation}\label{eq: estimation on first part of phi}
	\frac{N_1}{p-1+\frac{(p-1)^2}{\kappa}y^4e^{-s}}\leq \frac{1+e^{-s}}{p-1}
	\end{equation}
Introducing \(z=e^{-\frac{s}{4}}y\), and denoting \(X = |z|^2\), we see that
\begin{align}\label{eq: estimation on second part of phi}
\frac{N_2}{D} \leq \frac{C e^{-\frac{s}{2}}|z|^2}{1 +|z|^4 } \leq \frac{C}{2}e^{-\frac{s}{2}}, 
\end{align}

Collecting the estimates in \eqref{eq: priori estimation on phi}, \eqref{eq: estimation on first part of phi} and \eqref{eq: estimation on second part of phi} then recalling the definition \(\kappa=(p-1)^{- \frac{1}{p-1}}\), we get the desired conclusion.
 
(iii) By definition \eqref{eq: profile form} of \(\varphi(y,s)\), we write for all \(s \geq 1\) and \(y \in \mathbb{R}\)
\begin{align}
(p - 1) \log \varphi = \log E - \log D
\end{align}
where \(E\) and \(D\) are defined in \eqref{eq: profile form}  (note that \(D > 0\) by definition, and that \(E > 0\) from item (i) of this lemma). Taking the gradient, we see that
\begin{align}
\del_y \varphi = \frac{\varphi}{p - 1} \left( \frac{\del_y E}{E} - \frac{\del_y D}{D} \right). \label{2.19}
\end{align}
 
Since
\begin{align}
|\del_y D| \leq C e^{-s} |y|^3=D_1, \label{eq: grad D estimation}
\end{align}
 we write from item (ii) of this lemma
\begin{align}
|\del_y \varphi| \leq C \left[ \frac{|\del_y E|}{E} + \frac{D_1}{D} \right]. \label{eq: grad phi estimation}
\end{align}
 
Noting that from \eqref{eq: estimation P(y)}
\begin{align}
E \geq \frac{E_0}{C} \quad \text{where} \quad E_0 = 1 + e^{-s}y^2 , \label{eq:estimation on E}\\
|\del_y E| \leq C e^{-s}|y| , \label{eq: estimation on grad E}
\end{align}
from \eqref{eq: estimation P(y)}, using the definitions \eqref{eq: profile form}  and \eqref{eq: grad D estimation} of \(D\) and \(D_1\), then proceeding a similar proof to the one of item (ii) of this lemma, we show that
\begin{align}
\frac{|\del_y E|}{E} \leq \frac{|\del_y E|}{E_0} \leq C e^{-\frac{s}{2}} \text{ and } \frac{D_1}{D} \leq C e^{-\frac{s}{4}}, \label{2.22}
\end{align}
in particular
\begin{align}\label{Gradient D estimation}
\frac{|\del_y D|}{D} \leq C e^{-\frac{s}{4}},
\end{align}
and the estimate on \(\del_y \varphi\) in item (ii) follows.
 
(iv) From \eqref{eq: profile form} and \eqref{eq: definition P(y)}, we see that the numerator of the fraction in \eqref{eq: profile form} is a polynomial of degree 2, whereas the denominator is of degree 4. Thus, the conclusion follows.
 
This concludes the proof of Lemma \ref{lemma: good properties of phi}.
\end{proof}
\subsection{Dynamics of $q$}\label{SDOq}
With the decomposition \eqref{eq: w decompose-in-terms-q-phi}, and Lemma \ref{lemma: good properties of phi}, the problem is then reduced to constructing a function $q$ satisfying 
\begin{equation}\label{eq: q target}
	\lim_{s\to \infty}\sup_{y\in \Rbb} |q(y,s)|=0.
\end{equation}
The equation for $q$ is as follows for all $y\in\Rbb$:
\begin{equation}\label{q's equation}
\del_s q = (\mathcal{L}+V)q+H(y,s)+\del_y G(y,s)+R(y,s)+B(y,s)+N(y,s)
\end{equation}
where
\begin{equation} \label{eq: linear operator}
\mathcal{L}=\del^2_y-\frac{1}{2}y\cdot\del_y +1,\quad V= p\varphi^{p-1}-\frac{p}{p-1},
\end{equation}
\begin{equation}
B(y,s)=|\varphi+q|^{p-1}|\varphi+q|-\varphi^p-p\varphi^{p-1}q,
\end{equation}
and
\begin{equation}\label{definition of B,R,F,N}
\begin{aligned}
	R(y,s)&= \del^2_y \varphi-\frac{1}{2}y\del_y \varphi-\frac{1}{p-1}\varphi+\varphi^p-\del_s \varphi,\\
	H(y,s)&=w_1(\del_y^2 \chi+\del_s \chi+\frac{1}{2}y\del_y\chi\del_y\chi)+ |w_1|^{p-1}w_1(\chi-\chi^{p}),\\
	G(y,s)&= -2\del_y\chi w_1,\\
	N(y,s)&= \frac{d-1}{y+e^{s/2}}w_1\del_y\chi.
\end{aligned}
\end{equation}
The control of $q$ near the collapsing ring $\{r=1\}$ is based on the following useful properties:
\medskip

\noindent\textbullet \textbf{ Apriori $L^{\infty}$ bound}\\
\begin{equation}
\forall s\geq s_0,\;\;\|w_1\|_{L^{\infty}(\Rbb)}\leq M,
\end{equation}
for some $M>0$.
\paragraph{\noindent\textbullet \textbf{ Spectral properties of the linear operator $\mathcal{L}$}\\}

The operator $\mathcal{L}$ is self adjoint on $\mathcal{D}(\mathcal{L})\subset L^{2}(\mathbb{R},\rho_1 dy)$ with 
\begin{equation}\label{eq: def of rho}
\rho_d(y)=\frac{e^{-\frac{|y|^2}{4}}}{(4\pi)^{d/2}},
\end{equation}
The spectrum of $\mathcal{L}$ is 
\begin{equation}\label{eq: spectrum}
	spec(\mathcal{L})=\left\{1-\frac{m}{2}|m\in \mathbb{N}\right\}.
\end{equation}
All the eigenvalues are simple and the corresponding eigenfunctions are derived from Hermite polynomials:
\begin{equation}\label{eq:definition of hm}
  h_m(y)=\sum_{n=0}^{[\frac{m}{2}]} \frac{m!}{n!(m-2n)!}(-1)^ny^{m-2n}. 
\end{equation}
the polynomials $h_m$ satisfies 
$$\int_\mathbb{R} h_mh_nd\mu=2^nn!\delta_{nm},$$
and we introduce $k_m=\frac{h_m}{\|h_m\|^2_{L_{\rho_1}^2(\Rbb)}}$.
\paragraph{\textbullet\ The potential $V$: \\}
It is bounded in $L^\infty$ and small in $L_{\rho_1}^r$ for any $r\geq 1$, in the sense that:
\begin{equation}\label{Potential control}
	\|V(\cdot,s)\|_{L^{\infty}}\leq \frac{p}{p-1}\; \text{and}\; \|V(\cdot,s)\|_{L_{\rho_1}^r}\leq C(r)e^{-s},
	\end{equation}  
\paragraph{\textbullet\ The nonlinear term $B$: \\}
Since $\varphi$ is uniformly bounded in space and time, thanks to item (iii) of Lemma \ref{lemma: good properties of phi}, assuming the following a priori estimate,
\begin{equation}\label{eq: apriori estimation on q}
	\|q\|_{L^{\infty}\left(\Rbb\times\left[s_{0}, \infty\right)\right)} \leq M 
\end{equation}
we easily see that $B(y, s, q)$ is superlinear, in the sense that
\begin{equation}\label{eq: B super linearity}
	|B(y, s, q)| \leq C(M)|q|^{\bar{p}} \text { where } \bar{p}=\min (p, 2)>1 
\end{equation}
for all $y \in \Rbb,\; q \in \mathbb{R}$ and large $s$.
\paragraph{\textbullet\ The remainder term $R$: \\}By comparing the expression \eqref{definition of B,R,F,N} of $R(y, s)$, with the expression of equation \eqref{eq: w-equation}, we see that $R(y, s)$ measures the quality of $\varphi(y, s)$ as an approximate solution of \eqref{eq: w-equation}. In fact, through a straightforward calculation (see below in Section \ref{Proof-a-priori-estimates}), one can show that
\begin{equation}\label{eq: R rough estimation}
	\forall r \geq 2, \quad \forall s \geq 0, \quad\|R(s)\|_{L_{\rho_1}^{r}} \leq C(r) e^{-2 s} 
\end{equation}
which is consistent with where we constructed $\varphi$ as an approximate solution for equation (1.12).
\paragraph{\textbullet\ The cut-off term $N$, $H$ and $G$: \\}
It is small in $L^\infty$ in the sense that:
\begin{equation}
\|N(\cdot,s)\|_{L^{\infty}}+\|\del_yG(\cdot,s)\|_{L^{\infty}}+\|H(\cdot,s)\|_{L^{\infty}}\leq C e^{-3s}
\end{equation} 
With these properties, although the control of $B(y,s)$ is delicate since the power function fail to be continuous in $L^{r}_{\rho_1}$, it still seems resonable to construct a small solution $q$ to \eqref{q's equation}. 
\subsection{Control of potential and the nonlinear terms in equation \eqref{q's equation}} \label{control of potential, non linear term}
In this section, we describe the strategy to control the potential and nonlinear terms \( V q \) and \( B \) appearing in equation \eqref{q's equation}. For clarity of exposition, we focus explicitly on the nonlinear term \( B \).  it means that we temporarily consider the following simplified equation:
\begin{equation}\label{eq: simplified q equation}
\partial_{s} q=(\mathcal{L}+V) q+B(y,s)+R(y, s).
\end{equation}
We also restrict  the discussion to the case \( p \geq 2 \), and thus setting \( \bar{p} = 2 \). However, the rigorous analysis of both terms \( V q \) and \( B \) will ultimately be carried out for all \( p > 1 \) and for equation \eqref{q's equation}. Under the assumption of the apriori estimate \eqref{eq: apriori estimation on q}, equation \eqref{eq: simplified q equation} reduces to the following linear equation with an associated source term:
\begin{equation}\label{eq: q reduce equation}
	\partial_{s} q=(\mathcal{L}+\bar{V}) q+R(y, s) 
\end{equation}
where
\begin{equation}\label{eq; estimation bar V}
	\begin{aligned}
	|\bar{V}(y, s)|&=\left|V(y, s)+\frac{B(y, s, q)}{q}\right| \leq C+C(M)|q|^{\bar{p}-1} \\
	& \leq C+C(M) M^{\bar{p}-1}\equiv \bar{C}(M). 
	\end{aligned}
\end{equation}

Owing to the regularizing properties of the operator \(\mathcal{L}\) (as detailed in Lemma \ref{lemma: Reg eff d oper} below), we can combine equation \eqref{eq: R rough estimation} with equations \eqref{eq: q reduce equation} and \eqref{eq; estimation bar V} to bound the \(L_{\rho}^{4}\)-norm of the solution in terms of its \(L_{\rho}^{2}\)-norm, possibly introducing a certain time delay:

$$
\|q(s)\|_{L_{\rho_1}^{4}} \leq e^{(1+\bar{C}) s^{*}}\left\|q\left(s-s^{*}\right)\right\|_{L_{\rho_1}^{2}}+C e^{-2 s} . 
$$

Clearly, this is equivalent to write that
$$
\left\|q(s)^{2}\right\|_{L_{\rho_1}^{2}} \leq 2 e^{2(1+\bar{C}) s^{*}}\left\|q\left(s-s^{*}\right)\right\|_{L_{\rho_1}^{2}}^{2}+2 C e^{-4 s}.
$$
This estimation enables us, via estimate \eqref{eq: B super linearity}, to obtain control of the nonlinear term \( B(y, s, q) \) in the \( L_{\rho_1}^{2} \)-norm. It is important to note that this control relies fundamentally on the a priori bound \eqref{eq: apriori estimation on q} the validity of which remains to be verified. We shall address this verification in the subsequent section.

\subsection{Definition of a shrinking set making $q(s)\to0$}\label{DOSS}

First, we introduce the following projectors

%\textbullet\  For $m\in\{0,1,2\}$
\begin{equation}\label{eq: Def of projections P_m}
        P_m (f)=f_m=\displaystyle \langle f,k_m\rangle_{L_{\rho_1}^2}\mbox{ for $m\in\{0,1,\cdots, 6\}$,}
\end{equation}

\begin{equation}\label{eq: Def of projection P_-}
    P_-(f)=f_-=\sum_{m\geq 7}P_m (f)h_m(y).
\end{equation}
For the sake of controlling $q$ in the region $|y|<e^{s/4}$, we consider the following decomposition:
\begin{equation}\label{decomposition of q}
q(y,s)=\sum_{m=0}^{6} q_m(s) \frac{h_m}{\|h_m\|_{L_{\rho_1}^2}}(y)+q_{-}(y,s).
\end{equation}
To achieve our target \eqref{eq: q target}, the core of our approach lies in constructing a solution to equation \eqref{q's equation} that exists on the semi-infinite interval $\left[s_{0}, \infty\right)$, where $s_{0} = -\log T$. This solution must satisfy the containment condition
$$
\forall s \geq s_{0},\quad q(s) \in \mathcal{V}(s),
$$
where the time-dependent set $\mathcal{V}(s)$ exhibits asymptotic collapse to zero in an appropriate sense. Furthermore, we simultaneously require that the solution $u$ to equation \eqref{eq: NLH introduction} maintains controlled smallness within the \textit{regular region}. Guided by this strategy, we introduce the following carefully designed shrinking set $\mathcal{S}(t)$.
\begin{definition}[Definition of the shrinking set $\mathcal{S}(t)$]
	\label{Def: shrinking set}
	For $A$, $ K_0$, $\varepsilon_0>0$, $0<\eta_0\leq 1$, $T>0$, we define for all $t \in [0, T)$   $$\mathcal{S}(t)=\mathcal{S}[A,K_0, \varepsilon_0,\eta_0](t),$$
	the set of all functions $J \in L^\infty(\mathbb{R}^d)$,with $j(|x|,t)=J((|x|,0,\cdots,0),t)$ and $v=(T-t)^{\frac{1}{p-1}}j(y,s)\chi(\frac{ye^{-s/2}+1}{\vep_0})-\varphi(y,s)$ satisfying:
	\begin{itemize}
		\item 	(i) Estimates in $\mathcal{R}_1$: we consider $ \mathcal{V}(s)= \mathcal{V}[K_0,A](s)$ (where $s=-\log (T-t)$ ), the set of all functions $r \in L^\infty(\mathbb{R})$ such that
		\[
		\begin{aligned}
		\left\|(T-t)^{\frac{1}{p-1}}(v+\varphi)\right\|_{L^{\infty}}&\leq 2\kappa,\\
			\lvert v_{k}(s) \rvert &\leq A e^{-2s} \quad (k\;\leq\;5), \\
			\lvert v_{6}(s) \rvert &\leq  As^2e^{-2s} , \\
			\| v_-(y, s) \|_{L_{\rho_1}^2} &\leq A^2 s^2e^{-3s}, \\
		\end{aligned}
		\]
		where
		\[
		\begin{aligned}
			r_-(s) &= P_-( r),
		\end{aligned}
		\]
		where $v_{k}(y, s) $ and $P_-$ are defined in \eqref{eq: Def of projections P_m} and \eqref{eq: Def of projection P_-}.
		
		\item 	(ii) Estimates in $\mathcal{R}_2$: For all $0 \leq \lvert x \rvert \leq \frac{\vep_0}{4}$, $\lvert j(x, t) \rvert \leq \eta_0$.
	\end{itemize}
\end{definition}
\noindent Clearly, we see that: if $U \in S(t)$, then we have
\begin{equation}\label{q L rho priori estimate}
	\|q\|_{L_{\rho_1}^2}\leq CAse^{-2s}
	\end{equation} 
 in the inner region $\mathcal{R}_1$. This implies indeed that $q(s)\to 0$ as $s\to\infty$ in $L_{\rho}^2(\Rbb)$. The next part of the paper is devoted to the rigorous proof of our goal in \eqref{eq: q target}, and to the fact that it implies our main result stated in Theorem \ref{th: Main theorem}.

\section{Dynamics of the equation and general form of initial data}\label{Dynamics of the equation}

\subsection{Dynamics of the equation \eqref{q's equation} in $\mathcal{V}$}
In order to construct a solution $q$ to the equation \eqref{q's equation} verifies 
\begin{equation}
q(y,s) \in \mathcal{V}(s),	
	\end{equation}
we firstly assume $q(s) \in \mathcal{V}(s)$ for all $s \in[s_0,s_1]$ for some $s_1 \geq s_0$. Then, derive the differential equations satisfied by $q_k$ and $q_-$ in this subsection. This idea leads us to the following proposition  
\begin{proposition}\label{Prop: Priori estimate}
For any $A\geq 1$, there exists $s_3(A)\geq 0$ such that for all $s_0\geq s_3 $ assuming:
\begin{equation}\label{eq: priori estimate q}
	q(s_0)\in L^{\infty},\;\nabla q(s_0)\in L^{\infty},\; \forall r\geq2,\;\|q(s_0)\|_{L_{\rho}^r}\leq C(r)As_0^2e^{-2s_0}
	\end{equation}
	and $q(s) \in \mathcal{V}(s)$ a solution to \eqref{q's equation} for all $[s_0,s_1]$, for some $s_1\geq s_0$. Then for all $s\in[s_0,s_1]$:
\begin{enumerate}
	\item for all $i  \leq 6$, $\left|q_i'(s)-(1-\frac{i}{2})q_i\right|\leq C_iAse^{-3s}+R_{i}$;
	\item $\frac{d}{ds}\|q_-\|_{L_{\rho_1}^2}\leq -3\|q_-\|_{L_{\rho_1}^2}+CAse^{-3s}+\|R_-\|_{L_{\rho_1}^2}$.
\end{enumerate}
\end{proposition}
\begin{proof}
	The proof is straightforward, except for the control of the potential term and the superlinear term. They both need the following delay regularizing estimate stated informally in section \ref{control of potential, non linear term}. For the clarity of the main idea of the present paper, we  postpone the proof in Section \ref{study of the projections}.        
	\end{proof}
The size of the components of the remainder term $R(y,s)$ is derived from the following lemma(this will be proved later in Section \ref{Proof-a-priori-estimates}):
\begin{lemma}[Estimates for term $R$]
\noindent For $i \leq 6 $,
\begin{equation}
	|P_{i}(R)|\leq Ce^{-2s};
\end{equation}
\noindent and we have also:
\begin{equation}
\|P_-(R)\|_{L^2_{\rho_1}}\leq Ce^{-3s}.
\end{equation}
\end{lemma}
As we have just mentioned, the following time dealy estimate is crucial to the proof of Proposition \ref{Prop: Priori estimate}.
\begin{proposition}[A delay regularizing estimate for equation \eqref{q's equation}]\label{delay regularizing estimate} Under the hypothesis of Proposition \ref{Prop: Priori estimate} and for any $r \geq 2$ and $s \in\left[s_0, s_1\right]$, it holds that

$$
\|q(s)\|_{L_{\rho_1}^r} \leq C(r) A s e^{-2 s}
$$
\end{proposition}
\begin{proof}
See below in Section 7.3.
\end{proof}

\subsection{Reduction to a finite-dimensional problem}

In this part, we show that the control of the infinite problem is reduced to a finite-dimensional one.  Since the definition of  the bootstrap $\mathcal S(s)$ shows two different types of
estimates, in the regions $\mathcal R_1$ and $\mathcal{R}_2$, accordingly, we need two different approaches to
handle those estimates:
\begin{itemize}
\item In  $\mathcal{R}_1$, we work in similarity variables \eqref{eq: self-similar-variables}, in particular we crucially use the projection of equation \eqref{q's equation} with respect to the decomposition given in \eqref{decomposition of q}.
\item  In,  $\mathcal{R}_2$, we directly work in the variables $U(x, t)$, using standard parabolic estimates. For more details see subsection \ref{section_regular_region}.
\end{itemize}

%It follows from the \eqref{eq:estimation on u in regularregion} that the solution will stay in regular region if the initial data is in the regular region, hence in this 

In the following, we restrict ourselves to the blow-up region. Recall that our aim is to suitably choose $q\left(s_0\right)$ so that $q(s) \in \mathcal{V}(s)$ for all $s \geq s_0$. As in the previous section, we assume that $q(s) \in \mathcal{V}(s)$ for all $s \in\left[s_0, s_1\right]$ for some $s_1 \geq s_0$. This time, we will assume in addition that at $s=s_1$, one component of $q\left(s_1\right)$ (as defined in \eqref{decomposition of q}) "touches" the corresponding part of the boundary of its bound defined in Definition \eqref{Def: shrinking set}. We will then derive the position of the flow of that component with respect to the boundary, and see whether it is inward (which leads to a contradiction) or outward. This way, only the components with an outward flow may touch. Hopefully, those components will be in a finite dimensional space, leading the way to the application of a consequence of Brouwer's lemma linked to some fine-tuning in initial data, in order to guarantee that the solution will stay in $\mathcal{V}(s)$ for all $s \geq s_0$. More precisely, this is our statement:

\begin{proposition}[Position of the flow on the boundary of $\mathcal{V}$]\label{Prop:control of q1,q2,q-,qe}
	There exists $A_1 \geq 1$ and $s_4\geq 0$ such that for all $A>A_1$; $s\geq s_0\geq s_4(A)$, if $q(s)\in \mathcal{V}(s)$ is true, then  for any $s'\in [s_0,s]$ the following holds: 
\begin{enumerate}[label=(\roman*)]
		\item If 
	\[
	q_{i}(s_1) \;=\; \theta \,A\, e^{-2s_1}
	\quad\text{with}\quad i < 6\;\text{and}\;\theta = \pm 1,
	\]
	then
	\[
	\theta\, q_{i}'(s_1) 
	\;>\;
	\left.
	\frac{d}{ds} \Bigl( A\,e^{-2s} \Bigr)
	\right|_{s = s_1}.
	\]
	\item If
	\[
	q_{6}(s_1) \;=\; \theta \,A\,s_1^2\,e^{-2s_1}
	\quad \text{with } \theta = \pm 1,
	\]
	then
	\[
	\theta\, q_{6}'(s_1)
	\;<\;
	\left.
	\frac{d}{ds}\Bigl(A\,s_1^2e^{-2 s_1}\Bigr)
	\right|_{s = s_1}.
	\]
		\item If
		\[ \left\|q_{-}\left(s_1\right)\right\|_{L_{\rho_1}^2}=A^2 s_1^2 e^{-3 s_1},\]  then \[ \frac{d}{d s}\left\|q_{-}\left(s_1\right)\right\|_{L_{\rho_1}^2}<\frac{d}{d s} A^2 s^2 e^{-3 s}\mid_{ s=s_1}\]
	\end{enumerate}
\end{proposition}
\begin{proof}
	See Section \ref{Proof of control q1 to qe} below
	\end{proof}  

\medskip

Consequently, we have the following result
\begin{corollary} \label{control of q by finite emts}($q_6$ and $q_{-}$never quit). Following Proposition \ref{Prop:control of q1,q2,q-,qe} and assuming that $s_1>s_0$, it holds that $\left|q_{6}\left(s_1\right)\right|<A^2 s_1^2 e^{-2 s_1}$ and $\left\|q_{-}\left(s_1\right)\right\|_{L_{\rho_1}^2}<A^2 s_1^2 e^{-3 s_1}$.\end{corollary}
\begin{proof}
We will only prove the estimate on $q_{-}$, since the other follows in the same way. Proceeding by contradiction, we assume that $\left\|q_{-}\left(s_1\right)\right\|_{L_{\rho_1}^2} \geq A^2 s_1^2 e^{-3 s_1}$. Since $q(s) \in$ $\mathcal{V}(s)$ for all $s \in\left[s_0, s_1\right]$ by hypothesis, it follows that

$$
\forall s \in\left(s_0, s_1\right], \quad\left\|q_{-}(s)\right\|_{L_{\rho_1}^2} \leq A^2 s^2 e^{-3 s} \text { and }\left\|q_{-}\left(s_1\right)\right\|_{L_{\rho_1}^2}=A^2 s_1^2 e^{-3 s_1}
$$

In particular, this translates into the following estimate between the derivatives of both curves:

$$
\frac{d}{d s}\left\|q_{-}\left(s_1\right)\right\|_{L_{\rho_1}^2} \geq \left.\frac{d}{d s} A^2 s^2 e^{-3 s}\right|_{s=s_1}
$$

which is a contradiction by item (iii) of Proposition \ref{Prop:control of q1,q2,q-,qe}
\end{proof}

\subsection{Choice of the initial data}
In this part, we aim to give a suitable family of initial data for our problem. For $T>0$, consider $A\geq 1$ and $0<\vep_0\leq 1$ to be chosen respectively large and small enough later. Given $s_0 \geq 0$ we consider initial data for the equation \eqref{eq: NLH introduction} depending on parameters $(d_0, d_1, d_2, d_3, d_4, d_5)\in \mathbb{R}^6$,  
defined for all $y\in \Rbb$ by:

\begin{equation}\label{eq: initial data}
	\begin{aligned}
	w_{1}(y,s_0,d_0,d_1,d_2,d_3,d_4,d_5)&=  \chi\left(\frac{1+ye^{-s_0/2}}{\vep_0}\right)\left\{ \frac{E}{D}+\frac{p-1}{\kappa D^2}Ae^{-2s_0}S(y)\right\}^{\frac{1}{p-1}},\\
	\end{aligned}
\end{equation}
where $E$ and $D$ are defined in \eqref{eq: profile form}, $\chi$ is defined in \eqref{eq: def Chi} and 
\begin{equation}\label{def S y}
	S(y)=d_0+d_1h_1(y)+d_2h_2(y)+d_3h_3(y)+d_4h_4(y)+d_5h_5(y)
\end{equation}

In the following, we exhibit a set for the parameters so that $q(s_0)$ is well-defined, $q(s_0) \in \mathcal{V}(s_0)$ with other smallness and decay properties, inherited from the profile $\varphi(y,s_0)$
given in \eqref{eq: profile form}. More precisely, this is our statement:
\begin{proposition} \label{prop: initialization} (Initialization). For any $A \geq 1$, and $\vep_0\in (0,1)$, there exists $s_{5}(A) \geq 1$ such that for all $s_0 \geq s_{5}(A)$, there exists a set $\mathcal{D}\left(A, s_0\right) \subset[-2,2]^6$ such that for all parameter $d \equiv\left(d_{0},d_{1},d_{2},d_3,d_{4},d_5\right) \in \mathcal{D}$, we have the following 2 properties:
	\begin{enumerate}
\item  $E+\frac{p-1}{\kappa D} A e^{-2 s_0} S(y)\geq \frac{1}{4}$, hence $u_0$ in \eqref{eq: initial data} is well-defined. 
\item $q\left(s_0\right) \in \mathcal{V}(s)$ introduced in Definition \ref{Def: shrinking set}, estimate \eqref{eq: priori estimate q}  holds true,\\$\left\|w_{1}\left(s_0\right)\right\|_{L^{\infty}} \leq
	\kappa+C e^{-\frac{s_0}{3}} $ and $\left\|\partial_y w_{1}\left(s_0\right)\right\|_{I_{\infty}}\leq C e^{-\frac{s_0}{6}}.$
	Moreover, the function 
\begin{equation}
\begin{aligned}\label{d one to one function}
	\mathcal{D} & \longrightarrow\left[-A e^{-2 s_0}, A e^{-2 s_0}\right]^6, \\
	d & \longmapsto\left(q_{0}\left(s_0\right), q_1(s_0),q_2(s_0), q_{3}\left(s_0\right), q_{4}\left(s_0\right),q_5(s_0)\right)
\end{aligned}
\end{equation}
is one-to-one.

\end{enumerate}
\end{proposition}
\begin{proof}
See the proof in Section \ref{Details of initialization}
\end{proof}
\subsection{Control of the solution in the bootstrap regime and proof of Theorem \ref{th: Main theorem}}
We prove Theorem \ref{th: Main theorem} using the previous results. We proceed in four parts:\\
\begin{itemize}
    \item In the first part, we show by a topological argument the existence of $q(y,s)$ trapped in $\mathcal{V}$ defined in Definition \ref{Def: shrinking set};
    \item In the second part, we will show that all the blow-up point is collected by the sphere of radius $r_0$;
    \item In the third part, we will prove the assymptotic behavior described in item 2 of Theorem \ref{th: Main theorem};
    \item In the last part, we will derive the profile of our blow-up solution.
\end{itemize}
\paragraph{Part 1: Existence of $q(y,s)$ trapped in $\mathcal{V}$}\;\\
Let $A$, and $T(=e^{-s_0})$ be chosen so that Proposition \ref{control of q by finite emts}  and Proposition \ref{Prop:control of q1,q2,q-,qe} are valid, We will find the parameters $(d_0,d_1,d_2,d_3,d_4,d_5)\in \mathcal{D}$ defined in Proposition \ref{prop: initialization} and advance by assuming that for all $(d_0,d_1,d_2,d_3,d_4,d_5)\in \mathcal{D}$, there exists $s_*(d_0,d_1,d_2,d_3,d_4,d_5) \geq -\log T$ such that $q_{d_0,d_1,d_2,d_3,d_4,d_5}(s) \in\mathcal V(s)$ for all $s \in [-\log T, s_*]$ and $q_{d_0,d_1,d_2,d_3,d_4,d_5}(s_*) \in \partial\mathcal V(s_*)$. From (i) of Proposition \ref{Prop:control of q1,q2,q-,qe}, we see that $$(\tilde{q}_0(s_*), \tilde{q}_1(s_*),\tilde{q}_2(s_*),\tilde{q}_3(s_*),\tilde{q}_4(s_*),\tilde{q}_5(s_*)) \in \partial[-Ae^{-2s_*},Ae^{-2s_*}]^6$$ and the following function is well-defined:

\begin{equation}\label{definition of Phi}
	\begin{aligned}
		\Phi: [-1,1]^6&\rightarrow\partial[-1, 1]^6\\
		(d_0,d_1,d_2,d_3,d_4,d_5)&\rightarrow \frac{e^{-2s_*}}{A}(\tilde{q}_0, \tilde{q}_1,\tilde{q}_2,\tilde{q}_3,\tilde{q}_4,\tilde{q}_5)_{d_0,d_1,d_2,d_3,d_4,d_5}(s_*).
	\end{aligned}
\end{equation}
Again by (i) of Proposition \ref{Prop:control of q1,q2,q-,qe}, we acquire the continuity of this function. If we manage to show that $\Phi$ is of degree 1 on the boundary, then we have a contradiction from the degree theory. We now focus on proving that.

Using the fact that $q(-\log T) = \psi_{d_0,d_1,d_2,d_3,d_4,d_5}$, we see that when\\ $(d_0,d_1,d_2,d_3,d_4,d_5)$ is on the boundary of the quadrilateral $\mathcal{D}$, we have
\[
\begin{aligned}
	(\tilde{q}_0, \tilde{q}_1,\tilde{q}_2,\tilde{q}_3,\tilde{q}_4,\tilde{q}_5)(-\log T) 
	&\in \partial\bigl[-Ae^{2\log T}, Ae^{-2\log T}\bigr]^6,\\[6pt]
	q(-\log T) &\in \mathcal V_A(-\log T)
\end{aligned}
\]
with strict inequalities for the other components.

 Applying the transverse crossing property (i.e. (i) of Proposition \ref{Prop:control of q1,q2,q-,qe}), we see that $q(s)$ leaves $\mathcal V(s)$ at $s = -\log T$, hence $s_*(d_0,d_1,d_2,d_3,d_4,d_5) = -\log T$. Using \eqref{definition of Phi}, we see that the restriction of $\Phi$ to the boundary is of degree 1. A contradiction then follows. Thus, there exists a value $(d_0,d_1,d_2,d_3,d_4,d_5) \in \mathcal{D}$ such that for all $s \geq -\log T$, $q_{d_0,d_1,d_2,d_3,d_4,d_5}(s) \in \mathcal V (s)$.\\ 
\begin{comment}
\textbf{Part 2: Proof of Theorem 1}

Consider the solution constructed in Part 1, such that $q(s)\in \mathcal V(s)$.
Then by Definition \ref{Def: shrinking set}, we see that
\[
\forall y \in \mathbb{R}, \quad \forall s \geq -\log T, \quad |q(y, s)| \leq \frac{CA^2 }{\sqrt{s}}.
\]

By definitions \eqref{self-similar-variables} \eqref{definition of varphi}, we see that
\[
\forall s \geq -\log T, \quad \forall |x| \geq \frac{ \varepsilon_0}{4}, \quad \left|W(y, s) - f\left(\frac{y}{\sqrt{s}}\right) \right|\leq \frac{CA^2}{\sqrt{s}} + \frac{C}{s}.
\]

By definition \eqref{self-similar-variables} of $W$, we see that $\forall t \in [0, T), \quad \forall |x| \geq \frac{\varepsilon_0}{4},$
\[
\left|(T - t)^{1/(p-1)} u(r, t) - f\left(\frac{r-r_{max}}{ \sqrt{(T-t)\log(T - t)}}\right)\right| \leq \frac{C(A)}{\sqrt{|\log(T - t)|}}.
\]

(i) If $r = r_{max}(=1)$, then we see from above that $|u(0, t)| \sim \kappa(T - t)^{-1/(p-1)}$ as $t \rightarrow T$. Hence, $u$ blows up at time $T$ at $r = r_{max}$.

It remains to prove that any $r\neq r_{max}(=1)$ is not a blow-up point.  Since we know from item (ii) in Definition \ref{Def: shrinking set} that if $r \leq \frac{3\vep_0}{4}$, and $0 \leq t \leq T$, $|u(r, t)| \leq \eta_0$, it follows that $r$ is not a blow-up point, provided $r \leq \frac{3\vep_0}{4}$.

Now, if $r \geq \frac{3\vep_0}{4}$, the following result from Giga and Kohn [13] allows us to conclude.
\end{comment}
\paragraph{Part 2: Proof of 1 in Theorem \ref{th: Main theorem}}\;\\
It remains to prove that any $x$, with $|x|=r\neq r_{max}(=1)$ is not a blow-up point.  Since we know from item (ii) in Definition \ref{Def: shrinking set} that if $r \leq \frac{3\vep_0}{4}$, and $0 \leq t \leq T$, $|u(r, t)| \leq \eta_0$, it follows that $r$ is not a blow-up point, provided $r \leq \frac{3\vep_0}{4}$.

Now, if $r \geq \frac{3\vep_0}{4}$, from the similarity variables definition \eqref{eq: second group ssv}, $U$ is indeed a solution of equation \eqref{eq: NLH introduction} defined for all $(x, t) \in \mathbb{R}^d \times[0, T)$. From the classical theory by Giga and Kohn \cite{GKcpam85}, we know that when $s \rightarrow \infty, W_{x_0}(s) \rightarrow \kappa$ when $x_0$ is a blow-up point, and $W_{x_0}(s) \rightarrow 0$ if not, in $L_{\rho_d}^2$. Therefore, it's enough to prove that $W_{\mathbf{e}_1}(s) \rightarrow \kappa$ and $W_{r\mathbf{e}_1}(s) \rightarrow 0$ as $s \rightarrow \infty$ in order to conclude.
First, since $q(s) \in \mathcal{V}(s)$ for all $s \geq s_0$ from Part 1, we see from \eqref{q L rho priori estimate}, the relations \eqref{eq: w decompose-in-terms-q-phi}, \eqref{eq: self-similar-variables},\eqref{eq: second group ssv}, \eqref{eq: transform between dd to 1d}, and the definition \eqref{eq: profile form} of $\varphi$ that $W_{\mathbf{e}_1}(s) \rightarrow \kappa$ as $s \rightarrow \infty$.

Second, if $r \neq 1$, we see that for any $s_1 \geq 2 \log \frac{A}{\left|r-1\right|}$, \eqref{x0ITR3F} in below holds and the following proposition holds here:
\begin{proposition}
There exist $A_{7} \geq 1$ such that for all $A \geq A_{7}$, there exists $s_{17}(A) \geq 1$ such that for all $s_0 \geq s_{17}(A), d \in \mathcal{D}\left(A, s_0\right)$ defined in Proposition \ref{prop: initialization} and $s_1 \geq s_0$, the following holds: Assume that $W_{x_0}$ is the solution of equation \eqref{eq: W's equation} with initial data $w_1\left(y, s_0\right)$ defined in \eqref{eq: initial data}, such that for all $s \in\left[s_0, s_1\right], q(s) \in \mathcal{V}(s)$ given in Definition \eqref{Def: shrinking set}, where $q(s)$ is defined in \eqref{eq: w decompose-in-terms-q-phi}. Then:
\begin{enumerate}[label=(\roman*)]
\item For all $s \in\left[s_0, s_1\right],\left\|w_{1}(s)\right\|_{L^{\infty}}<2 \kappa$.
\item For all $x_0 \in \mathbb{R}^d$ such that

\begin{equation}
||x_0|-1| \geq A e^{-\frac{s_1}{2}}
\end{equation}
there exist $\bar{s}(x_0) \geq s_0$ and $\bar{M}(x_0) \geq 0$ such that if $\bar{s}(x_0) \geq s_1$, then for all $s \in\left[\bar{s}(x_0), s_1\right]$, $\left\|W_{x_0}(s)\right\|_{L_{\rho_d}^2} \leq \bar{M}(x_0) e^{-\frac{s}{p-1}}$.
\end{enumerate}
\end{proposition}
\begin{proof}
See Proposition \ref{Control of W_{x_0} for all x_0} below.
\end{proof}
Applying its item (ii), we see that for some $\bar{s}(x_0) \geq s_0$ and $\bar{M}(x_0) \geq 0$, for any $s_1 \geq \max \left(2 \log \frac{A}{\left|r-1\right|}, \bar{s}(x_0)\right)$, for any $s \in\left[\bar{s}(x_0), s_1\right],
\left\|W_{x_0}(s)\right\|_{L_\rho^2} \leq \bar{M}(x_0) e^{-\frac{s}{p-1}}$. Making $s_1 \rightarrow \infty$, we see that $W_{x_0}(s) \rightarrow 0$. This concludes the proof of 1 in Theorem \ref{th: Main theorem}.

\paragraph{Part 3: Proof of 2 in Theorem \ref{th: Main theorem}}\;\\
Consider the solution constructed in Part 1, such that $q(s)\in \mathcal V(s)$.
Then by Definition \ref{Def: shrinking set}, we see that
\[
 \quad \forall s \geq -\log T, \quad \|q(y, s)\|_{L^2_{\rho_1}} \leq CAse^{-2s}.
\]
Choose $\vep_{0}$ sufficiently small, we have that all
$x\in \Lambda_C = \left\{||x|-r_0|\leq C(T-t)^{\frac{1}{4}}\right\}$, verifies $|x| \geq \frac{ \varepsilon_0}{4}$.
Then, upon \eqref{eq: w decompose-in-terms-q-phi} and \eqref{eq: self-similar-variables}, we see that
\begin{equation}\label{DBTWVP}
\forall t \in [0,T),\; x\in \Lambda_C, \quad \left| \tilde{w} - \varphi(y,s) \right|\leq CAse^{-2s}.
\end{equation}
Then apply control in regular region from Section \ref{section_regular_region} in this paper, we know that $u(|x|,t)\in S(t)$. Now, the Definition \ref{Def: shrinking set} of $S(t)$ and \eqref{eq: self-similar-variables} indicates that 

\begin{equation}\label{DBuTw}
\forall t \in [0,T),\; x\in \Lambda_C, \quad \left| (T-t)^{\frac{1}{p-1}}u - \tilde{w} \right|\leq \eta_0 e^{-\frac{s}{p-1}}.
\end{equation}
Note that by \eqref{eq: profile form}, \eqref{profile expansion} and \eqref{eq:TVP}, we have
\begin{equation}\label{DBfVP}
\forall t \in [0,T),\; x\in \Lambda_C, \quad \left|f\left(\frac{|x|-r_0}{(T-t)^{1/4}}\right)-\varphi\right|\leq Ce^{-s}.
\end{equation}
Now, the relation \eqref{eq: transform between dd to 1d}, together with \eqref{DBTWVP}, \eqref{DBuTw} and \eqref{DBfV} concludes the proof of item 2 of Theorem \ref{th: Main theorem}.

\paragraph{Part 4: Proof of 3 in Theorem \ref{th: Main theorem}}
This result is in fact a consequence of the following ODE localization property of the PDE,  and which is a direct consequence of the Liouville Theorem stated in Proposition \ref{Proposition:liouville thm}:

\begin{proposition}[ODE localization for $u(x, t)$ ].\label{ODE localization} For any $\epsilon>0$, there exists $C_\epsilon>0$ such that for all $(x, t) \in \mathbb{R}^d \times[0, T)$,

$$
(1-\epsilon) u(x, t)^p-C_\epsilon \leq \partial_t u(x, t) \leq(1+\epsilon) u(x, t)^p+C_\epsilon
$$
\end{proposition}
\begin{proof}
	See Proposition 6.6 in \cite{MZJEMS24}
	\end{proof}
Now, The existence of the limiting profile $u(r-r_0, T)$ follows by compactness, exactly as in Proposition 2.2 page 269 in \cite{FKcpam92}. It remains to prove that
\[
U(x,T) \sim u^*(|x|-r_0) \text{ as } |x|\to r_0.
\]
	
	Consider any $K_0>0$. Given any $x\in \Rbb^d$ with $ |x| \neq r_0$, we introduce the time $t^*(|x|) \in[0, T)$ such that
	
	\begin{equation}\label{t |x| definition}
      (|x|-r_0)^4=K_0 \frac{\kappa}{p-1}\left(T-t^*(|x|)\right)
	\end{equation}
Note that $t^*(r) \rightarrow T$ as $r \rightarrow r_0$. The conclusion will follow from the study of $u(r, t)$ on the time interval $\left[t^*(r), T\right)$. Consider first some arbitrary $\varepsilon>0$.
\paragraph{Step 1: Initialization.}
From 2 of Theorem \ref{th: Main theorem}, we see that
\begin{equation}\label{u r t etoile estimation}
\left|\left(T-t^*(|x|)\right)^{\frac{1}{p-1}} U\left(|x|, t^*(|x|)\right)-\kappa\left(1+K_0\right)^{-\frac{1}{p-1}}\right| \leq \epsilon \kappa\left(1+K_0\right)^{-\frac{1}{p-1}}
\end{equation}
provided that $||x|-r_0|$ is small enough.

\paragraph{Step 2: Dynamics for $t \in [ t^*(|x|), T)$.} From Proposition \ref{ODE localization}, we see that

\begin{equation}\label{ODE localization app on urt}
\forall t \in\left[t^*(r), T\right), \quad(1-\epsilon) U(x, t)^p-C_\epsilon \leq \partial_t U(x, t) \leq(1+\epsilon) U(x, t)^p+C_\epsilon
\end{equation}
for some $C_\epsilon>0$. Using \eqref{u r t etoile estimation} and \eqref{ODE localization app on urt}, one easily shows that for $|r-r_0|$ small enough,

$$
\forall t \in\left[t^*(r), T\right), \quad C_\epsilon \leq \epsilon U(x, t)^p .
$$
Therefore, we see from \eqref{ODE localization app on urt} that

$$
\forall t \in\left[t^*(r), T\right), \quad(1-2 \epsilon) U(x, t)^p \leq \partial_t U(x, t) \leq(1+2 \epsilon) U(x, t)^p .
$$
Using again \eqref{u r t etoile estimation}, we may explicitly integrate the 2 differential inequalities we have just derived and write for all $t \in\left[t^*(|x|), T\right)$,

$$
\begin{aligned}
	\kappa & {\left[\left(T-t^*(|x|)\right)\left(1+K_0\right)(1-\epsilon)^{1-p}-(1-2 \epsilon)\left(t-t^*(|x|)\right)\right]^{-\frac{1}{p-1}} \leq U(x, t) } \\
	& \leq \kappa\left[\left(T-t^*(|x|)\right)\left(1+K_0\right)(1+\epsilon)^{1-p}-(1+2 \epsilon)\left(t-t^*(|x|)\right)\right]^{-\frac{1}{p-1}}
\end{aligned}
$$
Making $t \rightarrow T$, we see that

$$
\begin{aligned}
	\kappa(T- & \left.t^*(|x|)\right)^{-\frac{1}{p-1}}\left[\left(1+K_0\right)(1-\epsilon)^{1-p}-(1-2 \epsilon)\right]^{-\frac{1}{p-1}} \leq U(x, T) \\
	& \leq \kappa\left(T-t^*(|x|)\right)^{-\frac{1}{p-1}}\left[\left(1+K_0\right)(1+\epsilon)^{1-p}-(1+2 \epsilon)\right]^{-\frac{1}{p-1}}
\end{aligned}
$$
Taking $\epsilon$ small enough, we see that

$$
\left|U(x, T)-\kappa\left(T-t^*(|x|)\right)^{-\frac{1}{p-1}} K_0^{-\frac{1}{p-1}}\right| \leq C\left(K_0\right) \epsilon\left(T-t^*(|x|)\right)^{-\frac{1}{p-1}}
$$
Using the definitions \eqref{t |x| definition} and \eqref{eq: profile} of $t^*(|x|)$ and $u^*(r)$, we see that

$$
\left|U(x, T)-u^*(|x|-r_0)\right| \leq C^{\prime}\left(K_0\right) \epsilon u^*(|x|-r_0)
$$
Since $\epsilon>0$ was arbitrarily chosen, this concludes the proof of 3 of Theorem \ref{th: Main theorem}.
\begin{remark}
	In (i) of Proposition \ref{Prop:control of q1,q2,q-,qe}, we show that the solution $q(s)$ crosses the boundary $\partial \mathcal{V}(s)$ at $s_1$, with positive speed, in other words, that all points on $\partial \mathcal{V}(s_1)$ are strict exit points. The construction is essentially an adaptation of Wazewski's principle (see \cite{Conbook78}, chapter II and the references given there).
\end{remark}
\section{Proof of the priori $L^\infty$ bound in similarity variables}\label{Proof of the priori bound}
Beginning with this section, our objective is to demonstrate that the $L^{\infty}$ bound imposed on $w_{1}$ as defined in Definition \ref{Def: shrinking set} for the set $\mathcal{S}(t)$ is never attained with equality. Specifically, considering $w_1\left(s_0\right)$ as given by \eqref{eq: initial data} for parameters $(d_0, d_1, d_2, d_3, d_4, d_5) \in \mathcal{D}$ established in Proposition \ref{prop: initialization}, and assuming that $q(y,s)$ solves equation \eqref{q's equation} with $q(s) \in \mathcal{V}$ according to Definition \ref{Def: shrinking set} for every $s \in [s_0, s_1]$, with $s_1 \geq s_0$ and sufficiently large $s_0$, then
\begin{equation}\label{eq: priori estimate goal 2}
		\forall y  \in [-e^{-s/2},+\infty), \quad \forall s \in [s_0, s_1], \quad \|w_{1}(\cdot,s)\|_{L^{\infty}}\leq 2\kappa.
	\end{equation}
Upon equation \eqref{eq: self-similar-variables} and \eqref{eq: second group ssv}, one may see the following correspondence: Let $\mathbf{e}_1$ be the first canonical basis in $\Rbb^d$, then
\begin{equation}\label{w1 to We1}
w_1(y,s) = W_{\mathbf{e}_1}\left(y\mathbf{e}_1,s\right), \;\text{with}\; y \in[-e^{-s/2},+\infty).
\end{equation}
Then, by \eqref{eq: self-similar-variables}, \eqref{eq: second group ssv}, \eqref{w1 to We1} and the radial symmetrical property, the following statement  equivalent to \eqref{eq: priori estimate goal 2}:
\begin{equation}\label{eq: priori estimate goal}
	\forall x_0 = |x_0|\mathbf{e}_1 \in \mathbb{R}^d, \quad \forall s \in [s_0, s_1], \quad \|W_{x_0}( 0,s)\|_{L^{\infty}}\leq 2\kappa.
\end{equation}
Thus, the following equation satisfied by $W_{x_0}$ will be useful:
\begin{equation}\label{eq: W's equation}
	\partial_s W_{x_0}=\Delta W_{x_0}-\frac{1}{2} Y \cdot \nabla W_{x_0}-\frac{W_{x_0}}{p-1}+|W_{x_0}|^{p-1} W_{x_0}.
\end{equation}
We will achieve this through a structured approach comprising several stages to obtain adequate control over $W_{x_0}(Y, s)$ for any $x_0 = |x_0|\mathbf{e}_1 \in \mathbb{R}^d$. 
 Initially, we derive a uniform bound on the gradient, followed by a detailed outline of our strategy, which will differ based on the location of the point $x_0$. The subsequent three subsections provide rigorous proofs tailored to each distinct region. We conclude in the final subsection with a summarizing statement of the results achieved throughout this section.

\subsection{Control of Gradient}
We first introduce the following Liouville theorem from \cite{MZcpam98} and \cite{MZma00}.
\begin{proposition}[Liouville theorem ,Merle-Zaag]\label{Proposition:liouville thm}
Let $p>1$, consider \(W(Y,s)\) to be a solution of equation \eqref{eq: W's equation}
	that is defined and uniformly bounded for all \((Y,s)\in \mathbb{R}^d \times (-\infty, \bar{s})\)
	for some \(\bar{s}\leq +\infty\). Then, one of the following holds:
	\[
W\equiv 0, 
	\quad 
W \equiv \pm \kappa, 
	\quad 
	\text{or} 
	\quad
 W(Y, s) = \pm \kappa \Bigl( 1 \pm e^{\,s - s^*} \Bigr)^{-\frac{1}{p-1}}
	\]
	for all \((Y,s)\in \mathbb{R}^d \times (-\infty,\bar{s}]\) and for some \(s^*\in \mathbb{R}\).
	In all cases, it holds that \(\nabla W \equiv 0\).
	\end{proposition}
\begin{proof}

If $\bar{s}=+\infty$, see [\cite{MZcpam98}, p. 143, Theorem 1.4] for the nonnegative case and [\cite{MZma00}, p. 106, Theorem 1] for the unsigned case.\\
If $\bar{s}<+\infty$, then the statement follows from a small adaptation of the previous case. See [\cite{MZma00}, p. 144, Corollary 1.5], where a similar adaptation is carried out.
\end{proof}
\noindent With this Liouville theorem we are able to derive the following bound on the gradient. 
\begin{proposition}[Smallness of the gradient in similarity variables]\label{gradient estimation}
	For all \(A \geq 1\) and \(\delta_0 > 0\), there exists 
	 \(s_6 \geq 1\), such that for all \(s_0\geq s_6\),  if \(q(s_0)\) is given with \eqref{eq: initial data}
	for some \(\bigl(d_{0}, d_{1}, d_{2},d_3,d_4,d_5\bigr) \in \mathcal{D}\) (as defined 
	in Proposition \ref{prop: initialization}), and if \(q(s) \in \mathcal{V}(s)\) for all \(s \in [s_0, s_1]\) for some 
	\(s_1 \geq s_0\) that satisfies \eqref{q's equation}, then for all \(s \in [s_0, s_1]\), 
	\[
	\|\nabla W \|_{L^\infty} \;\leq\; \delta_0.
	\]
\end{proposition}
\begin{proof}
The proof follows from the Liouville theorem given in Proposition \ref{Proposition:liouville thm}, exactly as in the similar step in \cite{MZJEMS24}. In order to keep this paper into reasonable length, we kindly refer to the readers to the similar step in \cite{MZJEMS24}, namely to Proposition 5.2 Page 22.
\end{proof}
	\subsection{Methodology for bounding \( W_{x_0} \)}
With the gradient estimate in Proposition \ref{gradient estimation}, we shall make a reduction of our goal in the following.
\begin{claim} \label{L infty estimation control by L rho}
	
Under the hypotheses of Proposition \ref{gradient estimation}, assuming that

\begin{equation}\label{setting of delta 0}
\delta_0 \leq \frac{\kappa}{4} \text { and } s_0 \geq s_{6}\left(A, \delta_0\right)
\end{equation}

we see that estimate \eqref{eq: priori estimate goal} follows from the following:

\begin{equation}\label{eq: Improved bound on W}
\forall x_0 \in \mathbb{R}^d, \quad \forall s \in\left[s_0, s_1\right], \quad\left\|W_{x_0}(s)\right\|_{L_{\rho_d}^2} \leq \frac{3}{2} \kappa 
\end{equation}

\end{claim}

\begin{proof} Under the hypotheses of Proposition \ref{gradient estimation}, assume that \eqref{setting of delta 0} holds. Noting that $\left\|\nabla W_{x_0}(s)\right\|_{L^{\infty}} \leq \delta_0$ from Proposition \ref{gradient estimation} together with the relations \eqref{eq: second group ssv}, we use a Taylor expansion to write

$$
\left|W_{x_0}(Y, s)-W_{x_0}(0, s)\right| \leq|Y| \cdot\left\|\nabla W_{x_0}(s)\right\|_{L^{\infty}} \leq \delta_0|Y|
$$
Therefore, since
$$
\int \rho_d(Y) d Y=1 \text { and } \int|Y| \rho_d(Y) d y=1
$$
by definition (1.18) of $\rho$, it follows that
$$
\left|W_{x_0}(0, s)\right| \leq \int\left|W_{x_0}(Y, s)\right| \rho_d(Y) d Y+\delta_0 \leq\left\|W_{x_0}(s)\right\|_{L_{\rho_d}^2}+\delta_0 \leq \frac{3}{2} \kappa+\frac{\kappa}{4}<2 \kappa
$$
and the claim follows.
\end{proof}
In the following subsections, assuming that \eqref{setting of delta 0} holds, we will prove either \eqref{eq: priori estimate goal} or \eqref{eq: Improved bound on W}, according to the context. Moreover, since the initial data \eqref{eq: initial data} is given for $w_1(y,s)$, we need to establish the relationship between $W_{x_0}(Y,s)$ and $w_1(y,s)$. For this purpose we claim the following:
\begin{claim}\label{CRBWaw}
    For all $Y,x_0\in \Rbb^d$, and $s\geq-\log T$, the following holds:
    \[
	 W_{x_0}(Y, s) =w_{1}\left((|Ye^{-s/2}+x_0|-1) e^{s/2}, s\right) 
	\]
\end{claim}
\begin{proof}
This follows directly from \eqref{eq: self-similar-variables} and \eqref{eq: second group ssv}. In fact, by \eqref{eq: second group ssv}, we easily deduce that:
\[
x=(Ye^{-s/2}+x_0).
\]
Then, replace $r$ by $|x|$ in \eqref{eq: self-similar-variables} concludes the proof of this claim.
\end{proof}
Using the sharper gradient estimate at initial time $s_0$ given in item (ii) of Proposition \ref{prop: initialization}, let us remark that at $s=s_0, W_{x_0}\left(Y, s_0\right)$ is "flat" in $L_{\rho_{d}}^2$, in the sense that it is close to some constant independent from space, as we prove in the following:
\begin{lemma}[Flatness of $W_{x_0}(y, s_0)$]\label{flatness wa}
	
	For any $A \geq 1$, $\vep_0 \in (0,1)$,  there exists $s_0$ and parameter $d = (d_{0}, d_{1}, d_{2},d_3,d_4,d_5) \in \mathcal{D}$, for any $x_0\in \Rbb^d$,
	\[
	\left\| W_{x_0}(\cdot, s_0) -w_{1}\left((|x_0|-1) e^{s_0/2}, s_0\right) \right\|_{L^2_{\rho_d}} \leq C e^{-s_0/6}, \tag{5.7}
	\]
	where $\mathcal{D }$ is defined in Propostion \ref{prop: initialization}.
\end{lemma}
\begin{proof}
Pick any $x_0\in \Rbb^{d}$.	With the following relationship established upon  \eqref{eq: transform between 1d to dd}, \eqref{eq: self-similar-variables} and \eqref{eq: second group ssv},
	\begin{equation}\label{eq: relation between W and w}
	W_{x_0}(Y,s)=w_{1}\left((|Ye^{-s/2}+x_0|-1)e^{s/2}, s\right),
	\end{equation}
 use a Taylor expansion to write
	
\begin{equation}
	\left|w_{1}\left((|Ye^{-s_0/2}+x_0|-1)e^{s_0/2}, s_0\right)-w_{1}((|x_0|-1)e^{s_0/2}, s_0)\right| \leq |Y| \cdot\left\|\del_y w_1(s_0)\right\|_{L^{\infty}}.
\end{equation}
Therefore, since
\begin{equation}\label{rho d is a density function}
	\int \rho_d(Y) d Y=1,
\end{equation}
by definition \eqref{eq: def of rho} of $\rho_d$,  the proposition follows with 2 of Propostion \ref{prop: initialization} by taking integration on $Y$.
 \begin{comment}We here use the same argument in \cite{MZJEMS24}. Interested readers are invited to see the proof of Claim 5.3 and particularly estimation (5.5) in \cite{MZIMRN22}. \end{comment}
\end{proof}
Following Lemma \ref{flatness wa}, observe that equation \eqref{eq: W's equation} satisfied by $W_{x_0}(Y, s)$ admits three bounded, nonnegative, and explicit "flat" solutions: $0, \kappa$, and
\begin{equation}\label{eq: heteroclinic orbit}\psi(s)=\kappa\left(1+e^s\right)^{-\frac{1}{p-1}},\end{equation}
where $\psi$ is a heteroclinic orbit connecting $\kappa$ to 0, and any time shift of $\psi$ remains a valid solution. Moreover, by Proposition \ref{prop: initialization}, we have $w_1\bigl((|x_0|-1)e^{\frac{s_0}{2}}, s_0\bigr)$ lying in the interval $\bigl[0, \kappa + C e^{-\frac{s_0}{3}}\bigr]$. In other words, Lemma \ref{flatness wa} places us near one of these three explicit solutions, all of which possess well-characterized stability properties. Indeed, as will be shown, both 0 and $\psi$ are stable, whereas $\kappa$ exhibits stable and unstable directions, as can be seen by linearizing equation \eqref{eq: W's equation} about $\kappa$. The linearized operator corresponds to $\mathcal{L}$ introduced in \eqref{eq: linear operator}, whose spectrum is provided in \eqref{eq: spectrum}. In particular, the presence of the heteroclinic orbit $\psi$ indicates the instability of $\kappa$. Therefore, if $w_1\bigl((|x_0|-1)e^{\frac{s_0}{2}}, s_0\bigr)$ lies close to $\kappa$, we must refine the estimate in Lemma \ref{flatness wa} to better capture the behavior of $W_{x_0}(Y, s)$ for $s \ge s_0$. 

Consequently, we partition the space into three regions, within which $w_1\bigl((|x_0|-1)e^{\frac{s_0}{2}}, s_0\bigr)$ is in the vicinity of one of these three aforementioned solutions. This decomposition naturally leads to three distinct scenarios regarding the evolution of $W_{x_0}(y, s)$ for $s \ge s_0$.

 More precisely, given $m < M$, we
introduce 3 regions $\tilde{\mathcal{R}}_i(m, M, s_0)$ for $i = 1, 2, 3$ as follows:
\begin{align}
\tilde{\mathcal{R}}_1 &= \left\{ x_0 \in \mathbb{R}^d \ \big| \ Me^{-s_0} \leq G_{1}(x_0) \right\}, \label{region R1} \\
\tilde{\mathcal{R}}_2 &= \left\{ x_0 \in \mathbb{R}^d \ \big| \ me^{-s_0} \leq G_{1}(x_0) \leq Me^{-s_0} \right\}, \label{region R2} \\
\tilde{\mathcal{R}}_3 &= \left\{ x_0 \in \mathbb{R}^d \Big|\  G_{1}(x_0) \leq me^{-s_0} \right\}, \label{region R3}
\end{align}
where
\begin{equation}\label{definition of G1}
	G_{1}(x_0) = \frac{p - 1}{\kappa}(|x_0|-1)^4. 
\end{equation}
In the following lemma, we will see that $G_1(x_0)$ is a kind of norm which measures the
size of $w_{1}((|x_0|-1)e^{s_0/2}, s_0) $:
\begin{lemma}[Size of initial data in the three regions]
\label{Lemma G_0(a) as norm}
For any $M \geq 1$, there exists a constant $C_{1}(M) > 0$ such that for any $A \geq 1$ and any $\vep_0\in (0,1)$ there exists $s_{7}(A, M)$, such that for all $s_0 \geq s_{7}(A, M)$ and $d \in \mathcal{D}(A, s_0)$ defined as in Proposition \ref{prop: initialization}, for any $m \in (0,1)$ the following statement holds:
\begin{itemize}
    \item If $x_0 \in \tilde{\mathcal{R}}_1 $:
    \[
    0 \leq w_{1}((|x_0|-1)e^{\frac{s_0}{2}}, s_0) \leq \kappa(1 + M)^{-1/(p-1)} + C_{15}(M)e^{-\frac{s_0}{3}}.
    \]
    \item If $x_0 \in \tilde{\mathcal{R}}_2$:
    \[\footnotesize
    \kappa(1 + M)^{-1/(p-1)} - C_{15}(M) \, e^{-\frac{s_0}{3}} \leq w_{1}((|x_0|-1)e^{\frac{s_0}{2}}, s_0) \leq \kappa(1 + m)^{-1/(p-1)} + C \, e^{-\frac{s_0}{3}}.
    \]
    \item If $x_0 \in \tilde{\mathcal{R}}_3$:
    \[
    \kappa(1 + m)^{-1/(p-1)} - C \, e^{-\frac{s_0}{3}} \leq w_{1}((|x_0|-1)e^{\frac{s_0}{2}}, s_0) \leq \kappa + C \, e^{-\frac{s_0}{3}}.
    \]
\end{itemize}
\end{lemma}
\begin{proof}
See Section \ref{Proof of G as norm}
\end{proof}
\subsection{Control of $W_{x_0}(Y,s)$ for any $x_0$ in $\tilde{\mathcal{R}}_1 $}
The stability of the zero solution for equation \eqref{eq: W's equation} (under some \( L^\infty \) a priori bound) is crucial for the argument, as we wrote above in Scenario 1. Let us state it in the following: 
\begin{proposition}[Merle-Zaag \cite{MZJEMS24}]
\label{prop: control of w in R1}
There exists \( \epsilon_0 > 0 \) and \( M_0 \geq 1 \) such that if  \( W(Y,t) \) solves equation \eqref{eq: W's equation} with \( \|W(Y,s)\|_{L^{\infty}(\Rbb^d)} \leq 2\kappa\) for all \( (Y,t) \in \mathbb{R}^d \times [0, \sigma_1] \) for some \( \sigma_1 \geq 0 \), with \( \|W(0)\|_{L^2_\rho} \leq \epsilon_0 \) and \( \nabla W(0)(1 + |y|)^{-k} \in L^\infty \) for some \( k \in \mathbb{N} \), then
\[
\forall s \in [0, \sigma_1], \quad \|W(s)\|_{L^2_{\rho_d}} \leq M_0 \|W(0)\|_{L^2_{\rho_d}} e^{-\frac{s}{p-1}}.
\]
\end{proposition}

\textbf{Proof.} See Proposition 5.6 in \cite{MZJEMS24}.

Fixing \( M \geq 1 \) such that
\begin{equation}\label{Choice of M}
\kappa(1 + M)^{-\frac{1}{p-1}} \leq \max\left\{ \frac{\epsilon_0}{2}, \frac{\kappa}{2M_0} \right\}
\end{equation}
where \( M_0 \) and \( \epsilon_0 \) are given in Proposition \ref{prop: control of w in R1}, then taking \( s_0 \) large enough, we see from Proposition \ref{prop: initialization} that \( \|\nabla W_{x_0}(\cdot,s_0)\|_{L^\infty(\Rbb^d)}= \|\del_y w_{1}(\cdot, s_0) \|_{L^\infty(\Rbb)}\leq +\infty \), and from Lemmas \ref{flatness wa} and \ref{Lemma G_0(a) as norm}  that the smallness condition required in Proposition~\ref{prop: control of w in R1} holds in Region \( \tilde{\mathcal{R}}_1  \), leading to the trapping of \( W_{x_0} \) near 0, proving the bound \eqref{eq: Improved bound on W}. More precisely, this is our statement:
 
\begin{corollary}[Exponential decay of \( W_{x_0}(s) \) in Region \( \tilde{\mathcal{R}}_1  \)]
\label{exponential decay of W in R1}
For all \( A \geq 1 \), $\epsilon_{0}\in (0,1)$ there exists \( s_{8}(A) \geq 1 \) such that if \( s_0 \geq s_{8}(A) \), \( d \in \mathcal{D}(A, s_0) \), and \( x_0 \in \tilde{\mathcal{R}}_1 \) defined in \eqref{region R1}, then \( \nabla W(Y,s_0) \in L^\infty \) and
\begin{equation}\label{bound on ini con in R1}
\|W_{x_0}(s_0)\|_{L^2_{\rho_d}} \leq 2\kappa(1 + M)^{-\frac{1}{p-1}} \leq \epsilon_0 \quad \text{(introduced in Proposition~\ref{prop: control of w in R1})    }.
\end{equation}
If in addition we have \( |W_{x_0}(Y,s)| \leq 2\kappa \), for all \( (Y,s) \in \mathbb{R}^d \times [s_0, s_1] \), for some \( s_1 \geq s_0 \), then,
\[
\begin{aligned}
&\forall s \in [s_0, s_1], \\
&\|W_{x_0}(s)\|_{L^2_\rho} \leq M_0 \|W_{x_0}(s_0)\|_{L^2_\rho} e^{-\frac{s - s_0}{p-1}} \leq 2\kappa(1 + M)^{-\frac{1}{p-1}} M_0 e^{-\frac{s - s_0}{p-1}} \leq \kappa,
\end{aligned}
\]
where \( M_0 \) is also introduced in Proposition \ref{prop: control of w in R1}. In particular, \eqref{eq: Improved bound on W} holds.
\end{corollary}
\begin{proof}
As we explained above \eqref{bound on ini con in R1} follows from Lemmas \ref{flatness wa} and \ref{Lemma G_0(a) as norm}. And this Corollary is a direct result of \eqref{bound on ini con in R1} and Lemma \ref{prop: control of w in R1}
\end{proof}
\subsection{Control of $W_{x_0}(y,s)$ for any $x_0$ in $\tilde{\mathcal{R}}_2 $}
As we explained above, in Scenario 2, the stability of the solution \( \psi \) defined in \eqref{eq: heteroclinic orbit} is the key ingreident. Let us first state that stability result:

\begin{proposition}[Merle-Zaag \cite{MZJEMS24}]\label{prop: stability R2}
There exists \( M_1 \geq 1 \) such that if \( W\) solves equation (1.12) with \( |W(Y,s)| \leq 2\kappa \) for all \( (Y,s) \in \mathbb{R}^d \times [0, \sigma_1] \) for some \( \sigma_1 \geq 0 \), and
\begin{equation}\label{stability of psi}
\nabla W(0)(1 + |y|)^{-k} \in L^\infty, \quad \|W(0) - \psi(\sigma^*)\|_{L^2_\rho} \leq \frac{|\psi'(\sigma^*)|}{M_1}
\end{equation}
for some \( k \in \mathbb{N} \) and \( \sigma^* \in \mathbb{R} \), where \( \psi \) is defined in (5.8), then
\[
\forall s \in [0, \sigma_1], \quad \|W(s) - \psi(s + \sigma^*)\|_{L^2_\rho} \leq M_1 \frac{\|W(0) - \psi(\sigma^*)\|_{L^2_\rho}}{|\psi'(s + \sigma^*)|} |\psi'(\sigma^*)|.
\]
\end{proposition}

\begin{proof}
See Proposition 5.8 in \cite{MZJEMS24}
\end{proof}

In the following corollary, using Proposition \ref{prop: initialization}, Lemmas \ref{flatness wa} and \ref{Lemma G_0(a) as norm}, we show that \( W_{x_0}(Y, s) \) is trapped near the heteroclinic orbit \( \psi \) whenever \( x_0 \) is in Region \( \tilde{\mathcal{R}}_2\) defined in \ref{region R2}. More precisely, this is our statement:

\begin{corollary}[Trapping of \( W_{x_0}(s) \) near \( \psi \) in Region \( \tilde{\mathcal{R}}_2 \)]\label{trapping of W_{x_0} near psi}
There exists \( \bar{\sigma} \in \mathbb{R} \) such that for all \( m \in (0,1) \), there exists \( \sigma(m) \leq \bar{\sigma} \) such that for all \( A \geq 1 \) and $\vep_0\in (0,1)$, there exists \( s_{9}(A,m) \geq 1 \) such that for all \( s_0 \geq s_{9}(A,m) \), \( d \in \mathcal{D}(A,s_0) \) and \( x_0 \in \tilde{\mathcal{R}}_2 \) defined in \eqref{region R2} (with \( M \) defined in Proposition \ref{Lemma G_0(a) as norm}), \( \nabla w_1(s_0) \in L^\infty \) and
\begin{equation}\label{initial data control in R_2}
\|w_{1}((|x_0|-1)e^{s_0/2},s_0) - \psi(\sigma^*)\|_{L^2_{\rho_{1}}} \leq C e^{-s_0/6} \leq \frac{|\psi'(\sigma^*)|}{M_1}
\end{equation}
for some \( \sigma^* \in [\sigma(m), \bar{\sigma}] \), where \( M_1 \) was introduced in Proposition \ref{prop: stability R2}. If in addition we have \( |W_{x_0}(Y,s)| \leq 2\kappa \), for all \( (Y,s) \in \mathbb{R}^d \times [s_0, s_1] \), for some \( s_1 \geq s_0 \), then,
\begin{equation}\label{W exponential decay in R_2}
\forall s \in [s_0, s_2], \quad \|W_{x_0}(s)\|_{L^2_{\rho_{d}}} \leq \psi(s + \sigma^* - s_0) + \frac{C M_1}{C(m)} e^{-s_0/6} \leq \kappa + \frac{\kappa}{4} \leq \frac{3}{2}\kappa.
\end{equation}
In particular, \eqref{eq: Improved bound on W} holds.
\end{corollary}

\begin{proof}
Notice that \eqref{initial data control in R_2} follows from \eqref{eq: profile form} definition of $D$, \eqref{eq: heteroclinic orbit} definition of $\psi$, Proposition \ref{prop: initialization}, \eqref{region R2} and  \eqref{definition of G1}; \eqref{stability of psi} then follows from Lemma \ref{flatness wa}. We can then conlude this corollary by Proposition \ref{stability of psi}
\end{proof}
\subsection{Control of $W_{x_0}(Y,s)$ in region $\tilde{\mathcal{R}}_3 $}\label{Control of W_{x_0}(Y,s) in region R_3}
Consider some initial time $s_0$ such that
\begin{equation}\label{initial data region}
	s_0\geq -\log \left[ \frac{p-1}{\kappa M}(1-\frac{\vep_0}{8})^{4}\right],
\end{equation}
 and consider $m\in(0,1)$. Assume that $x_0 \in \tilde{\mathcal{R}}_3 $ defined in \eqref{region R3}. Directly, we have:
\[
1 - \left(\frac{\kappa m}{p-1}\right)^{\frac{1}{4}}e^{-\frac{s_0}{4}}\leq|x_0|\leq 1 + \left(\frac{\kappa m}{p-1}\right)^{\frac{1}{4}}e^{-\frac{s_0}{4}}.
\]
We remark the following: 
\[
\frac{\vep_0}{8}\leq|x_0|\leq 2-\frac{\vep_0}{8}.
\]
Since we are in the radially symmetric case, it is enough to consider the control of $W_{x_0}(Y,s)$ in the case where $x_0$ is of the same direction as $\mathbf{e}_1$, i.e. 
\begin{equation}\label{eq: condition on Y}
	x_0=|x_0| \mathbf{e}_1.
\end{equation} 
Let us introduce $K \geq (\frac{\vep_0}{8}-1)e^{s_0/2   }$ such that:
\begin{equation} \label{eq: form of a in R3} |x_0|-1=Ke^{-\frac{s_0}{2}}.
\end{equation}
Taking into account \eqref{eq: transform between 1d to dd}, \eqref{eq: self-similar-variables}, \eqref{eq: second group ssv} and \eqref{eq: condition on Y}, we see that for all $Y \in \Rbb^d$ and $s\geq s_0$,
$$
	W_{x_0}(Y,s)=w_{1}\left((|Ye^{-s/2}+x_0|-1)e^{s/2}, s\right),
$$
where
\begin{equation}\label{form of x0 in R3}
	x_0=\left(1+Ke^{-s_0/2}\right)\mathbf{e}_1.
\end{equation}
 We then employ the explicit formula \eqref{eq: initial data} for $w_1(y,s_0)$ to evaluate the components of $W_{x_0}\left(s_0\right)$ in $L_{\rho_{d}}^2$. Using this as initial data, our next step is to integrate \eqref{eq: W's equation} (satisfied by $W_{x_0}$) to estimate $W_{x_0}(s)$ for times $s \geq s_0$. Notably, the outcome of this integration depends on $||x_0|-1|$, the distance between $x_0$ and the blow-up sphere, which can be gauged using the parameters $K$ and $s_0$. Consequently, we distinguish two scenarios in the subsequent discussion.
\subsubsection{Case where $|K|\geq A$} \label{Case large K}
From the decomposition \eqref{eq: form of a in R3} of $x_0$, this is the case of "large" $||x_0|-1|$, where $||x_0|-1| \geq A e^{-\frac{s_0}{2}}$. Let us first estimate $W_{x_0}\left(Y, s_0\right)$. In order to be consistent with the definition \eqref{eq: profile form} of our profile $\varphi$ and the decomposition in regions we suggest in \eqref{region R3}, we will give the expansion of $W_{x_0}\left(s_0\right)$ in $L_{\rho_d}^2$, uniformly with respect to the small variable
\begin{equation}\label{eq: iota}
	\iota=e^{-s_0}K^4,
\end{equation}
and the large parameter $A$. This is our statement:

\begin{lemma}[Initial value of \( W_{x_0}(y,s) \) for large \( |x_0|\)]\label{Initialvalue expansion for large x_0}
For any \( A \geq 1 \) and $\vep_0 \in (0,1)$ there exists \( s_{10}(A) \geq 1 \) such that for any \( s_0 \geq s_{10}(A) \) and any parameter \( (d_{0},d_{1} ,d_{2}, d_{3}, d_{4},d_{5},d_{6}) \in \mathcal{D} \) defined in Proposition~\ref{prop: initialization}, if \( w_1(y,s_0) \) is given by \eqref{eq: initial data}, then for any \( m \in (0,1) \) and \( x_0 \in R_3 \) defined in \eqref{region R3} with $|x_0|$ represented as in \eqref{eq: form of a in R3} for some $K > (\frac{\vep_0}{8}-1)e^{s_0/2}$ with \(  | K |  \geq A \) the following expansion holds in \( L^r_{\rho_d}(\mathbb{R}^d) \) for any \( r \geq 2 \):
\[
  \begin{aligned}
W_{x_0}(Y,s_0)&=w_{1}(||Y|e^{-s_0/2}-1|e^{s_0/2},s_0)\\
& = \kappa-\iota -e^{-s_0}\Bigl[h_4(Y_1)+6K^2h_2(Y_2)\\&
+4K^3h_1(Y_1)+4Kh_3(Y_1)\Bigr]+O\left(\frac{\iota}{A}\right)+O(\iota^2),
\end{aligned}
\]
where \( \iota \) is defined in \eqref{eq: iota}.
\end{lemma}
Since that from \eqref{region R3} and \eqref{eq: iota}: 
\begin{equation}\label{iota bound}
\iota \leq \frac{m\kappa}{p-1}\leq\frac{\kappa}{p-1}.
\end{equation}
It follows from this lemma that $\|W_{x_0} (Y,s_0)-\kappa\|_{L_{\rho}^2}\leq C\iota+CJ$ where 
\begin{equation}\label{J bound}
	J\equiv e^{-s_0}K^2\leq2\iota^{\frac{1}{3}};
\end{equation}
\begin{equation}\label{eq: definition of s etoile}
	e^{s^*-s_0}\iota=\eta^*.
	\end{equation}
\begin{lemma}\label{lemma decreaing from k to k-eta}
There exists $M_{2}>0$, $A_{2}\geq 1$ and $\eta_{2}>0$ such that for all $A\geq A_{2}$, there exists $s_{11}(A)$ such that for any $s_0\geq s_{11}(A)$ and any parameter $(d_0,d_1,d_2,d_3,d_4,d_5)\in\mathcal{D}(A,s_0)$ defined in Proposition \ref{prop: initialization}, if $w_1(y,s_0)$ is given by \eqref{eq: initial data}, then, for any $\eta^*\in(0,\eta_{2}]$ and $m\in\left(0,\min\left(1,\frac{\eta^*(p - 1)}{k}\right)\right)$, for any $x_0\in \tilde{\mathcal{R}}_3 $ defined in \eqref{region R3} where $x_0$ is given by \eqref{eq: form of a in R3} for some $ K\geq A\geq 0$, the following holds
 
\begin{enumerate}
    \item[(i)] $s^*\geq s_0$, where $s^*$ is introduced in \eqref{eq: definition of s etoile}.
    \item[(ii)] If we assume in addition that $\|W_{x_0}(s)\|_{L^\infty}\leq 2k$ for all $s\in[s_0,s_1]$ for some $s_1\geq s_0$, then, for all $s\in[s_0,\min(s^*,s_1)]$,

    \[
    \|W_{x_0}(\cdot,s)-(\kappa - e^{s - s_0}\iota)\|_{L^2_\rho}\leq M_{2}(\eta^* + A^{-1})e^{s - s_0}\iota + M_{2}e^{-\frac{s_0}{3}}.
    \]
\end{enumerate}
 \end{lemma}
With this result, we are able to prove estimate \eqref{eq: Improved bound on W}:
 
\begin{corollary}(Proof of the $L^2_{\rho_d}$ bound when $x_0\in \tilde{\mathcal{R}}_3 $ is ``far'' under some $L^\infty$ a priori bound).\label{W_{x_0} control when x_0 far 1}
There exists $A_{3}>0$, $\eta_{3}>0$ and $s_{12}(\eta^*) \geq 1$, such that under the hypotheses of Lemma \ref{lemma decreaing from k to k-eta}, if in addition $A\geq A_{3}$, $\eta^*\leq\eta_{3}$ and $s_0\geq s_{12}(\eta^*)$, then, for all $s\in[s_0,s_1]$, $\|W_{x_0}(s)\|_{L^2_p}\leq\frac{3}{2}k$.
\end{corollary}
\begin{proof}
Following Lemma \ref{lemma decreaing from k to k-eta}, we further assume that $s_0 \geq s_{6}(A, 1)$ defined in Proposition \ref{gradient estimation}, so we can apply that proposition. Consider then $s \in\left[s_0, s_1\right]$. We distinguish 2 cases:
\paragraph{Case 1:} $s \leq s^*$. We write from Lemma \ref{lemma decreaing from k to k-eta}, together with the definitions \eqref{eq: iota} and \eqref{eq: definition of s etoile} of $\iota$ and $s^*$

$$
\begin{aligned}
	\left\|W_{x_0}(s)\right\|_{L_{\rho_d}^2} & \leq \kappa+e^{s^*-s_0} \iota+M_{2}\left(\eta^*+A^{-1}\right) e^{s^*-s_0} \iota+M_{2} e^{-\frac{s_0}{3}}, \\
	& \leq \kappa+\eta^*+M_{2}\left(\eta^*+A^{-1}\right) \eta^*+M_{2} e^{-\frac{s_0}{3}}.
\end{aligned}
$$
Since $A \geq 1$, taking $\eta^*$ small enough and $s_0$ large enough, we get the result.
\paragraph{Case 2:} $s \geq s^*$. In this case, we have $s_0 \leq s^* \leq s \leq s_1$. We will show that starting at time $s^*, w$ will be trapped near the heteroclinic orbit $\psi$ \eqref{eq: heteroclinic orbit}, thanks to Proposition \ref{prop: stability R2}. In particular, its $L_{\rho_d}^2$ will remain bounded by $\frac{3}{2} \kappa$. More precisely, from item (ii) of Lemma \ref{lemma decreaing from k to k-eta} and \eqref{eq: definition of s etoile}, we see that

$$
\left\|W_{x_0}\left(s^*\right)-\left(\kappa-\eta^*\right)\right\|_{L_{\rho_d}^2} \leq M_{2}\left(\eta^*+A^{-1}\right) \eta^*+M_{2} e^{-\frac{s_0}{3}}.
$$
Assuming that $\eta^*<\kappa$, we may introduce $\sigma^* \in \mathbb{R}$ such that $\psi\left(\sigma^*\right)=\kappa-\eta^*$, where $\psi$ is defined in \eqref{eq: heteroclinic orbit}. Noting that
$$
\left|\psi^{\prime}\left(\sigma^*\right)\right| \sim \frac{\kappa e^{\sigma^*}}{p-1} \sim \kappa-\psi\left(\sigma^*\right)=\eta^* \text { as } \eta^* \rightarrow 0,
$$
we see that taking $\eta^*$ small enough, then $A$ and $s_0$ large enough, we have

$$
\left\|W_{x_0}\left(s^*\right)-\psi\left(\sigma^*\right)\right\|_{L_{\rho_d}^2} \leq \frac{\left|\psi^{\prime}\left(\sigma^*\right)\right|}{M_1\left[1+\frac{2}{\kappa}\left\|\psi^{\prime}\right\|_{L^{\infty}}\right]},
$$
where $M_1$ is introduced in Proposition \ref{prop: stability R2}. Since $\nabla W_{x_0}\left(s^*\right) \in L^{\infty}$, thanks to Proposition \ref{gradient estimation}, together with the definition \eqref{eq: initial data} of $w_1$ and the relation \eqref{eq: relation between W and w}, Proposition \ref{prop: stability R2} applies and we see that at time $s$, we have

$$
\begin{aligned}
	& \left\|W_{x_0}(s)-\psi\left(s+\sigma^*-s^*\right)\right\|_{L_{\rho_d}^2} \leq M_1\left\|W_{x_0}\left(s^*\right)-\psi\left(\sigma^*\right)\right\|_{L_{\rho_d}^2} \frac{\left|\psi^{\prime}\left(s+\sigma^*\right)\right|}{\left|\psi^{\prime}\left(\sigma^*\right)\right|} \\
	& \quad \leq \frac{\left|\psi^{\prime}\left(s+\sigma^*-s^*\right)\right|}{1+\frac{2}{\kappa}\left\|\psi^{\prime}\right\|_{L^{\infty}}} \leq \frac{\kappa}{2} .
\end{aligned}
$$
Since $\psi \leq \kappa$ by definition \eqref{eq: heteroclinic orbit}, we see that

$$
\left\|W_{x_0}(s)\right\|_{L_{\rho_d}^2} \leq \psi\left(s+\sigma^*-s^*\right)+\frac{\kappa}{2} \leq \frac{3}{2} \kappa .
$$
This concludes the proof of Corollary \ref{W_{x_0} control when x_0 far 1} .
\end{proof}
\subsubsection{Case where $|K|\leq A$}\label{Case when |K| smaller than A}
Given $m \in(0,1)$ and $x_0 \in \tilde{\mathcal{R}}_3$ defined in (5.9), we aim in this section to handle the case where $x_0$ is close to the blow-up sphere, namely when it belongs to the region $\mathcal{T}_0$ defined by

$$
\mathcal{T}_0=\left\{\left.K e^{-\frac{s_0}{2}}\right\rvert\, -A\leq K\leq A\right\}
$$
Let us recall from the beginning of Section \ref{Proof of the priori bound} that $W_{x_0}\in \mathcal{S}$ given in definition \ref{Def: shrinking set}, for all $s \in\left[s_0, s_1\right]$ for some $s_1 \geq s_0$, hence,

$$
\forall s \in\left[s_0, s_1\right], \quad\left\|W_{x_0}(s)\right\|_{L^{\infty}} \leq 2 \kappa
$$
Introducing the segment

\begin{equation}\label{ G sigma definition}
\mathcal{G}_\sigma=\left\{\left.K^{\prime} e^{-\frac{\sigma}{2}} \right\rvert\,  K'\geq \left(\frac{\vep_0}{8}-1\right)e^{s_0/2}\text { and } K^{\prime}=\pm A\right\},
\end{equation}
and the interval

$$
\mathcal{T}_1=\left\{\left.K^{\prime} e^{-\frac{s_1}{2}} \right\rvert\,  K'\geq \left(\frac{\vep_0}{8}-1\right)e^{s_0/2}\text { and } |K^{\prime}| \leq A\right\},
$$
we see that
$$
\mathcal{T}_0=\cup_{s_0 \leq \sigma<s_1} \mathcal{G}_\sigma \cup \mathcal{T}_1
$$
We then proceed in 2 steps:
\begin{itemize}[label=-]
\item  In Step 1, we handle the case of "small" $s$, namely when $x_0 \in \mathcal{G}_\sigma$ with $s_0 \leq s \leq \sigma \leq s_1$, and also the case where $x_0 \in \mathcal{T}_1$ with $s_0 \leq s \leq s_1$.
\item In Step 2, we handle the case of "large" $s$, namely when $x_0 \in \mathcal{G}_\sigma$ with $s_0 \leq \sigma \leq s \leq s_1$.
\end{itemize}
\paragraph{Step 1: Case of "small" $||x_0|-1|$ and "small" $s$}\,\\

Consider $x_0 \in \mathcal{G}_\sigma$ with $s_0 \leq s \leq \sigma<s_1$, or $x_0 \in \mathcal{T}_1$ with $s_0 \leq s \leq s_1$. The conclusion will follow from the $L_{\rho_d}^2$ estimate on $W_{x_0}$ together with the gradient estimate of Proposition \ref{gradient estimation}. Indeed, using a Taylor expansion together with \eqref{eq: second group ssv} and \eqref{rho d is a density function}, we write

$$
\begin{gathered}
		W_{x_0}(0, s)=W_{e_1}\left((|x_0|-1)e^{\frac{s}{2}}\mathbf{e}_1, s\right)=W_{e_1}(0, s)+O\left( ||x_0|-1|e^{\frac{s}{2}}\left\|\nabla W_{0}(s)\right\|_{L^{\infty}}\right) \\
		W_{\mathbf{e}_1}(0,s)=w_1(0,s)\\  
\end{gathered}
$$
on the one hand. 

On the other hand, since $w_1(0,s)=\tilde{w}(0,s)=\varphi(0,s)+q(0,s)$ by definition \eqref{eq: second group ssv}, \eqref{eq: def Chi}, \eqref{region R1}  and \eqref{eq: w decompose-in-terms-q-phi} on $s$,  recalling definition \eqref{eq: profile form} and \eqref{Def: shrinking set}, we write:
$$
\begin{aligned}
w_1(0,s) = \kappa + O(se^{-2s})
\end{aligned}
$$
Introducing  $K^{\prime}$ such that
$$
|x_0|-1=K^{\prime} e^{-\frac{\sigma}{2}}.
$$
with $\sigma=s_1$ if $a \in \mathcal{T}_1$, we see that
$$
-A \leq K^{\prime} \leq A 
$$
Recalling that $s \leq \sigma$, we write

$$
\left|(|x_0|-1) e^{\frac{s}{2}}\right|= |K'|e^{\frac{s-\sigma}{2}} \leq \sqrt{2} A
$$

Taking $A \geq 1$ and recalling that $\left\|\nabla W_0(s)\right\|_{L^{\infty}} \leq \delta_0$ from Proposition \ref{gradient estimation}, provided that $s_0 \geq s_{6}\left(A, \delta_0\right)$, we write from the previous estimates that

$$
\left|W_{x_0}(0, s)\right| \leq \kappa+C \delta_0+C A \delta_0 + O(se^{-2s})\leq \frac{3}{2} \kappa
$$
whenever $\delta_0 \leq \delta_{1}(A)$ and $s_0 \geq s_{13}\left(A, \delta_0\right)$ for some $s_{13}\left(A, \delta_0\right) \geq 1$ and $\delta_{1}(A)>0$. In particular, estimate \eqref{eq: priori estimate goal 2} holds.

\paragraph{Step 2: Case of "small" $\left||x_0|-1\right|$ and "large" $s$}\;

Now, we consider $x_0 \in \mathcal{G}_\sigma$ with $s_0 \leq \sigma \leq s \leq s_1$. As we will shortly see, the conclusion follows here from the case $|K| \geq A$ treated in Section \ref{Case large K}, if one replaces there $K$ and $s_0$ by $K^{\prime}$ and $\sigma$. Indeed, by definition \eqref{ G sigma definition} of $\mathcal{G}_\sigma$, we have
$$
K^{\prime}=\pm A
$$
where $K^{\prime}$ are defined in \eqref{eq: form of a in R3}. Following our strategy in Section \ref{Case large K}, we first start by expanding $W_{x_0}(Y  , \sigma)$ as in Lemma \ref{Initialvalue expansion for large x_0}:
\begin{lemma}[Expansion of $W_{x_0}(Y, \sigma)$]\label{Ex. W Y sigma} For any $A \geq 1$, there exists $s_{13}(A) \geq 1$ such that for any $s_0 \geq s_{13}(A)$, the following holds:
Assume that $U(t)\in \mathcal{S}(t)$, and $W_{x_0}(Y,s)$ is  defined as in \eqref{eq: second group ssv} satisfies equation \eqref{eq: W's equation} for any $s \in\left[s_0, \sigma\right]$ for some $\sigma \geq s_0$ such that $\nabla q\left(s_0\right) \in L^{\infty}\left(\mathbb{R}^2\right)$. Assume in addition that $|x_0|-1=K^{\prime}e^{-\frac{\sigma}{2}}$ such that $K= \pm A$ holds. Then,  for all $Y\in \Rbb^d$, 
 $$
\begin{aligned}
	W_{x_0}(Y , \sigma)=\kappa-\iota^{\prime} & +O\left(\frac{\iota^{\prime}}{A}\right)+O\left(\iota^{\prime 2}\right) \\
	-e^{-\sigma} & \left\{h_4(Y_1)+6K^2h_2(Y_2)
	+4K^3h_1(Y_1)+4Kh_3(Y_1)\right\} \\
	+q_{6}(\sigma) & \left\{h_6(Y_1)+5K'h_5(Y_1)+15K'^2h_4(Y_1)+20K'^3h_3(Y_1)\right. \\
	&\left.+15K'^4h_2(Y_1)+5h_1(Y_1)K'^5+K'^6h_0(Y_1)\right\}
\end{aligned}
$$
in $L_\rho^r$ for any $r \geq 2$, with $\left|q_{6}(\sigma)\right| \leq A \sigma e^{-2 \sigma}$, where

\begin{equation}\label{iota prime defintion}
\iota^{\prime}=e^{-\sigma}{K^{\prime}}^4
\end{equation}
\end{lemma}
\begin{proof}
	See in Section \ref{proof of technical details}
\end{proof}

Now, arguing as for Lemma \ref{Initialvalue expansion for large x_0}, we see $W_{x_0}(Y, \sigma)$ as initial data then integrate equation \eqref{eq: W's equation} to get an expansion of $W_{x_0}(Y,s)$ for later times:

\begin{lemma}[Decreasing $w_a(\sigma)$ from $\kappa$ to $\kappa-\eta^*$] There exists $M_{3}>0, A_{4} \geq 1$ and $\eta_{4}>0$ such that for all $A \geq A_{4}$ and $\eta^* \in\left(0, \eta_{4}\right]$, there exists $s_{14}\left(A, \eta^*\right)$ such that for any $\sigma \geq s_{14}$ and $s_1 \geq \sigma$, if $q(s) \in V_A(s)$ given in Definition \ref{Def: shrinking set} satisfies equation \eqref{q's equation} on $\left[\sigma, s_1\right]$ and $\nabla q(\sigma) \in L^{\infty}$, if $x_0=K^{\prime} e^{-\frac{\sigma}{2}}$ with $K^{\prime}= \pm A$, then:
\begin{enumerate}
\item $s^* \geq \sigma$, where $s^*$ is such that $e^{s^*-\sigma} \iota^{\prime}=\eta^*$ where $\iota^{\prime}$ is defined in \eqref{iota prime defintion}.
\item For all $s \in\left[\sigma, \min \left(s^*, s_1\right)\right]$,
\end{enumerate}
$$
\left\|W_{x_0}(s)-\left(\kappa-e^{s-\sigma} \iota^{\prime}\right)\right\|_{L_{\rho_d}^2} \leq M_{3}\left(\eta^*+A^{-1}\right) e^{s-\sigma} \iota^{\prime}+M_{3} e^{-\frac{\sigma}{3}}
$$
\end{lemma}
\begin{proof}
	The proof follows from a straightforward adaptation of the proof of Lemma \ref{lemma decreaing from k to k-eta}, given in this section. For that reason, the proof is omitted.
\end{proof}
As in the case $|K| \geq A$, we derive from this result the following corollary where we prove estimate \eqref{eq: Improved bound on W}:
\begin{corollary} (Proof of the $L_{\rho_d}^2$ bound when $x_0 \in \tilde{\mathcal{R}}_3$ is ``small'' and $s$ is ``large'') \label{W_x0 control when x_0 near 1 } There exists $A_{5}>0, \eta_{5}>0$ and $s_{15}\left(\eta^*\right) \geq 1$, such that under the hypotheses of Lemma \ref{Ex. W Y sigma}, if in addition $A \geq A_{5}, \eta^* \leq \eta_{5}$ and $s_0 \geq s_{15}\left(\eta^*\right)$, then, for all $s \in\left[\sigma, s_1\right]$, $\left\|W_{x_0}(s)\right\|_{L_{\rho_d}^2} \leq \frac{3}{2} \kappa$.
\end{corollary}
\begin{proof} The proof is omitted since the control follows from Lemma \ref{Ex. W Y sigma} exactly in the same way Corollary \ref{W_{x_0} control when x_0 far 1} follows from Lemma \ref{lemma decreaing from k to k-eta}
\end{proof}
\subsubsection{Concluding statement for the control of $W_{x_0}(s)$ when $x_0 \in R_3$ }
\begin{lemma}[Control of $W_{x_0}(s)$ when $x_0 \in R_3$]\label{Control of W_{x_0} when x_0 in R3}
Assume that $W_{x_0}\left(s_0\right)$ is given by \eqref{eq: initial data} and $q(s) \in \mathcal{V}(s)$. There exist $A_{6} \geq 1$ and $m_{6} \in(0,1)$ such that for all $A \geq A_{6}$,   there exists $\delta_{2}\left(A\right) \in (0,1)$ such that for all $\delta_0 \in (0, \delta_{2})$  there exists $s_{16}\left(A,\delta_0\right) \geq 1$ such that for all $s_0 \geq$ $s_{16}\left(A,\delta_0 \right), d \in \mathcal{D}\left(A, s_0\right)$ defined in Proposition \ref{prop: initialization} and $s_1 \geq s_0$, the following holds: Assume that $W_{x_0}$ is the solution of equation \eqref{eq: W's equation} with initial data $W_{x_0}\left(Y, s_0\right)$ defined in \eqref{eq: initial data}, such that for all $s \in\left[s_0, s_1\right], q(s) \in \mathcal{V}(s)$ given in Definition \ref{Def: shrinking set}, where $q(s)$ is defined in \eqref{eq: w decompose-in-terms-q-phi}. Then, for all $s \in\left[s_0, s_1\right]$ :
\begin{enumerate}[label=(\roman*)]
	\item $\|\nabla W(s)\|_{L^{\infty}} \leq \delta_0$.
	\item For all $m \in\left(0, m_{6}\right]$ and $x_0 \in \tilde{\mathcal{R}}_3 $ defined in \eqref{region R3}, $\left|W_{x_0}(0, s)\right|<2 \kappa$.
\end{enumerate}

Proof. The proof is omitted since it is straightforward from Claim \eqref{L infty estimation control by L rho}, Corollary \eqref{W_{x_0} control when x_0 far 1} in Step 1 and Corollary \ref{W_x0 control when x_0 near 1 }.
\end{lemma}
\subsection{Concluding statement for the control of  $W_{x_0}(s)$ for any $x_0 \in \mathbb{R}^d$}
Fixing $M$ as in \eqref{Choice of M} and taking $m=m_{25}$ introduced in Lemma \ref{Control of W_{x_0} when x_0 in R3}, we see that the 3 regions in \eqref{region R1}, \eqref{region R2} and \eqref{region R3} are properly defined. Then, combining the previous statements given for the different regions (namely, Corollaries \ref{exponential decay of W in R1} and \eqref{trapping of W_{x_0} near psi}, together with Lemma \ref{Control of W_{x_0} when x_0 in R3}), we derive the following:
\begin{proposition} \label{Control of W_{x_0} for all x_0}There exist $A_{7} \geq 1$ such that for all $A \geq A_{7}$, there exists $s_{17}(A) \geq 1$ such that for all $s_0 \geq s_{17}(A), d \in \mathcal{D}\left(A, s_0\right)$ defined in Proposition \ref{prop: initialization} and $s_1 \geq s_0$, the following holds: Assume that $W_{x_0}$ is the solution of equation \eqref{eq: W's equation} with initial data $w_1\left(y, s_0\right)$ defined in \eqref{eq: initial data}, such that for all $s \in\left[s_0, s_1\right], q(s) \in \mathcal{V}(s)$ given in Definition \eqref{Def: shrinking set}, where $q(s)$ is defined in \eqref{eq: w decompose-in-terms-q-phi}. Then:
\begin{enumerate}[label=(\roman*)]
\item For all $s \in\left[s_0, s_1\right],\left\|w_{1}(s)\right\|_{L^{\infty}}<2 \kappa$.
\item For all $x_0 \in \mathbb{R}^d$ such that

\begin{equation}\label{W exponential decay region x0}
||x_0|-1| \geq A e^{-\frac{s_1}{2}}
\end{equation}

there exist $\bar{s}(x_0) \geq s_0$ and $\bar{M}(x_0) \geq 0$ such that if $\bar{s}(x_0) \geq s_1$, then for all $s \in\left[\bar{s}(x_0), s_1\right]$, $\left\|W_{x_0}(s)\right\|_{L_{\rho_d}^2} \leq \bar{M}(x_0) e^{-\frac{s}{p-1}}$.
\end{enumerate}
\end{proposition}
\begin{proof}
\begin{enumerate}[label=(\roman*)]
\item The proof is omitted since it is straightforward from the above-mentioned statements.
\item Take $x_0 \in \mathbb{R}^d$ such that \eqref{W exponential decay region x0} holds. The conclusion follows according to the position of $x_0$ in the 3 regions $\tilde{\mathcal{R}}_i$ defined in \eqref{region R1}, \eqref{region R2} and \eqref{region R3} with the constants $M$ and $m$ fixed in the beginning of the current subsection.
If $x_0 \in \tilde{\mathcal{R}}_1$, then the conclusion follows from Corollary \ref{exponential decay of W in R1}.
If $x_0 \in \tilde{\mathcal{R}}_2$, then we see from Corollary \ref{trapping of W_{x_0} near psi} and its proof that Proposition \eqref{prop: stability R2} applies. In particular, $W_{x_0}(s)$ is trapped near the heteroclinic orbit $\psi$ defined in \eqref{eq: heteroclinic orbit}, and the exponential bound follows from \eqref{W exponential decay in R_2} and \eqref{eq: heteroclinic orbit}.
Finally, if $x_0 \in \tilde{\mathcal{R}}_3$, then, following Section \ref{Control of W_{x_0}(Y,s) in region R_3}, we write

\begin{equation}\label{x0ITR3F}
|x_0|-1=Ke^{-\frac{s_0}{2}}
\end{equation}
If $|K| \geq A$, then we see from Corollary \eqref{W_{x_0} control when x_0 far 1} and its proof that $W_{x_0}(s)$ is trapped near the heteroclinic orbit $\psi(s)$ defined in \eqref{eq: heteroclinic orbit} and the exponential bound follows from estimate \eqref{W exponential decay in R_2}
given in Proposition \ref{prop: stability R2}.
If $|K| \leq A$, using \eqref{W exponential decay region x0} and \eqref{form of x0 in R3}, we may write $x_0=K^{\prime} e^{-\frac{\sigma}{2}}$ for some $\sigma \in\left[s_0, s_1\right]$ with $K^{\prime}=A$, as we did in Section \ref{Case when |K| smaller than A}. Using Corollary \eqref{W_x0 control when x_0 near 1 }, we see that $W_{x_0}(s)$ is trapped near the heteroclinic orbit $\psi(s)$, and the conclusion follows from estimate \eqref{W exponential decay in R_2} again. This concludes the proof of Proposition \ref{Control of W_{x_0} for all x_0}
\end{enumerate}
\end{proof}
\section{Study of the projections of equation \eqref{q's equation}}\label{study of the projections}
From this section, we are going to complete all the proof for the technical lemmas used in earlier sections. There are two types of such kind of lemmas: Lemmas aimed to control $w_1$ defined in \eqref{eq: self-similar-variables} and Lemmas served for the apriori $L^\infty$ bound in Section \ref{Proof-a-priori-estimates}. We will proof the first type in this section, and leave the the second type in Section \ref{proof of technical details}.   

To control $w_1$, since the definition of the shrinking set $\mathcal{S}$, given by Definition \ref{Def: shrinking set} shows two different types of estimates, in the blow-up region and regular region. Accordingly, we need two different approaches to handle those estimates:
\begin{itemize}
	\item In the blow-up region, we work in similarity variables \eqref{eq: self-similar-variables}, in particular, we crucially use the projection of equation \eqref{q's equation} with respect to the decomposition given in \eqref{decomposition of q}.
	
	\item In the regular region, we directly work in the variables $u(x, t)$, using standard parabolic estimates.
\end{itemize}
With this idea, we organize this section as follows:
\begin{itemize}
\item In Section \ref{Proof-a-priori-estimates}, we will compute projections of each term in \eqref{q's equation} in order to analyze their contributions to $q$.
\item In Section \ref{Sec:POPRE}, we show the delay regularizing estimate used to control $q$ of operator $\mathcal{L}$ defined in \eqref{q's equation}. 
\item In Section \ref{Proof of control q1 to qe}, we will show Proposition \ref{Prop:control of q1,q2,q-,qe} saying that $q$ is in fact control by a finite elements.
\item In Section \ref{Details of initialization}, we will give the proof of Proposition \ref{prop: initialization}, which gives the good properties of the initial data given in \eqref{eq: initial data}.
\item Lastly, in Section \ref{section_regular_region}, we give the proof of the control in the regular region. 
\end{itemize}

%\subsection{Details on the dynamics of the linearized solution}
\subsection{Proof of Proposition \ref{Prop: Priori estimate}}\label{Proof-a-priori-estimates}

%\textbf{section A priori estimates}\\
In this section, we provide the proof of Proposition \ref{Prop: Priori estimate}. Specifically, we begin by projecting the linearized equation \eqref{q's equation} onto the Hermite polynomials to obtain equations satisfied by each coordinate in the decomposition \eqref{decomposition of q}. This approach will reveal the principal contributions within the projections $P_i$ (for $0\leq i\leq 2$) and $P_-$ arising from the various terms in equation \eqref{q's equation}.
The proof is structured as follows:
\begin{itemize}
	\item In the first part, we derive all the estimations on all the projections of each term appeared in \eqref{q's equation}.   
	\item In the second part, we sum up all the estimations then integrate in sapce to establish the final result stated in assertion $1$ and $2$ of Proposition \ref{Prop: Priori estimate}.
\end{itemize}
\textbf{Part 1: Study of the projections}
\medskip
The objective of this section, as previously stated, is to prove assertion (1) of Proposition \ref{Prop: Priori estimate}. To achieve this, we conduct a thorough analysis of the projections associated with the terms $V(y,s)$, $B(y,s)$, $H(y,s)$, $G(y,s)$, and $N(y,s)$. In the subsequent analysis, each projection is examined individually to clarify its specific contribution. Additionally, the projection of $R(y,s)$ is computed here as well for use in subsequent sections.
\medskip

\noindent\textbf{First term $\partial_s q$}\\
Let $P_i$ and $P_-$ defined as in \eqref{eq: Def of projections P_m} and \eqref{eq: Def of projection P_-}, then the following holds:
\begin{equation}
	\begin{aligned}
		P_i (\del_s q)&= \del_s q_i \mbox{ with } i \leq 6, \\
		P_-(\del_s q)&= \del_s q_-.
	\end{aligned}
\end{equation}
%	\end{lemma}
\noindent\textbf{Second term $\mathcal{L} q$}\\
By the definition of $h_i$ given by \eqref{eq:definition of hm}, we easily obtain the projection of $\mathcal{L} q$ as follows
\begin{lemma}\label{second term}
Let $P_i$ and $P_-$ defined as in \eqref{eq: Def of projections P_m} and \eqref{eq: Def of projection P_-}, then the following holds:
\begin{equation}
	\begin{aligned}
		P_i (\mathcal{L} q)&= \left (1-\frac{i}{2}\right ) q_i,\mbox{ for }i \leq 6. \\
		%P_1 (\mathcal{L} q)&=  \frac{1}{2} q_1,\\
		%P_2(\mathcal{L} q)&= 0,\\
		P_-(\mathcal{L} q)&= \mathcal{L} q_-.
	\end{aligned}
\end{equation}
\end{lemma}
\noindent\textbf{Third term $Vq$}\\
In this part we aim to derive esimation in \eqref{Potential control} on the potential $V$.  It is stated as following: 

\begin{lemma}\label{Third term}
Let potential $V$ be defined in \eqref{eq: linear operator} the following estimations holds:
 for any $r>1$,
	\begin{equation}
		\|V\|_{L_{\rho_1}^r(\Rbb)}\leq C(r)e^{-s}.
	\end{equation} 
\end{lemma}
\begin{proof}
To begin with, we analyze 
\begin{equation}
	V= p\varphi^{p-1}-\frac{p}{p-1}
\end{equation}
in its expanded form:
\begin{equation}\label{eq: V's expansion}
	V=-p(p-1)\kappa^{p-2} e^{-s}h_4+\frac{1}{2}p^2(p-1)\kappa^{p-3} e^{-2s}h_8+\frac{p(p-1)(p-2)}{2}\kappa^{p-3}e^{-2s}h_4^2.
\end{equation}
Here, the term $p\varphi^{p-1}$ is expanded around $\kappa$ as following:
\begin{equation}
	p\varphi^{p-1}=\frac{p}{p-1}+p(p-1)\kappa^{p-2}(\varphi-\kappa)+\frac{p(p-1)(p-2)}{2}\kappa^{p-3}(\varphi-\kappa)^2+O((\varphi-\kappa)^3),
\end{equation}
and \eqref{eq: V's expansion} is conclude with the help of \eqref{eq:flat profile}.

Given the relation $h_4^2=h_8(y)+32h_6(y)+408h_4(y)+2208h_2(y)+1824$, we derive the estimates for $\|V\|_{L_{\rho_1}^r(\Rbb)}$.
\end{proof}
\paragraph{Fourth term $B(y,s)$} \; \\
For the bilinear term $B(y,s)$ we have the following priori estimate established by Lemma 3.6 in \cite{MZdm97}
\begin{lemma}
For all $A > 0$, there exists $s_{0}\geq 0$ such that for all $\tau \geq s_{0}$, if $q(\tau) \in \mathcal{V}(\tau)$, then
\begin{equation}\label{eq:B inner estimation}
	|\chi_c(y, \tau) B(q(y, \tau))| \leq C |q|^2, 
\end{equation}
and
\begin{equation}\label{eq:B outer estimation}
	|B(q)| \leq C |q|^{\bar{p}} ,
\end{equation} 
where $\bar{p} = \min(p, 2)$.
\end{lemma}
Our estimation on $B(y,s)$ is then stated as:
\begin{lemma}\label{Fourth term}
For all $A>0$, there exists an \(s_{17}\geq 0\) such that for all \(\tau \geq s_{0}\geq s_{17}\), if \(q(\tau)\in \mathcal V(\tau)\), the following estimations holds: 
\begin{enumerate}
	\item For $i \leq 6$, 
	\begin{equation}\label{estimation P02(B) central}
		P_i(B)\leq CAse^{-3s},
	\end{equation}
	\begin{comment}
		\item For $i>6$,
		\begin{equation}\label{estimation Pi(B) stable}
			P_{-}(B)\leq CAs^2e^{-3s}(1+|y|^{8}).
		\end{equation}
	\end{comment}
\end{enumerate}
\end{lemma}
\begin{proof}
When $i\leq 6$, multiplying $B(y,s)$ by $k_i=\frac{h_i(y)}{\|h\|_{L_{\rho_1}^2}}$ then integrate in $L_{\rho_1}(\Rbb)$ gives :
\[
|P_i(B)|\leq \int_{\Rbb}|B(y,s)|k_i\rho_1(y)dy.
\]
Now by \eqref{eq:B outer estimation}, we deduce that:
\[
|P_i(B)|\leq \int_{\Rbb}C(M)|q|^{\bar p}k_i\rho_1(y)dy.
\]
Estimation \eqref{estimation P02(B) central} then follows from \eqref{q L rho priori estimate} together with Cauchy-Schwartz inequality.
\end{proof}

\begin{comment}

\noindent\textbf{Fifth term $B$}:

For the quadratic term $B$, we first remind the readers of the following Lemma:
\begin{lemma}
	For all $A > 0$, there exists $s_{0}\geq 0$ such that for all $\tau \geq s_{0}$, if $q(\tau) \in \mathcal{V}(\tau)$, then
	\begin{equation}\label{eq:B inner estimation}
		|\chi_c(y, \tau) B(q(y, \tau))| \leq C |q|^2, 
	\end{equation}
	and
	\begin{equation}\label{eq:B outer estimation}
		|B(q)| \leq C |q|^{\bar{p}} ,
	\end{equation} 
	where $\bar{p} = \min(p, 2)$.
\end{lemma}
\begin{proof}
	This Lemma was argued in Lemma 3.6 of \cite{MZdm97}, interested readers are invited to read the proof in \cite{MZdm97}.
\end{proof}
Then we are enabled to claim the following lemma:
\begin{lemma}\label{fifthterm}
	There exits $s_0\geq 0$ such that if $q(s)\in \mathcal{V}(s)$ for $s>s_0$, then $B$ verifies:
	\begin{equation}
		\begin{aligned}
			P_i(B)&\leq CA^2s^{-3}, \quad i\in \{0,1,2\}\\
			%B_2 &\leq Cs^{-3},\\
			P_{-}(B)& \leq CAs^{-2} (1+|y|^3).\\
		\end{aligned}
	\end{equation}
\end{lemma}
\begin{proof}
	We argue as in the proof of Lemma 5.10 and Lemma 5.17 in \cite{MZjfa08}.
\end{proof}
\end{comment}
\noindent\textbf{Sixth term $H$:}

\begin{lemma}\label{sixth term}
The following estimations holds: 
\begin{enumerate}
	\item for $i\leq6$	
    \begin{equation}
		\begin{aligned}
			P_i(H)&\leq  Ce^{-3s},\\
		\end{aligned}
	\end{equation}
	\item \begin{equation}
		\|H\|_{L_{\rho_1}^2(\Rbb)} \leq Ce^{-3s}.
	\end{equation}
\end{enumerate}
\end{lemma}
\begin{proof} We argue it as in the proof of Lemma 3.9 from  \cite{MNZNon2016}.
\end{proof}
\noindent\textbf{Seventh term $\del_y G$}:
\begin{lemma}\label{seventh term}
For $\del_y G$ we have the following estimations: 
\begin{enumerate}
	\item for $i\leq 6$ 
    \begin{equation}
		\begin{aligned}
			P_i(\del_y G)&\leq  Ce^{-3s},  \\
		\end{aligned}
	\end{equation}
	\item \begin{equation}
		\|\del_y G\|_{L_{\rho_1}^2(\Rbb)}\leq    Ce^{-3s},
	\end{equation}
\end{enumerate}
\end{lemma}
\begin{proof}
This can be done with integration by part; interested readers are invited to see the proof of Lemma 5.19 in \cite{MNZNon2016}. 
\end{proof}
\noindent\textbf{Eighth term $N$: }
\begin{lemma}[projection of the last term: $N$ ]\label{Eighth term}For all $A>0$, there exists an \(s_{18}\geq 0\) such that for all \(\tau \geq s_{0}\geq s_{18}\), if \(w_1(y,s)\in \mathcal {V}(s)\), the following estimations holds: 
\begin{enumerate}
		\item \begin{equation}\label{control of $N$ on H+}
		|P_i(N)|\leq Ce^{-3s}\mbox{ where }i\;\leq\;6.
	\end{equation}
	\item \begin{equation}
		\|N\|_{L_{\rho_1}^2(\Rbb)}\leq Ce^{-3s}.
	\end{equation}
\end{enumerate}
\end{lemma}
\begin{proof}
Let us first recall that  
\[
N(y,s)= \frac{d-1}{y+e^{s/2}}w_1\del_y\Xcal,
\]  
where $w_1$ is defined by \eqref{eq: W}, and $\chi$ denotes the cut-off function introduced in \eqref{eq: def Chi}.  

We now proceed to estimate the terms $P_i(N)$ for $i \in \{0,1,2\}$. According to Definition \ref{Def: shrinking set}, we know that for $|y| \geq e^{\frac{s}{2}}(\frac{3}{8}\vep_0 - 1)$, the $L^{\infty}$ norm of $w_1(s)$ satisfies $\|w_1(s)\|_{L^\infty} \leq 2\kappa $. Moreover, using the definition \eqref{eq: def Chi} of $\chi$, we readily obtain the estimate
\begin{equation}\label{eq:estim dely Chi}
	|\del_y \mathcal{\chi}|\leq e^{-s/2}\frac{C}{\vep_0}\mathbb{I}_{\{(\frac{3}{8}\vep_0-1)e^{s/2}\leq y \leq (\frac{3}{4}\vep_0-1)e^{s/2}\}}.
\end{equation}

\textbf{1.} Using the estimates above, we derive
\begin{equation}\label{eq:estim:abs:eta0}
	\begin{aligned}
		|P_0(N)| 
		&\leq e^{-s/2}\int_{\Rbb}\left|\frac{d-1}{ye^{-s/2}+1}w_1\del_y\Xcal\right|\frac{e^{-\frac{|y|^2}{4}}}{(4\pi)^{1/2}}dy,\\
		&\leq C e^{-s} \int_{\{y\geq (\frac{3}{8}\vep_0-1)e^{s/2}\}}\left|\frac{d-1}{ye^{-s/2}+1}\right|\frac{e^{-\frac{|y|^2}{4}}}{(4\pi)^{1/2}}dy,\\
		&\leq C e^{-s} \int_{\{y\geq (\frac{3}{8}\vep_0-1)e^{s/2}\}}\frac{e^{-\frac{|y|^2}{4}}}{(4\pi)^{1/2}}dy\\
		&\leq C e^{-s}e^{-(\frac{3}{8}\vep_0-1)^2e^{s}/4} \int_{\{y\geq (\frac{3}{8}\vep_0-1)e^{s/2}\}}\frac{e^{-\frac{|y|^2}{8}}}{(4\pi)^{1/2}}dy\leq C e^{-3s}.
	\end{aligned}
\end{equation}

By applying the same reasoning, we can derive the corresponding estimates for $P_i(N)$ with $i \leq 6$.  

\textbf{2.} Regarding the estimate of $\|N\|_{L_{\rho_1}^2(\Rbb)}$, we proceed as follows:  
\begin{equation}
	\begin{aligned}
		\|N\|^2_{L_{\rho}^2(\Rbb)} 
		&= e^{-s}\int_{\Rbb}\left|\frac{d-1}{ye^{-s/2}+1}w_1\del_y\Xcal\right|^{2}\frac{e^{-\frac{|y|^2}{4}}}{(4\pi)^{1/2}}dy,\\
		&\leq C e^{-2s} \int_{\{y\geq (\frac{3}{8}\vep_0-1)e^{s/2}\}}\left|\frac{d-1}{ye^{-s/2}+1}\right|^2\frac{e^{-\frac{|y|^2}{4}}}{(4\pi)^{1/2}}dy,\\
		&\leq C e^{-2s} \int_{\{y\geq (\frac{3}{8}\vep_0-1)e^{s/2}\}}\frac{e^{-\frac{|y|^2}{4}}}{(4\pi)^{1/2}}dy\\
		&\leq C e^{-2s}e^{-(\frac{3}{8}\vep_0-1)^2e^{s}/4} \int_{\{y\geq (\frac{3}{8}\vep_0-1)e^{s/2}\}}\frac{e^{-\frac{|y|^2}{8}}}{(4\pi)^{1/2}}dy\leq C e^{-6s}.
	\end{aligned}
\end{equation}

This concludes the desired estimate.
\end{proof}

\noindent\textbf{Fifth term $R(y,s)$}

\begin{lemma}[Estimates for term $R$]\label{Lemma:esti:R}\;

\noindent For $i \leq 6 $,
\begin{equation}\label{eq: estimation P024R}
	|P_{i}(R)|\leq Ce^{-2s};
\end{equation}
\noindent and we have also:
\begin{equation}
\|P_-(R)\|_{L^2_{\rho_1}}\leq Ce^{-3s}.
\end{equation}
\end{lemma}
\begin{proof}
To give the estimates on $R$, we first compute each term in $R$ its Taylor expansion in $L_{\rho_1}^r$ with $r \geq 2$. To have a better vision of this, we remind the readers that
\[R(y,s)= \del^2_y \varphi-\frac{1}{2}y\del_y \varphi-\frac{1}{p-1}\varphi+\varphi^p-\del_s \varphi,\]
with \eqref{eq:flat profile}, we build the estimations for $i \leq 6$ by the following expansions:\\

\noindent \bigskip
\textbf{expansion of $\del^2_y \varphi$}
\begin{equation}
	\begin{split}
		\del^2_y\varphi &= -e^{-s}(12h_2)+\frac{e^{-2s}}{\kappa}\left(28p y^6-360y^4+(1872-864p)y^2+(2-p)288\right)+O(e^{-3s});
	\end{split}
\end{equation}
\textbf{expansion of $-\frac{1}{2}y\del_y \varphi$}
\begin{equation}
	\begin{split}
		-\frac{1}{2}y\del_y \varphi&=e^{-s}(2h_4+12h_2)\\
		&-\frac{e^{-2s}}{\kappa}\left(2py^8-36y^6+(312-144p)y^4-144(2-p)y^2\right)+O(e^{-3s});
	\end{split}
\end{equation}
\textbf{expansion of $-\frac{1}{p-1}\varphi$}
\begin{equation}
	\begin{split}
		-\frac{1}{p-1}\varphi&=-\frac{\kappa}{p-1}+\frac{e^{-s}}{p-1}h_4\\
		&-\frac{e^{-2s}}{\kappa(p-1)}\left(\frac{p}{2}y^8-12y^6+(156-72p)y^{4}-(2-p)144y^2+72(2-p)\right)\\
		&+O(e^{-3s});
	\end{split}
\end{equation}
\\
\textbf{expansion of $\varphi^p$}
\begin{equation}
	\begin{split}
		\varphi^p=&\frac{\kappa}{p-1}-\frac{pe^{-s}}{p-1}h_4\\
		&+\frac{pe^{-2s}}{p-1}\left(\frac{p}{2\kappa}y^8-\frac{12}{\kappa}y^6+(156-72p)y^{4}-(2-p)144y^2+72(2-p)\right)\\
		&+O(e^{-3s});\\
	\end{split}
\end{equation}
\textbf{expansion of $\del_s \varphi$ }
\begin{equation}
	\begin{split}
		\del_s\varphi&=e^{-s}h_4-\frac{2}{\kappa}e^{-2s}\left(\frac{p}{2}y^8-12y^6+(156-72p)y^{4}-(2-p)144y^2+72(2-p)\right)\\
		&+\frac{p}{2\kappa}e^{-2s}h_4^2+O(e^{-3s}).
	\end{split}
\end{equation}
We utilize the relation $h_4^2=h_8(y)+32h_6(y)+408h_4(y)+2208h_2(y)+1824$.
and sum up all the expansions, we obtain that
\begin{equation}
	\begin{aligned}
		R(y,s)=&\frac{pe^{-2s}}{\kappa}(3ph_8+(48+184p)h_6+(2268+3764p)h_4)\\
		&+\frac{pe^{-2s}}{\kappa}((20880+20880p)h_2(y)+19584+14016p)+O(e^{-3s}).
	\end{aligned}
\end{equation}
It directly gives the estimation for $P_i(R)$ with $i\leq 6$. Furthur, that
\medskip
$$
\|P_{-}(R)\|_{L^2_{\rho_1}}\leq \int_{\Rbb}\left|R-\sum_{i=0}^{6}P_{i}(R)\frac{h_{i}}{\|h_i\|_{L_{\rho_1}^2}}(y)\right|\rho_{1}(y)dy\leq Ce^{-3s}
$$
is then build upon \eqref{eq: estimation P024R}, and \eqref{eq: Def of projection P_-}.
\end{proof}

\paragraph{Proof of item 1 in Proposition \ref{Prop: Priori estimate}}\;\\
Consider a solution $q(s) \in V_A(s)$ of equation \eqref{q's equation} on a time interval $[s_0, s_1]$, with some $A>0$. Observe that $\|q(s)\|_{L^{\infty}}$ remains uniformly bounded; thus, Lemma \ref{Fourth term} applies. Suppose the initial condition \eqref{eq: priori estimate q} is satisfied. For any $s \in [s_0, s_1]$ and any $i \leq 6$, multiplying equation \eqref{q's equation} by $\frac{h_m}{\|h\|^2_{L_{\rho_1}^2(\Rbb)}}$ yields the desired result. This follows from the definition of $q_i$ in \eqref{eq: Def of projections P_m}, along with Lemmas \ref{second term}, \ref{Third term}, \ref{Fourth term}, \ref{sixth term}, \ref{seventh term}, \ref{Eighth term}, and an application of Hölder's inequality.
\paragraph{Part 2: Proof of the item 2 in Proposition \ref{Prop: Priori estimate}}\;\\
If $P_{-}$ is the $L_{\rho_1}^2$  projector on 
\begin{equation}
E_{-} \equiv \operatorname{span}\left\{h_i \mid i > 6 \right\},
\end{equation}
then $q_{-}(s)=P_{-}(q(s))$ and $R_{-}(s)=P_{-}(R(s))$. Applying this projector to equation \eqref{q's equation}, we get
$$
\partial_s q_{-}=\mathcal{L} q_{-}+P_{-}(V q+B+H+\del_y G+N)+R_{-}
$$
Multiplying by $q_{-} \rho$ then integrating in space, we write
$$
\begin{aligned}
\frac{1}{2} \frac{d}{d s}\left\|q_{-}\right\|_{L_{\rho_1}^2}^2=&\int_{\mathbb{R}} q_{-} \mathcal{L}_{-} q_{-} d y\\
&+\int_{\mathbb{R}} q_{-}\left[P_{-}(V q+B+H+\del_y G+N)+R_{-}\right] \rho d y
\end{aligned}
$$
Since the highest eigenvalue of $\mathcal{L}$ on $E_{-}$is $\lambda=-3$ (see \eqref{eq: spectrum}), it follows that

$$
\int_{\mathbb{R}^2} q_{-} \mathcal{L} q_{-} \rho d y \leq-3\left\|q_{-}\right\|_{L_\rho^2}^2
$$
Recalling that $|B| \leq C|q|^{\bar{p}}$ where $\bar{p}=\min (p, 2)>1$ thanks to \eqref{eq:B outer estimation}, then using the Cauchy-Schwarz inequality, we write

$$
\begin{aligned}
&\left|\int_{\mathbb{R}} q_{-}\left[P_{-}(V q+B)+R_{-}\right] \rho d y\right| \\ &\leq\left\|q_{-}\right\|_{L_{\rho_1}^2}\left[\left\|P_{-}(V q+B+H+\del_y G+N)\right\|_{L_{\rho_1}^2}+\left\|R_{-}\right\|_{L_{\rho_1}^2}\right] \\
& \leq\left\|q_{-}\right\|_{L_{\rho_1}^2}\left[\|V q\|_{L_{\rho_1}^2}+C\|B\|_{L_{\rho_1}^2}+C\|H\|_{L_{\rho_1}^2}+C\|\del_y G\|_{L_{\rho_1}^2}+C\|N\|_{L_{\rho_1}^2}+\left\|R_{-}\right\|_{L_{\rho_1}^2}\right] \\
& \leq\left\|q_{-}\right\|_{L_{\rho_1}^2}\left[\|V\|_{L_{\rho_1}^4}\|q\|_{L_{\rho_1}^4}+C\|q\|_{L_{\rho_1}^{2 \bar{p}}}^{\bar{p}}+\left\|R_{-}\right\|_{L_{\rho_1}^2}\right]+Ce^{-3s} .
\end{aligned}
$$
Using the delay regularizing estimate given in Proposition \ref{delay regularizing estimate}, together with the Lemma \ref{Third term} satisfied by $V$, we conclude the proof of Proposition \ref{Prop: Priori estimate}.

\subsection{Proof of the delay regularizing estimate}\label{Sec:POPRE}
Now, the only thing left for the proof of Proposition \ref{Prop: Priori estimate} is to prove Proposition \ref{delay regularizing estimate}, the delay regularizing estimate use just before to control $q_-$. Before that, we first introduce the following statement about the regularizing effect: 
\begin{lemma}\label{lemma: Reg eff d oper}(Regularizing effect of the operator $\mathcal{L}$)
	\begin{enumerate}
		\item \cite{HVaihn93} For any $r>1$, $\bar{r}>1$,$v_0\in L_{\rho}^r(\mathbb{R}^d)$ and $s>\max\left\{0,-\ln\frac{r-1}{\bar{r}-1}\right\}$, it holds that
		\begin{equation}
			\|e^{s\mathcal{L}}v_0\|_{L_{\rho}^{\bar r}}(\Rbb^d)\leq\frac{C(r,\bar r)e^{s}}{(1-e^{-s})^{\frac{N}{2r}}(r-1-e^{-s}(\bar r -1 ))^{\frac{N}{2\bar r}}}\|v_0\|_{L_{\rho}^r(\Rbb^d)}
		\end{equation}
		\item Consider \( r \geq 2 \) and \( v_0 \in L^r_\rho(\mathbb{R}^d) \) such that 
		\[ 
		|v_0(y)| + |\nabla v_0(y)| \leq C(1 + |y|^k) 
		\] 
		for some \( k \in \mathbb{N} \). Then, for all \( s \geq 0 \), we have
		\[ 
		\| e^{s\mathcal{L}}v_0 \|_{L^r_\rho(\mathbb{R}^d)} \leq Ce^{s} \| v_0 \|_{L^r_\rho(\mathbb{R}^d)}.
		\]
		\item For any \( v_0 \in L^\infty(\mathbb{R}^d) \) and \( s \geq 0 \), it holds that 
		\( e^{s\mathcal{L}}v_0 \in L^\infty \) with 
		\[
		\| e^{s\mathcal{L}}v_0 \|_{L^\infty} \leq Ce^{s} \| v_0 \|_{L^\infty}.
		\]
		\item For any \( v_0 \in W^{1,\infty}(\mathbb{R}^d) \) and \( s > 0 \), it holds that 
		\( e^{s\mathcal{L}} \nabla v_0 \in L^\infty \) with 
		\[
		\| e^{s\mathcal{L}} \nabla v_0 \|_{L^\infty} \leq \frac{Ce^{s}}{\sqrt{1 - e^{-s}}} \| v_0 \|_{L^\infty}.
		\]
	\end{enumerate}
\end{lemma}
\begin{proof}
    This is exactly the same lemma as Lemma 7.1 on Page 44 of \cite{MZJEMS24}. We hence omit here. Interested readers are kindly invited to read Lemma 7.1 on Page 44 of \cite{MZJEMS24}.
\end{proof}
With the help of Lemma \ref{lemma: Reg eff d oper}, we are ready to give the proof of Proposition \ref{delay regularizing estimate}.
\paragraph{Proof of Proposition \ref{delay regularizing estimate}}
 Let us consider $A>0$ and $s_1 \geq s_0 \geq 0$, and assume that $q(y, s)$ satisfies equation \eqref{q's equation} for all $(y, s) \in [(\frac{\vep_0}{8}-1)e^{-s},+\infty) \times\left[s_0, s_1\right]$, with $q(s) \in V_A(s)$ defined in \eqref{Def: shrinking set}. Assume also that the initial condition \eqref{eq: initial data} holds. Note in particular that $\|q(s)\|_{L^{\infty}}$ is uniformly bounded. Using \eqref{q's equation} together with \eqref{eq; estimation bar V}, we see that for some universal constant $C^*>0$ and for almost every $(y, s) \in [(\frac{\vep_0}{8}-1)e^{-s},+\infty) \times\left[s_0, s_1\right]$, we have
\begin{equation}\label{reduce q's equation}
\partial_s|q| \leq\left(\mathcal{L}+C^*\right)|q|+|\tilde{R}|,
\end{equation}
where
$$
\tilde{R}=|R|+|H|+|N|+|\del_yG|.
$$

Consider now some $r \geq 2$ and introduce $s^*=-\log \left(\frac{2-1}{r-1}\right) \geq 0$, which is involved in the statement of Lemma \ref{lemma: Reg eff d oper}. Consider then some $s \in\left[s_0, s_1\right]$, and introduce
\begin{equation}\label{s' definition}
s^{\prime}=\max \left[s_0, s-s^*\right].
\end{equation}

We then introduce the following Duhamel formulation of \eqref{reduce q's equation} on the interval $\left[s^{\prime}, s\right]$ :

$$
|q(s)| \leq e^{\left(\mathcal{L}+C^*\right)\left(s-s^{\prime}\right)}\left|q\left(s^{\prime}\right)\right|+\int_{s^{\prime}}^s e^{\left(\mathcal{L}+C^*\right)(s-\tau)}|\tilde{R}(\tau)| d \tau,
$$
which implies that

\begin{equation}\label{qs L rho decomposition}
\|q(s)\|_{L_{\rho_{1}}^r\left(\mathbb{R}\right)} \leq I+J,
\end{equation}
where

$$
I=\left\|e^{\left(\mathcal{L}+C^*\right)\left(s-s^{\prime}\right)}\left|q\left(s^{\prime}\right)\right|\right\|_{L_{\rho_1}^r\left(\mathbb{R}\right)} \text { and } J=\int_{s^{\prime}}^s\left\|e^{\left(\mathcal{L}+C^*\right)(s-\tau)}|\tilde{R}(\tau)|\right\|_{L_{\rho_1}^r\left(\mathbb{R}\right)} d \tau.
$$

Let us first bound $J$ then $I$.
Since $\tilde{R}(y, s)$ and $\nabla \tilde{R}(y, s)$ are clearly bounded by a polynomial in $y$ by definition \eqref{definition of B,R,F,N}, we write from item (ii) of Lemma \ref{lemma: Reg eff d oper} and the bound form Lemmas \ref{sixth term}, \ref{seventh term},\ref{Eighth term}, and \ref{Lemma:esti:R} on each components from $\tilde{R}$ :

$$
\begin{aligned}
	J & \leq C \int_{s^{\prime}}^s e^{\left(1+C^*\right)(s-\tau)}\|\tilde{R}(\tau)\|_{L_{\rho_1}^r\left(\mathbb{R}\right)} d \tau \\
	& \leq C \int_{s^{\prime}}^s e^{\left(1+C^*\right)(s-\tau)} e^{-2 \tau} d \tau \leq e^{\left(1+C^*\right)\left(s-s^{\prime}\right)}\left(s-s^{\prime}\right) e^{-2 s^{\prime}} d \tau
\end{aligned}
$$

Since $s^{\prime} \geq s-s^*$ from \eqref{s' definition}, it follows that $s-s^{\prime} \leq s^*$ and $e^{-2 s^{\prime}} \leq e^{2 s^*-2 s}$, hence
\begin{equation}\label{J estimation}
J \leq C s^* e^{\left(3+C^*\right) s^*} e^{-2 s}
\end{equation}
In order to bound $I$, we consider 2 cases:
\begin{itemize}
\item If $s-s^* \geq s_0$, then $s^{\prime}=s-s^*$ and $s-s^{\prime}=s^*$. Using item (i) in Lemma \ref{lemma: Reg eff d oper}, we write $I \leq C(r)\left\|q\left(s-s^*\right)\right\|_{L_{\rho_1}^2}$.
Since $q\left(s-s^*\right) \in V_A\left(s-s^*\right)$ by hypothesis, it follows from \eqref{q L rho priori estimate} that

$$
\left\|q\left(s-s^*\right)\right\|_{L_{\rho_1}^2} \leq C A\left(s-s^*\right) e^{-2\left(s-s^*\right)} \leq C e^{2 s^*} A s e^{-2 s}.
$$

Therefore, we conclude that

\begin{equation}\label{I estimation s star small}
I \leq C(r) A s e^{-2 s}.
\end{equation}

\item If $s-s^*<s_0$, then $s^{\prime}=s_0$. This time, from hypothesis \eqref{eq: priori estimate q}, we see that we can apply item (ii) of Lemma \ref{lemma: Reg eff d oper} to write

$$
I \leq C e^{\left(1+C^*\right)\left(s-s_0\right)}\left\|q\left(s_0\right)\right\|_{L_{\rho_1}^r} \leq C(r) e^{\left(1+C^*\right) s^*} A s_0 e^{-2 s_0}.
$$

Since $s_0 \leq s<s_0+s^*$, it follows that

\begin{equation}\label{I estimation s star large}
I \leq C(r) A s e^{-2 s}.
\end{equation}
\end{itemize}
Combining \eqref{qs L rho decomposition}, \eqref{J estimation}, \eqref{I estimation s star large} and \eqref{I estimation s star small} concludes the proof of Proposition \ref{delay regularizing estimate}.\qed
\subsection{Proof of Proposition \ref{Prop:control of q1,q2,q-,qe}}\label{Proof of control q1 to qe}
Arguing as the proof of Proposition 4.4 in  \cite{MZIMRN22}. Consider $i \leq 5$, and assume that
\[
q_{i}(s_1) = \theta A e^{-2 s_1},
\]
for some $\theta = \pm 1$. Since Lemma \ref{Lemma:esti:R} implies that $|R_{i}(s)| \leq C e^{-2 s}$ by Definition \ref{Def: shrinking set}, we write from Lemma \ref{second term}, Lemma \ref{Third term}, Lemma \ref{Fourth term}, Lemma \ref{Lemma:esti:R}, Lemma \ref{sixth term}, Lemma \ref{seventh term} and Lemma \ref{Eighth term},
\begin{equation}
\begin{aligned}
	\theta q_{i}'(s_1) &\geq \left(1 - \frac{i}{2}\right) A e^{-2 s_1} - C A s_1 e^{-3 s_1} - C e^{-2 s_1} -Ce^{-3s_1}\\
	&\geq \left(1 - \frac{i}{2}\right) A e^{-2 s_1},
\end{aligned}
\end{equation}
on the one hand, taking $A$ large enough, then $s_0$ large enough. On the other hand, we have
\begin{equation}
\left. \frac{d}{ds} \left( A e^{-2 s} \right) \right|_{s = s_1} = -2 A e^{-2 s_1}. 
\end{equation}
Since 
\[
1 - \frac{i}{2} \geq 1 - \frac{5}{2} = -\frac{3}{2} > -2,
\]
the conclusion follows.
This concludes item (i)  of Proposition \ref{Prop:control of q1,q2,q-,qe}. 
The case of item (ii) is similar. Assume that
\[
q_{6}(s_1) = \theta As e^{-2 s_1},
\]
upon Definition \ref{Def: shrinking set}, Lemma \ref{second term}, Lemma \ref{Third term}, Lemma \ref{Fourth term}, Lemma \ref{Lemma:esti:R}, Lemma \ref{sixth term}, Lemma \ref{seventh term} and Lemma \ref{Eighth term}, we obtain that

\begin{equation}
\begin{aligned}
	\theta q_6'(s_1) &\leq -2CAs_1e^{-2s_1} + C A s_1 e^{-3 s_1} +C e^{-2 s_1} + Ce^{-3s_1},\\
	&\leq CAe^{-2s_1}(-2s_1+s_1e^{-s_1}+\frac{1}{A}+\frac{Ce^{-s_1}}{A}).
\end{aligned}
\end{equation}
again on the one hand, taking large enough $A_0$ and $s_0$, we have that
\begin{equation}
\theta q_6'(s_1) < (1-2s_1)CAe^{-2s_1}.
\end{equation}

On the other hand, 
\begin{equation}
\left. \frac{d}{ds} \left( A s e^{-2 s} \right) \right|_{s = s_1} = (1-2s_1) A e^{-2 s_1}. 
\end{equation}
This concludes the proof of item (iii) in  Proposition \ref{Prop:control of q1,q2,q-,qe}. 

\subsection{Details of initialization}\label{Details of initialization}
In this section, we will show the two good properties of the initial data given by \eqref{eq: initial data}. They are mentioned in Proposition \ref{prop: initialization}.

\noindent \textbf{Proof of Proposition}\ref{prop: initialization}. We will be using the notation $d=\left(d_{0},d_{1}, d_{2}, d_{3},d_4,d_5\right)$ for simplicity. Let us consider $A \geq 1$ and $s_0 \geq 1$. The first item (i) will be proved for any $d \in[-2,2]^6$. The set $\mathcal{D}$ will be introduced while proving item (ii). Since the set $\mathcal{D}$ we intend to construct will be in $[-2,2]^6$, there will be no need to revisit the proof of item (i) afterwards. Note that all the expansions given below are valid in $L_{\rho_1}^r$ for any $r \geq 2$.

(i) By definition \eqref{def S y} of $S(y)$ and definition \eqref{eq: profile form} of $D$,  we write for $s_0$ large enough,

\begin{equation}
\begin{aligned}\label{S over D estimation}
	\frac{A e^{-2 s_0}|S(y)|}{D} =O(Ae^{-2s_0}(1+|y|)).
	\end{aligned}
\end{equation}

Using \eqref{eq: definition P(y)}, the conclusion follows for $s_0$ large enough.

(ii) Since $D \geq p-1>0$ by definition \eqref{eq: profile form} of $D$, using the expressions \eqref{eq: initial data} and \eqref{eq: profile form} of $w_0\left(y, s_0\right)$ and $\varphi$, together with item (ii) of Lemma \ref{lemma: good properties of phi} and \eqref{S over D estimation}, we write

\begin{equation}\label{initial w to power p-1 bound}
	\begin{aligned}
\left|w_0\left(y, s_0\right)\right|^{p-1} &\leq \chi^{p-1}(\frac{ye^{-s_0/2}+1}{\vep_0})\left(\left|\varphi\left(y, s_0\right)\right|^{p-1}+\frac{p-1}{\kappa D^2}A e^{-2 s_0} S(y)\right),\\
 &\leq \frac{1}{p-1}+C e^{-\frac{s_0}{3}},
    \end{aligned}
\end{equation}
for $s_0$ large enough, and the bound on $\left\|w_0\left(s_0\right)\right\|_{L^{\infty}}$ follows.

Taking the gradient of $w_0\left(y, s_0\right)$ defined in \eqref{eq: initial data}, we write
\begin{equation}
\begin{aligned}
\left|\del_y w_0\left(y, s_0\right)\right| &\leq \del_y \chi\left(\frac{ye^{-s_0/2}+1}{\vep_0}\right)\left(\frac{E}{D}+\frac{p-1}{\kappa D^2}A e^{-2 s_0} S(y)\right)^{\frac{1}{p-1}}\\
&+\chi\left(\frac{ye^{-s_0/2}+1}{\vep_0}\right)\del_y\left(\frac{E}{D}+\frac{p-1}{\kappa D^2}A e^{-2 s_0} S(y)\right)^{\frac{1}{p-1}}\\
&=I+II
\end{aligned}
\end{equation}
We we start from control $I$. Note that $\del_y \chi\left(\frac{ye^{-s_0/2}+1}{\vep_0}\right)= O(e^{-s_0/2})$, from (ii) of Lemma \ref{lemma: good properties of phi}, we have that:
\begin{equation}
	I \leq C e^{-s_0/2}
\end{equation}
Now, let $\hat{w}_0(y,s_0) = \left(\frac{E}{D}+\frac{p-1}{\kappa D^2}A e^{-2 s_0} S(y)\right)^{\frac{1}{p-1}}$. Taking the logarithm then gradient of $\hat{w}_0\left(y, s_0\right)$ gives:
\begin{equation}\label{gradient hat w }
\left|\del_y \hat{w}_0\left(y, s_0\right)\right| \leq \frac{\left|\hat{w_0}\left(y, s_0\right)\right|}{p-1}\left[\frac{|\del_y \bar{E}|}{\bar{E}}+\frac{|\del_y D|}{D}\right]
\end{equation}
where
$$
\bar{E}=E+\frac{p-1}{\kappa D}A e^{-2 s_0} S(y)
$$
Using \eqref{S over D estimation}, \eqref{eq: profile form} and \eqref{eq: estimation P(y)}, we write

\begin{equation}\label{estimation bar E}
\bar{E} \geq \frac{E_0}{C} \text { where } E_0=1+e^{-s_0}y^2+e^{-2 s_0}y^5
\end{equation}
Then, we write

$$
\del_y\left[\frac{S(y)}{D}\right]=\frac{\del_y S(y)}{D}-\frac{S(y) \del_y D}{D^2}
$$
hence, by definition \eqref{def S y} of $S(y)$, we have

$$
A e^{-2 s_0}\left|\del_y\left[\frac{S(y)}{D}\right]\right| \leq C A e^{-2 s_0} \frac{\left(1+|y|^4\right)}{D}+A e^{-2 s_0} \frac{|S(y)|}{D} \frac{|\del_y D|}{D} \leq C A e^{-\frac{2}{3} s_0},
$$
thanks to \eqref{S over D estimation} and \eqref{Gradient D estimation}.

Therefore, using \eqref{gradient hat w }, \eqref{estimation bar E} together with  \eqref{2.22}, we write

\begin{equation}\label{gradient bar E estimation}
\frac{|\del_y \bar{E}|}{\bar{E}} \leq C \frac{|\del_y E|}{E_0}+C A e^{-\frac{2}{3} s_0}\leq C e^{-\frac{s_0}{3}}
\end{equation}
Using \eqref{Gradient D estimation}, \eqref{gradient bar E estimation}, together with a similar strategy with \eqref{initial w to power p-1 bound}, we obtain the following bound

$$
\left|\del_y w_0\left(y, s_0\right)\right| \leq C e^{-\frac{s_0}{6}}
$$
Arguing as for \eqref{profile expansion}, we show the following:

$$
\begin{aligned}
	& w_0\left(y, s_0\right)=\kappa+e^{-s_0}\left(\frac{\kappa}{p-1} P(y)-y^4\right) \\
	& +e^{-2 s_0}\left(A S(y)+\frac{\kappa(2-p)}{2(p-1)^2} P(y)^2-\frac{P(y)}{p-1} y^4+\frac{p}{2 \kappa} y^8\right) \\
	& +O\left(A e^{-3 s_0}\right)
\end{aligned}
$$
uniformly for $d \in[-2,2]^6$. Using again the expansion \eqref{profile expansion} of $\varphi$ together with definition \eqref{eq: w decompose-in-terms-q-phi} of $q\left(y, s_0\right)$, we derive that

$$
q\left(y, s_0\right)=A e^{-2 s_0} S(y)+O\left(A e^{-3 s_0}\right) \text { as } s_0 \rightarrow \infty.
$$
By definition \eqref{def S y} of $S(y)$, together with the definition \eqref{eq: Def of projections P_m} and \eqref{eq: Def of projection P_-} of the projections, we clearly see that for all $d \in[-2,2]^6$ and $i\in I_0 \equiv\{0,1,2,3,4,5\}$,

$$
q_{i}\left(s_0\right)=A e^{-2 s_0} d_{i}+O\left(A e^{-3 s_0}\right)
$$
(please note that this identity holds after differentiation in $d$ ). We also have $\left\|q_{-}\left(s_0\right)\right\|_{L_\rho^2}=$ $O\left(A e^{-3 s_0}\right)$.
 
Recalling estimate \eqref{initial w to power p-1 bound} and the definition \eqref{eq: w decompose-in-terms-q-phi} of $q\left(y, s_0\right)$, we see that this clearly gives the existence of $\mathcal{D} \subset[-2,2]^6$ such that $q\left(s_0\right) \in \mathcal{V}\left(s_0\right)$, \eqref{eq: priori estimate q} and  $\left\|q_{-}\left(s_0\right)\right\|_{L_\rho^2} \leq C A e^{-3 s_0}$ holds, whenever $d \in \mathcal{D}$, with the function \eqref{d one to one function} in  one-to-one, provided that $s_0$ is large enough. This concludes the proof of Proposition \ref{prop: initialization}.\qed

\subsection{Estimates in the regular region}\label{section_regular_region}
Our goal here is to show that:
\begin{equation}\label{eq:estimation on u in regularregion}
|x|\leq \frac{\vep_0}{4}, \text{ then we have } U(x,t^*) \leq \frac{\eta}{2}.
\end{equation}
This is shown in 3 steps:
\begin{itemize}
    \item In the first step, we improve the bounds on the solution $u(x,t)$ in the intermediate region.
    \item In the second step, we use parabolic regularity to obtain an estimation of the solution in the region $\mathcal{R}_2$.
    \item Finally, we use the two steps above to get \eqref{eq:estimation on u in regularregion}
\end{itemize}

\noindent\textbf{Step 1: Improved estimates in the intermediate region}

Here, we refine the estimates on the solution in the following intermediate region:

\begin{equation}\label{def: intermediate region}
\frac{\vep_0}{8}\leq|x|\leq 1- 2K(T-t)^{\frac{1}{4}}.
\end{equation}
By the priori estimate $w_1\leq 2\kappa$, we have
\begin{equation}
    \forall t \in [0, t^*], \text{ and } \forall x\in \Rbb^d,  \left| U(t) \right| \leq C(T - t)^{-\frac{1}{p-1}}, \label{eq:bound intermediate region}
\end{equation}
valid in particular in the intermediate region given by  \eqref{def: intermediate region}. This bound is unsatisfactory since it goes to infinity as $t \to T$. In order to refine it, given a $x$ small enough in norm $\|\cdot\|_{\Rbb^{d}}$ , we use this bound when $t = t_0(x)$ defined by
\begin{equation}
    \left| x\right| = 1 - 2K_0 (T - t_0(x))^{\frac{1}{4}}, \label{eq:t0}
\end{equation}
to see that the solution is, in fact, flat at that time. Then, advancing the PDE \eqref{eq: NLH introduction}, we see that the solution remains flat for later times. More precisely, we claim the following:
\begin{lemma}[Flatness of the Solution in the Intermediate Region in \eqref{def: intermediate region}]\label{Lemma: flateness}
 For all $K_{19} > 0$, $\varepsilon_{19} > 0$, $A_{19} \geq 1$, there exists $s_{19}(K_{19}, \varepsilon_{19}, A)$ such that if $s_0 \geq s_{19}$ and $0 < \eta_0 \leq 1$, then, $\forall t_0(x) \leq t \leq t^*$,
    \begin{equation}
        \left| \frac{U(x, t)}{u^*(1-|x|)} - \frac{U_{K}(x)}{U_{K}(1)} \right| \leq C|1-|x||^4 ,
    \end{equation}
    where $u^*$ is defined in \eqref{eq: profile} and
    \begin{equation}
        U_{K}(\tau) = \kappa \left( (1 - \tau) + \frac{(p-1)K^2}{4p} \right)^{-\frac{1}{p-1}}. \label{eq:UK0}
    \end{equation}
    In particular, $\left| U(x, t) \right| \leq 2 \left| u^*(|x|) \right|$.
\end{lemma}
\begin{proof}
    Proof see in \cite{MNZNon2016} P.316 Lemma 3.12.
\end{proof}

\noindent\textbf{Step 2: A parabolic estimate in regular region:}

Recall from the definition on $\mathcal{V}$, that:
$$\forall x\in \Rbb^d \mbox{ such that } 0\leq|x|\leq \frac{\vep_0}{4}, U(x,t)\leq \eta_0 .$$
Using parabolic estimation on the solution, for $U(x,t)$ in region $\mathcal{R}_2$, we claim the following: 
 \begin{proposition}\label{Prop: parabolic estimate}
	For all \( \varepsilon > 0 \), \( \varepsilon_0 > 0 \), \( \sigma_1 \geq 0 \), there exists \( T \geq 0 \) such that for all \( \overline{t} \leq T \), if \( U \) is a  solution of
 		\[
 		\partial_t U = \Delta U + |U|^{p-1}U \quad \text{for all } x \in [0,\vep_0/4], t \in [0, \overline{t}],
 		\]
 		which satisfies:
 		\begin{enumerate}
 			\item[(i)] For \( |x| \in [\frac{\varepsilon_0}{8}, \frac{\varepsilon_0}{4}] \), \( |U(x, t)| \leq \sigma_1 \).
 			\item[(ii)] For \( 0 \leq |x| \leq \frac{\vep_0}{8} \), \( U(x, 0) = 0 \).
 		\end{enumerate}
 		Then, for all \( t \) in \( [0, \overline{t}] \), for all \(  |x| \leq \frac{\vep_0}{4} \), $\left|U(x, t)\right| \leq \varepsilon$.
 	\end{proposition}
 \begin{proof}
 	Consider $\overline{U}$, recalled here, after a trivial chain rule to transform the $\partial_x U$ term:
 	\[
 	\forall t \in [0, \overline{t}], \quad \forall \theta \in \mathbb{R}, \quad \partial_{t}\overline{U} = \Delta \overline{U} + |\overline{U}|^{p-1} \overline{U} - 2\nabla (\nabla\overline{\chi} U) + \Delta\overline{\chi} U.
 	\]
 	Therefore, since $\overline{U}(x, 0) \equiv 0$, we write
 	\[
 	\|\overline{U}(t)\|_{L^\infty} \leq \int_0^t S(t - t')\left(|U|^{p-1} I_{|x| \leq \frac{\varepsilon_0}{4}} \overline{U} - 2\nabla \left(\nabla\overline{\chi} U I_{|x| \leq \frac{\varepsilon_0}{4}}\right)\right) + \Delta\overline{\chi} U(t') I_{|x| \leq \frac{\varepsilon_0}{4}} \, dt'.
 	\]
 	where $S(t)$ is the heat kernel. Since $\nabla\overline{\chi}$ and $\Delta\overline{\chi}$ are supported by $\left\{\frac{\varepsilon_0}{8} \leq |x| \leq \frac{\varepsilon_0}{4}\right\}$ and satisfy $|\nabla\overline{\chi}| \leq \frac{C}{\varepsilon_0}$, $|\Delta\overline{\chi}| \leq \frac{C}{\varepsilon_0^2}$ and using parabolic regularity, we write
 	\[
 	\|\overline{U}(t)\|_{L^\infty} \leq \sigma_1^{p-1} \int_0^t \|\overline{U}(t')\| \, dt' + C\sigma_1 \frac{1}{\varepsilon_0} \int_0^t\frac{1}{ \sqrt{t - t'}} dt' + C\sigma_1 \frac{1}{\varepsilon_0^2} \int_0^t dt'.
 	\]
 	If $\overline{t} < 1$, by Gronwall's estimate, this implies that
 	\[
 	\|\overline{U}(t)\|_{L^\infty} \leq Ce^{\sigma_1^{p-1}} \left( \frac{\sigma_1}{\varepsilon_0} \sqrt{\overline{t}} +  \frac{\sigma_1}{\varepsilon_0^2} \overline{t}\right).
 	\]
 	Taking $\overline{t}$ small enough, we can obtain $\forall t \in [0,\overline{t}]$, $\|\overline{u}(t)\|_{L^\infty} \leq \varepsilon$.
 	
 	\end{proof}

\textbf{Step 3}:\textbf{Proof of the improvement in Definition \ref{Def: shrinking set}}

Here, we use Step 1 and Step 2 to prove \eqref{eq:estimation on u in regularregion}, for a suitable choice of parameters.

Let us consider $K > 0$, and $\delta_0(K) > 0$ defined in Lemma \ref{Lemma: flateness}. Then, we consider $\varepsilon_0 \leq 2\delta_0$, $0 < \eta_0 \leq 1$ defined in Lemma \ref{Lemma: flateness} and Proposition \ref{Prop: parabolic estimate}; $A \geq 1$, $s_0$ sufficiently large such that Corollary \ref{control of q by finite emts} and Lemma \ref{Lemma: flateness} and Proposition \ref{Prop: parabolic estimate} holds.

Applying Lemma \ref{Lemma: flateness}, we see that for all  $|x| \leq \delta_0$, $A \geq 1$, for all $t \in [0, t^*]$, $|U(x, t)| \leq 2|u^*(|x|)|$.

In particular, for all $\delta_0 \leq \frac{ \varepsilon_0}{8} \leq |x| \leq \frac{\varepsilon_0}{4} $, for all $t \in [0, t^*]$, $|U(x, t)| \leq 2|u^*(\frac{\varepsilon_0}{8})|$.

Using definition of initial data \eqref{eq: initial data}, we see that for all $\leq |x| \leq  \frac{\varepsilon_0}{8} $, $U(x, 0) = 0$.

Therefore Proposition \ref{Prop: parabolic estimate} applies with $\varepsilon = \frac{\eta_0}{2}$ and $\sigma_1 = 2u^*(\frac{\varepsilon_0}{4})$, and we see that for all $|x| \leq \frac{\vep_0}{4}$, for all $t \in [0, t^*]$, $|U(x, t)| \leq \frac{\eta_0}{2}$ and estimate $\eqref{eq:estimation on u in regularregion}$ holds.

\section{Proof of technical results}\label{proof of technical details}
In this section, we will give all the details used in Section \ref{Proof of the priori bound} concerning the proof of the priori $L^{\infty}$ bound. This section is organized as follows:
\begin{itemize}
    \item In Section \ref{Proof of G as norm}, we first prove that $G_1$ defined in \eqref{definition of G1} is a kind of norm  measuring the size of $w_1((|x_0|-1)e^{s_0/2})$.
    \item Then, in Section \ref{Sec: DOCOWITR3}, we give all the proofs of lemmas we used in Section \ref{Control of W_{x_0}(Y,s) in region R_3} to control $W_{x_0}$ in the region $\tilde{\mathcal{R}}_3$ defined in \eqref{region R3}.
\end{itemize}

\subsection{Proof of $G_1(x_0)$ measuring the size of initial data}\label{Proof of G as norm}
As we mentioned before, in this section, we will prove that $G_1(x_0)$ measures the size of initial data in the three regions: $\tilde{\mathcal{R}}_1$, $\tilde{\mathcal{R}}_2$, $\tilde{\mathcal{R}}_3$.

\noindent\textbf{Proof of Lemma \ref{Lemma G_0(a) as norm}.}
Consider $A \geq 1, s_0 \geq s_{5}(A)$ and $d=\left(d_{0},d_1,d_2,d_3,d_4,d_5,d_6\right) \in$ $\mathcal{D}$, where $s_{5}(A)$ and $\mathcal{D}$ are defined in Proposition \ref{prop: initialization}. Consider then $x_0 \in \mathbb{R}^d$ and introduce $y=(|x_0|-1) e^{\frac{s_0}{2}}$. By definition  \eqref{eq: profile form} and  \eqref{eq: initial data} of initial data $w_1\left(y, s_0\right)$ and the profile $\varphi$, arguing as for \eqref{eq: priori estimation on phi}, we may improve that estimate and write
{\small
	\begin{equation} \label{w_1 first expansion}
		\begin{aligned}
			\left|w_1\left(y, s_0\right)^{p-1}-\frac{1}{D}\right| &\leq
				\left|w_1\left(y, s_0\right)^{p-1}-\chi^{p-1}\left(\frac{ye^{-s_0/2}+1}{\vep_0}\right)\frac{1}{D}\right|\\&+\left|\left(1-\chi^{p-1}\left(\frac{ye^{-s_0/2}+1}{\vep_0}\right)\right)\frac{1}{D}\right|\\
			&\leq C\left\{\frac{e^{-s_0}}{D}+\frac{e^{-s_0}|y|^2}{D}+\frac{A e^{-2 s_0} S(y)}{D^2}\right\}
			+O(\frac{e^{-s_0}|y|^2}{D})\\
			& \leq C\left\{e^{-s_0}+\frac{e^{-s_0}|y|^2}{D}+\frac{A e^{-2 s_0} (1+|y|^5)}{D^2}\right\} \\
			&\leq C e^{-\frac{s_0}{3}}\\
		\end{aligned}
	\end{equation}
}
where $D$ and $S(y)$ are defined in \eqref{eq: profile form} and \eqref{eq: initial data}.
and we use the following expansion for $\chi$ defined in \eqref{eq: def Chi}:
\[
\chi^{p-1}\left(\frac{ye^{-s_0/2}+1}{\vep_0}\right)=\chi^{p-1}\left(\frac{1}{\vep_0}\right)+O(y^2e^{-s_0})=1+O(y^2e^{-s_0}).
\]
By definitions \eqref{eq: profile form} and \eqref{definition of G1} of $D$ and $G_1(x_0)$, we see that

$$
D=p-1+\frac{(p-1)^2 e^{-s_0}}{\kappa} y^4 =(p-1)\left[1+e^{s_0} G_0(x_0)\right]
$$

Consider now 2 nonnegative numbers $m$ and $M$ such that $0<m \leq 1 \leq M$. If $x_0 \in R_1$ (resp. $R_2$, resp. $R_3$ ) defined in \eqref{region R1}, \eqref{region R2} and \eqref{region R3}, we see that $(p-1)(1+M) \leq D$ (resp. $(p-1)(1+m) \leq D \leq(p-1)(1+M)$, resp. $D \leq(p-1)(1+m))$. Since $0 \leq w_1\left(y, s_0\right) \leq \kappa+C e^{-\frac{s_0}{3}}$ by Proposition \ref{prop: initialization}, combining this with \eqref{w_1 first expansion} concludes the proof of Lemma \ref{Lemma G_0(a) as norm}.
\qed

\subsection{Details for the control of $W_{x_0}$ in $\tilde{\mathcal{R}}_3$}
\label{Sec: DOCOWITR3}
In this section, we are going to give the proofs for Lemmas \ref{Initialvalue expansion for large x_0}, \ref{lemma decreaing from k to k-eta} and \ref{Ex. W Y sigma}.

\noindent\textbf{Proof of Lemma \ref{Initialvalue expansion for large x_0}}\;\\
(i) Consider \(A \geq 1\), $\vep_0 \in (0,1)$  and \(s_0 \geq s_{10}(A)\), together with the parameter
\[
d = (d_{0}, d_1,d_{2}, d_ {3},d_4,d_5) \in \mathcal{D}(A, s_0),
\]
where \(s_{13}\) and \(\mathcal{D}\) are defined in Proposition \ref{prop: initialization}. Let $x_0 \in R_3$ defined in \eqref{region R3} decomposed as in \eqref{form of x0 in R3} to satisfy $|x_0|> \frac{\vep_0}{16}$ (i.e we have $K \in ((\frac{\vep_0}{8}-1)e^{s_0/2},-A]\cup[A,+\infty)$). Consider also \(w_1(y, s_0)\) defined in \eqref{eq: initial data}. Recalling that \(\mathcal{D} \subset [-2, 2]^6\), we claim the following Taylor expansion:
\begin{claim}\label{the t.e. of ini.data}
\begin{equation}\label{w1 taylor expansion}
\begin{aligned}
&\left|w_1(y, s) - \kappa \left[1 - \frac{X}{p-1}-\frac{e^{-s_0}P}{p-1}\right] \right|\\
&\quad \leq C \left( I + J ^2+ \mathds{1}_{\{1<p<\frac{3}{2}\}} J^\frac{1}{p - 1} \right) 
\end{aligned}
\end{equation}
where
\begin{equation}
X(y, s_0) = \frac{p - 1}{\kappa} e^{-s_0} y^4,
\end{equation}

\begin{equation}
I = X e^{-s_0} |P|+ X^2 \left( 1 + e^{-s_0} |P| \right) + A e^{-2s_0} \left( 1 + |y|^5 \right),
\end{equation}
and
\begin{equation}
J = X + e^{-s_0} |P| + I.
\end{equation}
The polynomials \(P\) is given in \eqref{eq: definition P(y)} of 2 degree.
\end{claim}
\begin{proof}
	See Appendix \ref{proof of the t.e of ini.data}.
\end{proof}
Consider now some \(m \in (0, 1)\) and \(x_0 \in \mathbb{R}^d\) defined in \eqref{region R3}, with \(x_0\) decomposed as in \eqref{eq: form of a in R3}, for some $K \in ((\frac{\vep_0}{16}-1)e^{s_0/2},-A]\cup[A,+\infty)$. Given some \(r \geq 2\), we may use the relation \eqref{eq: relation between W and w} together with  \eqref{w1 taylor expansion} to derive an expansion for \(W_{x_0}(Y, s_0)\), showing error terms bounded by small terms in scales of \(1/ K\) and \(e^{-s_0}\). In particular, the following expansions are useful:
{\footnotesize
\[
\begin{aligned}
|P((|Ye^{-s_0/2}+x_0|-1)e^{s_0/2})| =& \frac{12(p-1)}{\kappa}\Bigl(K^2+Y_1^2+2KY_1-1\\&+\frac{(K+Y_1)(|Y|^2-Y_1^2)}{r_0}e^{-s_0/2}\Bigr)+O(e^{-s_0})\\
\end{aligned}
\]
\[
\begin{aligned}
|X(|Ye^{-s_0/2}+x_0|-1)e^{s_0/2}, s_0)| = \frac{p-1}{\kappa}e^{-s_0}\bigl[K^4&+Y_1^4+6K^2Y_1^2\\&
+4K^3Y_1+4KY_1^3\bigr]+O(e^{-3s_0/2})
\end{aligned}
\]
}
in \(L^r_{\rho_{d}}(\mathbb{R}^d)\). This latter estimate can be easily written in the Hermite polynomials basis $h_1(Y_1)$ defined in \eqref{eq:definition of hm}. Since by definition \eqref{eq: iota} of \(\iota\) and \eqref{iota bound}, it follows that
\[
\iota \leq \frac{m\kappa}{p - 1} \leq \frac{\kappa}{p- 1}, \quad K \leq \iota^{1/6} e^{s_0/3} \quad \text{and} \quad K^2 \leq \sqrt{\iota} e^{s_0/2},
\]
one can easily bound all the error terms by \(O\left(\frac{\iota}{A}\right)\) and \(O(\iota^2)\), as required by the statement of the lemma. 
This concludes the proof of Lemma \ref{Initialvalue expansion for large x_0}. \qed

Now we start to prove  Lemma \ref{lemma decreaing from k to k-eta}.
\paragraph{Proof of  Lemma \ref{lemma decreaing from k to k-eta}}
Take $A \geq 1$ and $s_0 \geq \max \left[s_{5}(A), s_{10}(A)\right]$ where $s_{5}$ and $s_{10}$ are defined in Proposition \ref{prop: initialization} and Lemma \ref{Initialvalue expansion for large x_0}. Consider then $\left(d_0,d_1,d_2,d_3,d_4,d_5,d_6\right) \in \mathcal{D}$ defined in Proposition \ref{prop: initialization} and initial data $w_1\left(y, s_0\right)$ defined in \eqref{eq: initial data}. Consider also some $\eta^* \leq \frac{\kappa}{p-1}, 0<m<\frac{\eta^*(p-1)}{\kappa} \leq 1$ and $x_0 \in \tilde{\mathcal{R}}_3$ from \eqref{region R3} given by \eqref{eq: form of a in R3} for some $K \geq A \geq 0$. In this case, Proposition \ref{prop: initialization} applies, and so does Lemma \ref{Initialvalue expansion for large x_0}. In particular, the expansion given there holds for $W_{x_0}\left(Y, s_0\right)$.
\begin{itemize}
\item[(i)] Note that condition \eqref{eq: definition of s etoile} holds from the choice of $\eta^*$ and $m$. Therefore, using \eqref{eq: iota}, we see that $\iota \leq \frac{m \kappa}{p-1} \leq \eta^*$. By definition \eqref{eq: definition of s etoile} of $s^*$, it follows that $s^*=$ $s_0+\log \frac{\eta^*}{\iota} \geq s_0$.
\item[(ii)] Assume now that

\begin{equation}
\forall s \in\left[s_0, s_1\right],\left\|W_{x_0}(s)\right\|_{L^{\infty}} \leq 2 \kappa
\end{equation}

for some $s_1 \geq s_0$. Introducing

\begin{equation}\label{V definition}
V_{x_0}=W_{x_0}-\kappa \text { and } \bar{s}=\min \left(s^*, s_1\right)
\end{equation}
\end{itemize}

we work in the following in the interval $\left[s_0, \bar{s}\right]$, and proceed in 3 steps in order to give the proof:
\begin{itemize}
\item  In Step 1, we write an equation satisfied by $V_{x_0}$ and project it on the various eigen-function $h_{j_1}\cdots h_{j_d}$ where $j=(j_1,\cdots,j_d) \in \mathbb{N}^d$.
\item In Step 2, we integrate those equations.
\item In Step 3, we collect the previous information to conclude the proof.
\end{itemize}
\paragraph{Step 1: Dynamics for $V_{x_0}$}
Given that $W_{x_0}$ satisfies the equation presented in \eqref{eq: W's equation}, by its defining property \eqref{V definition}, it immediately implies that $V_{x_0}$ fulfills the equation:

\begin{equation}\label{eq: v's equation}
	\forall s \in \left[s_0, \bar{s}\right], \quad \partial_s V_{x_0}=\mathcal{L}_d V_{x_0}+\bar{B}\left(V_{x_0}\right),
\end{equation}
in which the linear operator $\mathcal{L}_d$ is defined as
\begin{equation}
\mathcal{L}_d= \Delta_d + \frac{1}{2}Y\cdot\nabla_d+1.
\end{equation}
Its spectrum properties is as following: let $j=(j_1,\cdots,j_d)$ be a multi-indice, then
\[
\mathcal{L}_d(h_{j_1}h_{j_2}\cdots h_{j_d})=\left(1-\frac{|j|}{2}\right)h_{j_1}h_{j_2}\cdots h_{j_d}
\]
and the nonlinear term is expressed as:
\[
\bar{B}\left(V_{x_0}\right)=\left|\kappa+V_{x_0}\right|^{p-1}\left(\kappa+V_{x_0}\right)-\kappa^p - p\kappa^{p-1}V_{x_0}.
\]

At this stage, by projecting equation \eqref{eq: v's equation}, we derive differential inequalities that the components $V_{x_0,j}$, the projection on $h_{j_1}\cdots h_{j_d}$, and the projection $\bar{P}(V_{x_0})$ must satisfy. Here, $\bar{P}$ denotes the orthogonal projector in the $L_{\rho_d}^2$ sense onto the subspace

\begin{equation}\label{bar E}
\bar{E} = \operatorname{span}\left\{h_{j_1}\cdots h_{j_d} \mid j_1 \notin \{0, \dots, 4\}, j_2, \dots, j_d \neq 0\right\},
\end{equation}
which represents the orthogonal complement of the directions explicitly identified in Lemma \ref{w1 taylor expansion}'s expansion. This is our statement:
\begin{lemma}[Projections of equation \eqref{eq: v's equation}.]\label{Proj. of v's eq} Under the hypotheses of Lemma \ref{lemma decreaing from k to k-eta}, for all $s \in\left[s_0, \bar{s}\right]$, for all $j =(j_1,\cdots,j_d)\in \mathbb{N}^d$, with $j_1 \in \{0,\cdots,4\}$ and $j_2=j_3=\cdots=j_d=0$, it holds that

\begin{equation}\label{ode of vj}
\left|V_{x_0, j}^{\prime}(s)-\left(1-\frac{|j|}{2}\right) V_{x_0, j}(s)\right| \leq C(i)\left\|v_a(s)\right\|_{L_\rho^2}^2
\end{equation}
In addition,
\begin{equation}\label{ode of ortho compl v}
\begin{aligned}
	\frac{d}{d s}\left\|\bar{P}\left(V_{x_0}(s)\right)\right\|_{L_\rho^2} \leq & -\frac{1}{2}\left\|\bar{P}\left(V_{x_0}(s)\right)\right\|_{L_\rho^2} \\
	& +\mathbbm{1}_{\left\{s_0 \leq s \leq s_0+2\right\}} C_{2}\left\|V_{x_0}\left(s_0\right)\right\|_{L_\rho^4}^2\\&+\mathbbm{1}_{\left\{s_0+2 \leq s \leq \bar{s}\right\}} C_{2}\left\|V_{x_0}(s-2)\right\|_{L_\rho^2}^2
\end{aligned}
\end{equation}
for some universal constant $C_{2}$, where $\bar{s}$ is defined in \eqref{V definition}and the projector $\bar{P}$ is the projector on $\bar E$ defined in \eqref{bar E} as we have mentioned before .
\end{lemma}
\begin{proof}
	Remark that the eigen value on $\bar{E}$ is $-\frac{1}{2}$. The proof of this Lemma is actually analogue to the one of Lemma 7.3 in \cite{MZJEMS24}. Hence, we omit it here and intereted reader's are invited to see it in \cite{MZJEMS24} 
\end{proof}

\paragraph{Step 2: Integration of the equations in Lemma \ref{Proj. of v's eq}.}\;\\
Now, with a similar argument as in $\cite{MZJEMS24}$, we obtain the following result.
\begin{lemma}[Integration of equations in Lemma \ref{Proj. of v's eq}.] \label{integration of v equation}There exist $M_{29}, A_{29} \geq 1$ and $\eta_{29}>0$ such that under the hypotheses of Lemma \ref{lemma decreaing from k to k-eta}, if $A \geq A_{29}$ and $\eta^* \leq \eta_{29}$, then for all $s \in\left[s_0, \bar{s}\right]$, it holds that 
	\begin{itemize}
	\item for all $j =(j_1,\cdots,j_d)\in \mathbb{N}^d$, with $j_1 \in \{0,\cdots,4\}$ and $j_2=\cdots=j_d=0$
\begin{equation}
	 \left|V_{x_0, j}(s)-e^{\left(1-\frac{i}{2}\right) \tau} V_{x_0, j}\left(s_0\right)\right| \leq M_{2}\left(\eta^*+A^{-2}\right) e^\tau \iota,
\end{equation}
\item 
\begin{equation}
	\left\|\bar{P}\left(V_{x_0}(s)\right)\right\|_{L_\rho^2} \leq e^{-\frac{\tau}{2}}\left\|\bar{P}\left(V_{x_0}\left(s_0\right)\right)\right\|_{L_\rho^2}+M_{2}\left(\eta^*+A^{-2}\right) e^\tau \iota,
\end{equation}

\end{itemize}
provided that $A$ and $s_0$ are large enough and $\eta^*$ is small enough, where $\bar{s}, \iota$ and $\bar{P}$ are defined in \eqref{V definition}, \eqref{eq: iota} and right before \eqref{bar E}, with $\tau=s-s_0$.
\end{lemma}
\begin{proof}
	See Lemma 7.4 in \cite{MZJEMS24}
\end{proof} 
\paragraph{Step 3: Conclusion of the proof of Lemma \ref{lemma decreaing from k to k-eta}.}\;\\
Recalling the definition \eqref{V definition}, then using Lemma \ref{integration of v equation} together with Lemma  \ref{Initialvalue expansion for large x_0} and the Cauchy-Schwarz inequality, we see that for all $s \in\left[s_0, \bar{s}\right]$,

$$
\left\|W_{x_0}(\cdot, s)-\left(\kappa-e^\tau \iota\right)\right\|_{L_\rho^2} \leq M_{2}^{\prime \prime}\left\{M_{2}\left(\eta^*+A^{-2}\right) e^\tau \iota+e^\tau\left(\frac{\iota}{A}+\iota^2\right)+2K^2 e^{-s_0}\right\},
$$
for some universal constant $M_{2}^{\prime \prime}>0$. Using the bounds \eqref{iota bound} and \eqref{J bound}, then we see that

$$
e^{-s_0}K^2 \leq 2 e^{-\frac{s_0}{3}}\left(\frac{\kappa}{p-1}\right)^{\frac{1}{3}},
$$
This concludes the proof of Lemma \ref{lemma decreaing from k to k-eta}.\qed

To finish this section, we will give the proof of Lemma  \ref{Ex. W Y sigma}.
\paragraph{Proof of Lemma  \ref{Ex. W Y sigma}}
By definition \eqref{eq: second group ssv}, relation \eqref{eq: w decompose-in-terms-q-phi} and \eqref{eq: relation between W and w}, we shall write $W_{x_0}$ in three parts:
\[W_{x_0}(Y,\sigma)=\Phi_{x_0}(Y, \sigma)+Q_{x_0}(Y, \sigma)+(1-\chi)W_{x_0}(Y,\sigma),\]
where
\[
\Phi_{x_0}(Y, \sigma)= \varphi\left((|x_0+Ye^{-\frac{\sigma}{2}}|-1)e^{\frac{\sigma}{2}}\sigma\right)
\]
and
\[
Q_{x_0}=q\left((|x_0+Ye^{-\frac{\sigma}{2}}|-1)e^{\frac{\sigma}{2}}\sigma\right).
\]
 
Now, the proof of Lemma \ref{Ex. W Y sigma} follows by adding the expansions of the following three functions performed in three steps:
\begin{enumerate}
\item
 $\left(1-\chi(\frac{|Ye^{-\sigma/2}+x_0|}{\vep_0})\right)W_{x_0}(Y,\sigma)$
\item
$\Phi_{x_0}(Y,\sigma)$
 \item $ Q_{x_0}(Y,\sigma)$
\end{enumerate}
\paragraph{Step 1, the control of $(1-\chi)W(Y,\sigma)$}\;

Note by \eqref{eq: relation between W and w}, we have 
\[
W_{x_0}(Y,\sigma) = w_{1}((|x_0+Ye^{-\frac{\sigma}{2}}|-1)e^{\frac{\sigma}{2}},\sigma).
\]
 Since under the assumption of Lemma \ref{Ex. W Y sigma}, we have $\|w_1(\cdot,\sigma)\|_{L^{\infty}(\Rbb)}\leq 2\kappa$, proceeding as in Lemma \ref{Initialvalue expansion for large x_0}, we hence have
\[
\begin{aligned}
	&\left\|W_{x_0}(Y,\sigma)(1-\chi(\frac{|Ye^{-\sigma/2}+x_0|}{\vep_0}))\right\|_{L^{r'}_{\rho_d}}\leq 2\kappa \left\|1-\chi(\frac{|Ye^{-\sigma/2}+x_0|}{\vep_0})\right\|_{L^{r'}_{\rho_d}} \\
\end{aligned}
\]
By definition \eqref{eq: def Chi}, it can be seen that:
\begin{equation*}
	\chi\left(\frac{|Ye^{-\sigma/2}+x_0|}{\vep_0}\right)=\chi\left(\frac{|x_0|}{\vep_0}\right)+O(\vep_0^{-4}|Y|^4e^{-2\sigma}).
	\end{equation*}
Since $\chi\left(\frac{|x_0|}{\vep_0}\right)=1$ by definition \eqref{eq: def Chi}, whence
\begin{equation*}
	1-\chi\left(\frac{|Ye^{-\sigma/2}+x_0|}{\vep_0}\right)=O(\vep^{-4}_0|Y|^4e^{-2\sigma}),
\end{equation*}
and
\begin{equation}\label{control of 1- chi W}
	\left\|1-\chi\left(\frac{|Ye^{-\sigma/2}+x_0|}{\vep_0}\right)W_{x_0}(Y,\sigma)\right\|_{L^{r'}_{\rho_d}}=O(\vep_0^{-4}e^{-2\sigma})=O\left(\frac{\iota'}{A}\right).
\end{equation}

\paragraph{Step 2, the expansion of $\Phi_{x_0}$:} \;

We claim that the expansion of $\varphi$ follows from Lemma \ref{Initialvalue expansion for large x_0}. Indeed, the input in that lemma is initial data $w_1\left(y, s_0\right)$ given in \ref{eq: initial data}, and if one takes the parameter $\left(d_0,d_2,d_4\right)=(0,0,0)$ and formally replaces $s_0$ by $\sigma$ in the definition \eqref{eq: initial data} of initial data $w_1\left(y, s_0\right)$, then, we recover $\varphi(y, \sigma)$ defined in \eqref{eq: profile form}. In addition, the point $x_0$ we consider is given by \eqref{ G sigma definition} with $K^{\prime}=\pm A$, which falls in the framework considered in Section \ref{Case large K}. Therefore, Lemma \ref{Initialvalue expansion for large x_0} applies and we see that
\begin{equation}\label{Exp. Phi x0}
\begin{aligned}
	\Phi_{x_0}(Y,\sigma)&=\varphi((|Ye^{-\sigma/2}+x_0|-1)e^{\sigma/2},s_0) = \kappa-\iota -e^{-\sigma}\Bigl[h_4(Y_1)+6K^2h_2(Y_2)\\&
	+4K^3h_1(Y_1)+4Kh_3(Y_1)\Bigr]+O\left(\frac{\iota'}{A}\right)+O(\iota'^2),
\end{aligned}
\end{equation}
in $L_{\rho_d}^r$ for any $r \geq 2$, where $\iota^{\prime}$ is given in \eqref{iota prime defintion}.

\paragraph{Step 3: The expansion of $q$}\;

Take $r \geq 2$. Using the decomposition \eqref{decomposition of q} for $q(Y, \sigma)$, we write  :
$$
Q(y, \sigma)=\bar{Q}(Y, \sigma)+\underline{Q}(Y, \sigma)
$$
where

\begin{equation}\label{q oveline}
\bar{Q}(Y, \sigma)=\sum_{i=0}^7 q_{i}(\sigma) h_{i}\left((|Ye^{-\sigma/2}+x_0|-1)e^{\sigma/2}\right),
\end{equation}
\begin{equation}\label{q underline}
 \underline{Q}(Y, \sigma)=q_{-}\left((|Ye^{-\sigma/2}+x_0|-1)e^{\sigma/2}, \sigma\right)
\end{equation}
Concerning $\bar{q}(y, \sigma)$, recalling that $q(\sigma) \in V_A(\sigma)$ defined in Definition \ref{Def: shrinking set}, then proceeding as for the proof of Lemma \ref{Initialvalue expansion for large x_0} given at the beginning of this subsection, we derive that

\begin{equation}\label{Exp. Q overline}
\begin{aligned}
	\bar{Q}(y, \sigma)=&q_{6}(\sigma)  \bigl\{h_6(Y_1)+5K'h_5(Y_1)+15K'^2h_4(Y_1)+20K'^3h_3(Y_1) \\
	&+15K'^4h_2(Y_1)+5h_1(Y_1)K'^5+K'^6h_0(Y_1) \bigr\}+O\left(\frac{\iota^{\prime}}{A}\right)
\end{aligned}
\end{equation}
in $L_{\rho_d}^{r'}$. As for $\underline{q}(y, \sigma)$, we can bound it thanks to the following parabolic regularity estimate on $q_{-}(\sigma)$ :

\begin{lemma}[Parabolic regularity for $q_{-}(\sigma)$] \label{Parabolic regularity for q - sigma}Under the hypotheses of Lemma \ref{Ex. W Y sigma}, it holds that for all $r^{\prime} \geq 2$,

$$
\forall s \in\left[s_0, \sigma\right], \quad\left\|q_{-}(s)\right\|_{L_{\rho_1}^{r^{\prime}}} \leq C\left(r^{\prime}\right) A^2 s^2 e^{-3 s}
$$
\end{lemma}
\noindent Indeed, using \eqref{q underline}, we write
$$
\begin{aligned}
	\|\underline{Q}\|_{L_{\rho_d}^r}^r&=\int\left|q_{-}((|Ye^{-\sigma/2}+x_0|-1)e^{\sigma/2}, \sigma)\right|^r \rho_d(Y) d Y.\\
	\begin{comment}
	&=\int_{|Y|\leq  e^{\sigma/d}}\left|q_{-}((|Ye^{-\sigma/2}+x_0|-1)e^{\sigma/2}, \sigma)\right|^r \rho_d(Y) d Y\\
	&+\int_{|Y|\geq e^{\sigma/d}}\left|q_{-}((|Ye^{-\sigma/2}+x_0|-1)e^{\sigma/2}, \sigma)\right|^r \rho_d(Y) d Y\\
	&= I +II.
	\end{comment}
\end{aligned}
$$

Like in \cite{MZJEMS24}, we wish to introduce the following change of variable:
\begin{equation}\label{change of variable control Q underline}
\begin{aligned}
	&Z_1=(|Ye^{-\sigma/2}+x_0|-1)e^{\sigma/2},\\
	&Z_i=Y_i,\; \text{for}\; i=1,\cdots,d.
\end{aligned}
\end{equation}
However, this is not differentiable when $Y=-x_0e^{\sigma/2}$. In order to deal with this problem, we introduce $r_1\leq\frac{1}{2}|x_0|e^{\sigma/2}<|x_0|e^{\sigma/2}$, such that when $|Y|\leq r_1$,  we have $|Ye^{-\sigma/2}+x_0|\geq \frac{1}{2}|x_0|> 0$ and \eqref{change of variable control Q underline} is diffrentiable hence well defined. For $|Y|\geq r_1$, we will take advantage of the fact that $|q_- |$ is bounded in $L_{\rho_1}^{r}$. 

Under this idea, let us write the following: 
\begin{equation}
	\begin{aligned}
		\|\underline{Q}\|_{L_{\rho_d}^{r'}}^{r'}&=\int_{|Y|\leq  r_1}\left|q_{-}((|Ye^{-\sigma/2}+x_0|-1)e^{\sigma/2}, \sigma)\right|^r \rho_d(Y) d Y\\
		&+\int_{|Y|\geq r_1}\left|q_{-}((|Ye^{-\sigma/2}+x_0|-1)e^{\sigma/2}, \sigma)\right|^r \rho_d(Y) d Y\\
		&= I +II.
		\end{aligned}
	\end{equation}
Now we estimate $I$ and $II$ to obtain control on $\|\underline{Q}\|_{L_{\rho_d}^{r'}}^{r'} $. We claim the following:
\begin{claim}\label{Claim Q underline estimation}Under the assumption of Lemma \eqref{Ex. W Y sigma}, for all $r_1\leq \frac{1}{2}|x_0|e^{\sigma/2}$, we have 
\begin{equation}
		\|\underline{Q}\|_{L_{\rho_d}^{r'}}^{r'} \leq 3C(d,K')r_1^{d-1}A^{2}\sigma^2e^{-3\sigma}+2C\bar{C}e^{-\frac{r_1^2}{8}}.
\end{equation}
\end{claim}
\noindent\textit{proof}.
\paragraph{Estimation for $I$.}\;\\
 We introduce the following change of variable: 
\begin{equation}
	\begin{aligned}
		&Z_1=(|Ye^{-\sigma/2}+x_0|-1)e^{\sigma/2},\\
		&Z_i=Y_i,\; \text{for}\; i=1,\cdots,d.
	\end{aligned}
\end{equation}
Since 
$$(Z_1+e^{\sigma/2})^2 = |Ye^{-\sigma/2}+x_0|^2e^{\sigma}\geq\frac{1}{4}|x_0|^2e^{\sigma}\geq |Y|^2\geq \sum_{i=2}^dY_i^2=\sum_{i=2}^dZ_i^2,$$
we can express $Y_1$ as follows:
\[
	\begin{aligned}
		&Y_1=\pm\sqrt{(Z_1+e^{\sigma/2})^2-\sum_{i=2}^dZ_i^2}-|x_0|e^{\sigma/2}.\\
	\end{aligned}
\]
Remark that 
\[
\begin{aligned}
	&\left|-\sqrt{(Z_1+e^{\sigma/2})^2-\sum_{i=2}^dZ_i^2}-|x_0|e^{\sigma/2}\right|\geq |x_0|e^{\sigma/2}>r_1,\\
\end{aligned}
\]
we are hence making the choice with 
\begin{equation}
	\begin{aligned}
		&Y_1=\sqrt{(Z_1+e^{\sigma/2})^2-\sum_{i=2}^dZ_i^2}-|x_0|e^{\sigma/2}.\\
	\end{aligned}
\end{equation}
This is differentiable when $|Y|\leq r_1<\frac{1}{2}|x_0|e^{\sigma/2}$. We hence can calculate its Jacobian as follows:
\[
	\begin{aligned}
	\left|\det \frac{\partial (Y_1,\cdots,Y_d)}{\partial(Z_1,\cdots,Z_d)}\right|&= \left|\det\frac{\partial (Z_1,\cdots,Z_d)}{\partial(Y_1,\cdots,Y_d)}\right|^{-1}\\
	&=\left|\frac{\del Z_1}{\del Y_1}\right|^{-1}\\
	&=\frac{|Ye^{-\sigma/2}+x_0|}{Y_1e^{-\sigma/2}+|x_0|}\\
	\end{aligned}
\]
Since $|Y|\leq r_1\leq\frac{1}{2}|x_0|e^{\sigma/2}$, it follows that
\begin{equation}\label{jacobian estimate}
		\left|\det \frac{\partial (Y_1,\cdots,Y_d)}{\partial(Z_1,\cdots,Z_d)}\right|\leq\frac{r_1e^{-\sigma/2}+|x_0|}{-r_1e^{-\sigma/2}+|x_0|}\leq 3.
\end{equation}
Then, by definition \eqref{eq: def of rho}, we have
\small
{
\begin{equation}\label{rhod Y to rho1 Z}
	\begin{aligned}
	&(4\pi)^{d/2}\rho_d(Y)\\
	&=\exp\left\{-\frac{1}{4}(Z_1+e^{\sigma/2})^2+\frac{1}{4}\sum_{i=2}^dZ_i^2+\frac{1}{2}|x_0|e^{\sigma/2}\sqrt{(Z_1+e^{\sigma/2})^2-\sum_{i=2}^dZ_i^2}-\frac{1}{4}|x_0|^2e^{\sigma}\right\}\\&
	\exp\left\{-\frac{1}{4}\sum_{i=2}^dZ_i^2\right\}\\
	&=\exp\left\{-\frac{1}{4}(Z_1+e^{\sigma/2})^2+\frac{1}{2}|x_0|e^{\sigma/2}\sqrt{(Z_1+e^{\sigma/2})^2-\sum_{i=2}^dZ_i^2}-\frac{1}{4}|x_0|^2e^{\sigma}\right\}\\
	&\leq\exp\left\{-\frac{1}{4}(Z_1+e^{\sigma/2})^2+\frac{1}{2}|x_0|e^{\sigma/2}(Z_1+e^{\sigma/2})-\frac{1}{4}|x_0|^2e^{\sigma}\right\}\\
	&=\exp\left\{-\frac{1}{4}(Z_1+e^{\sigma/2}-|x_0|e^{\sigma/2})^2\right\}\\
	&=\exp\left\{-\frac{1}{4}(Z_1+(1-|x_0|)e^{\sigma/2})^2\right\}\\
	&=\exp\left\{-\frac{1}{4}(Z_1-K')^2\right\}\\
	&=(4\pi)^{\frac{1}{2}}\rho_1(Z_1)e^{\frac{K'Z_1}{2}}e^{\frac{K'^{2}}{4}}.
	\end{aligned}
	\end{equation} 
}
Take \eqref{change of variable control Q underline} in mind, using Fubini's theorem, equations \eqref{jacobian estimate}, \eqref{rhod Y to rho1 Z}, and \eqref{eq: def of rho}, we remark that

$$
\begin{aligned}
	I&=\int_{|Y|\leq  r_1}\left|q_{-}((|Ye^{-\sigma/2}+x_0|-1)e^{\sigma/2}, \sigma)\right|^r \rho_d(Y) d Y,\\
    &\leq (4\pi)^{\frac{d-1}{2}}\int_{Z_1\in \Rbb}\left|q_{-}(Z_1,\sigma)\right|^r  \rho_1(Z_1)e^{\frac{K'Z_1}{2}}e^{\frac{K'^{2}}{4}}d Z_1\\&\;\;\;\;\int_{\sum_{i=2}^dZ_i^2\leq r_1^2}\left|\det\frac{\partial (Y_1,\cdots,Y_d)}{\partial(Z_1,\cdots,Z_d)}\right|dZ_2 \cdots dZ_d,\\
    &\leq 3C(d)(4\pi)^{\frac{d-1}{2}}r_{1}^{d-1}e^{\frac{K'^{2}}{4}}\int_{Z_1\in \Rbb}\left|q_{-}(Z_1,\sigma)\right|^r  \rho_1(Z_1)e^{\frac{K'Z_1}{2}}d Z_1,\\
    &\leq 3C(d)(4\pi)^{\frac{d-1}{2}}r_{1}^{d-1}e^{\frac{K'^{2}}{4}}\int_{Z_1\in \Rbb}\left|q_{-}(Z_1,\sigma)\right|^{2r}  \rho_1(Z_1)d Z_1 \int_{Z_1\in\Rbb}e^{K'Z_1}\rho(Z_1)dZ_1,\\
    &\leq 3C(d)(4\pi)^{\frac{d-1}{2}}r_{1}^{d-1}e^{\frac{3K'^{2}}{4}}(A^{2}\sigma^2e^{-3\sigma}).
\end{aligned}
$$
It then follows that 
\begin{equation}\label{Q underline I estimation}
I \leq 3C(d,K')r_1^{d-1}A^{2}\sigma^2e^{-3\sigma}
\end{equation}

\paragraph{Estimation for $II$.}
By definition \eqref{decomposition of q}, we write:
\[
\left\|\frac{q_-(\cdot,\sigma)}{1+|y|^{6}}\right\|_{L^{\infty}}\leq \|q\|_{L^{\infty}}+C\sum_{i=0}^{6}|q_i(s)|\leq 2\kappa+C\sigma^2e^{-2\sigma}.
\]
Therefore, $II$ satisfies:
\begin{equation}\label{Q underline II estimation}
\begin{aligned}
II &\leq C\int_{|Y|\geq r_1} (2\kappa+C\sigma^{2}e^{-2\sigma})^r\left[1+(|Y|+K')^6\right]e^{-\frac{|Y|^2}{4}}dY\\
&\leq Ce^{-\frac{r_1^2}{8}}\int_{|Y|\geq r_1} (2\kappa+C\sigma^{2}e^{-2\sigma})^r\left[1+(|Y|+K')^6\right]e^{-\frac{|Y|^2}{8}}dY\\
&\leq C\tilde{C}e^{-\frac{r_1}{8}}
\end{aligned}
\end{equation}
Now, by summing \eqref{Q underline I estimation} and \eqref{Q underline II estimation}, we conclude that:
\begin{equation}
\|\underline{Q}\|_{L_{\rho_d}^r}^r \leq 3C(d,K')r_1^{d-1}A^{2}\sigma^2e^{-3\sigma}+2C\bar{C}e^{-\frac{r_1^2}{8}}.
\end{equation}
\paragraph{Conclusion of the proof of Lemma \ref{Ex. W Y sigma}}\;\\
 When $s_0$ sufficiently large, we have $|x_0|e^{\sigma/2}\geq \sigma \geq s_0 $ choosing $r_1= \sigma$ so that Claim \ref{Claim Q underline estimation} is avalable. We then deduce from Claim \ref{Claim Q underline estimation} that 
 \begin{equation}\label{Q underline O iota prime}
 	\|\underline{Q}\|_{L_{\rho_d}^r}^r \leq 3C(d,K')\sigma^{d+1}A^{2}e^{-3\sigma}+2C\bar{C}e^{-\frac{\sigma^2}{8}} = O\left(\frac{\iota'}{A}\right)
 	\end{equation}
Now we can conclude the proof of Lemma \ref{Ex. W Y sigma} by summing up equations \eqref{control of 1- chi W}, \eqref{Exp. Phi x0}, \eqref{Exp. Q overline} and \eqref{Q underline O iota prime}. \qed\\
The only thing left for this paper is to prove Lemma \ref{Parabolic regularity for q - sigma}.
\paragraph{Proof of Lemma \ref{Parabolic regularity for q - sigma}}. 
Assume that $s_0$ sufficiently large so that Poposition \ref{delay regularizing estimate} applies. Consider some $r'\geq 2$ we project equation \eqref{q's equation} for all $s\in [s_0,\sigma]$ as follows:
\begin{equation}
\del_s q_-=\mathcal{L} q_-+\hat{Q},
\end{equation}
where
\[
\hat{Q}=P_-(Vq+H(y,s)+\del_y G(y,s)+R(y,s)+B(y,s)+N(y,s)),
\]
and $P_-$ is defined in \eqref{eq: Def of projection P_-}. Given some $s'\in[s_0,\sigma]$
\begin{small}
	\begin{equation}\label{q- solution}
		\begin{aligned}
			q_-(\sigma)&=e^{(\sigma-s')\mathcal{L}}q_{-}(\tau)
		+\int_{s'}^\sigma e^{(\sigma'-s')\mathcal{L}} \hat{Q} (y,\sigma')d\sigma'.
		\end{aligned}
	\end{equation}
\end{small}
We then argue as in  \cite{MZJEMS24}; Taking $L_{\rho_1}^{r'}$ norm on both side of the equation \eqref{q- solution} gives:  
\[
\begin{aligned}
	\left\|q_{-}(s)\right\|_{L_\rho^{r'}(\Rbb)} 
	\;\le\; I + II 
	\;\equiv\;& 
	\left\| e^{\mathcal{L}(\sigma-s')} q_-(s') \right\|_{L_{\rho_1}^\infty(\Rbb)}
\\&\;+\;
	\int_{s'}^{\sigma} 
	\left\| e^{(s-s')\mathcal{L}} \hat{Q}(y,\sigma') \right\|_{L^{r'}_{\rho_1}(\Rbb)} 
	\, d\sigma'.
\end{aligned}
\]
We start by bounding $I I$. We claim that for any $\sigma^{\prime} \in\left[s', \sigma\right], Q\left(\sigma^{\prime}\right)$ and its gradient have polynomial growth in $y$, allowing the application of item (ii) of Lemma \ref{lemma: Reg eff d oper}. Indeed, by definition \eqref{definition of B,R,F,N} , together with the $L^{\infty}$ bound on $q(s)$ from Definition \ref{Def: shrinking set} and the gradient estimate of Proposition \ref{gradient estimation},  together with \eqref{w1 to We1} we see that $V q+B+R+H+\del_yG+N$ and its gradient have polynomial growth in $y$. By definition of the $P_{-}$ operator (see \eqref{eq: Def of projection P_-}), so has $Q$. Applying item (ii) of Lemma \ref{lemma: Reg eff d oper}, we see that
\begin{equation}
	|II|\leq \int_{s'}^\sigma\bigl\| Q(\sigma') \bigr\|_{L^{r'}_{\rho}(\Rbb)}d\sigma'
\end{equation}
Upon using Proposition \ref{delay regularizing estimate}, Lemma \ref{Lemma:esti:R}, Lemma \ref{sixth term}, Lemma \ref{seventh term} and Lemma \ref{Eighth term} we obtain that
\begin{equation}\label{eq: II}
	\begin{aligned}
		\bigl\| Q(\sigma') \bigr\|_{L^{r'}_{\rho}(\Rbb)}\leq C(r')A\sigma'e^{-3\sigma'}.
	\end{aligned}
\end{equation}
Then as for the bound on  $I$, introducing $\sigma^*\left(r^{\prime}\right)>0$ and $C^*\left(r^{\prime}\right)>0$ such that the following delay regularizing effect holds for any $v \in L_{\rho_1}^2$ (see item (i) of Lemma \ref{lemma: Reg eff d oper}):

\begin{equation}\label{regularizing effect for q-}
\left\|e^{\sigma^*\left(r^{\prime}\right) \mathcal{L}}(v)\right\|_{L_{\rho_1}^{r^{\prime}}} \leq C^*\left(r^{\prime}\right)\|v\|_{L_{\rho_1}^2}
\end{equation}
we distinguish two cases in the following:
\paragraph{Case 1:} $\sigma \geq s_0+\sigma^*$. Fixing $s'=\sigma-\sigma^*$, we see that $s' \geq s_0$. Using \eqref{regularizing effect for q-}, we see by Definition \eqref{Def: shrinking set} of $V_A(s)$ that

$$
|I| \leq C^*\left(r^{\prime}\right)\left\|q_{-}\left(\sigma-\sigma^*\right)\right\|_{L_\rho^2} \leq C^*\left(r^{\prime}\right) A^2\left(\sigma-\sigma^*\right)^2 e^{-3\left(\sigma-\sigma^*\right)}
$$
\paragraph{Case 2:} $s_0 \leq \sigma \leq s_0+\sigma^*$. Fixing $s'=s_0$, and noting that $q\left(s_0\right)$ and $\nabla q\left(s_0\right)$ are bounded (see the hypotheses of Lemma \ref{Ex. W Y sigma} and Definition \ref{Def: shrinking set} of $V_A\left(s_0\right)$ ), we can apply item (ii) of Lemma \ref{lemma: Reg eff d oper} and write by Definition \ref{Def: shrinking set}:

$$
|I| \leq C e^{\sigma-s_0}\left\|q_{-}\left(s_0\right)\right\|_{L_\rho^2} \leq C e^{\sigma-s_0} A^2 s_0^2 e^{-3 s_0} \leq C e^{\sigma^*} A^2 \sigma^2 e^{-3\left(\sigma-\sigma^*\right)}
$$
By summing $|I|$ and $|II|$ we can conclude the proof of Lemma \ref{Parabolic regularity for q - sigma}\qed

\begin{appendix}
\section{Proof of the taylor expansion of the initial data}\label{proof of the t.e of ini.data}
In this appendix we give the proof of Claim \ref{the t.e. of ini.data}.
\paragraph{Proof of Claim \ref{the t.e. of ini.data}}\;

As defined in \eqref{eq: initial data}, for some small parameter $\vep_0 \in (0,1)$ and a sufficiently large constant $A$, the initial data $w_1(y,s_0)$ is given by
\[
w_1(y,s_0)=\chi\left(\frac{1+ye^{-s_0/2}}{\vep_0}\right)\left\{ \frac{E}{D}+\frac{p-1}{\kappa D^2}Ae^{-2s_0}S(y)\right\}^{\frac{1}{p-1}},
\]
where $E$, $D$, and $S(y)$ are defined in \eqref{def S y} and \eqref{eq: profile form}. Extracting the factor $\frac{E}{D}$ from the expression yields:
\[
w_1(y,s_0)=\chi\left(\frac{1+ye^{-s_0/2}}{\vep_0}\right)\varphi\left\{ 1+\frac{p-1}{\kappa DE}Ae^{-2s_0}S(y)\right\}^{\frac{1}{p-1}},
\]
where $\varphi = \left(\frac{E}{D}\right)^{\frac{1}{p-1}}$.

To conclude the proof of Claim \ref{the t.e. of ini.data}, we use the following asymptotic expansions:
\[
\begin{aligned}
	\varphi(y,s_0) &= \kappa - e^{-s_0} h_4 + e^{-2s_0} \left( \frac{p}{2\kappa}y^8 - \frac{y^4 P(y)}{p-1} + \frac{(2-p)\kappa}{2(p-1)^2}P^2(y) \right) + O(e^{-3s_0}) \\
	&= \kappa \left(1 - \frac{e^{-s_0}}{p-1}P(y) - \frac{X}{p-1} \right) + O\left(e^{-s_0}XP(y) + X^2(1 + e^{-s_0}P(y))\right),
\end{aligned}
\]
\[
(ED)^{-1} = \frac{1}{p-1} - \frac{X}{p-1} - e^{-s_0}\frac{P(y)}{p-1} + O\left(e^{-s_0}XP(y) + X^2(1 + e^{-s_0}P(y))\right),
\]
and
\[
\chi\left(\frac{1 + y e^{-s_0/2}}{\vep_0}\right) = \chi\left(\frac{1}{\vep_0}\right) + O\left(\left(\frac{y e^{-s_0/2}}{\vep_0}\right)^n\right).
\]
\end{appendix}

%\begin{equation}
	%\text{Series expansion of } (1 + by)^{-\frac{1}{p-1}}:
	%+ \left(\frac{b^2}{2(p - 1)^2} + \frac{b^2}{2p - 2}\right)y^2
	%- \frac{by}{p - 1} + 1
%\end{equation}
\bibliographystyle{alpha}
\bibliography{heat-20-02-2023}

\def\cprime{$'$}
\begin{thebibliography}{DDPW20}

\bibitem[BE89]{BEbook89}
J.~Bebernes and D.~Eberly.
\newblock {\em Mathematical problems from combustion theory}, volume~83 of {\em
  Applied Mathematical Sciences}.
\newblock Springer-Verlag, New York, 1989.

\bibitem[BK88]{BKcpam88}
M.~Berger and R.~V. Kohn.
\newblock A rescaling algorithm for the numerical calculation of blowing-up
  solutions.
\newblock {\em Comm. Pure Appl. Math.}, 41(6):841--863, 1988.

\bibitem[CGMN23]{CTNVJFA23}
C.~Collot, T.~Ghoul, N.~Masmoudi, and Van~Tien Nguyen.
\newblock Collapsing-ring blowup solutions for the {K}eller-{S}egel system in
  three dimensions and higher.
\newblock {\em J. Funct. Anal.}, 285(7):Paper No. 110065, 41, 2023.

\bibitem[Con78]{Conbook78}
C.~Conley.
\newblock {\em Isolated invariant sets and the {M}orse index}, volume~38 of
  {\em CBMS Regional Conference Series in Mathematics}.
\newblock American Mathematical Society, Providence, R.I., 1978.

\bibitem[DDPW20]{DDWIM20}
J.~D\'{a}vila, M.~Del~Pino, and J.~Wei.
\newblock Singularity formation for the two-dimensional harmonic map flow into
  {$S^2$}.
\newblock {\em Invent. Math.}, 219(2):345--466, 2020.

\bibitem[DNZ23a]{DNZMAMS20}
G.~K. Duong, N.~Nouaili, and H.~Zaag.
\newblock Construction of blow-up solutions for the complex ginzburg-landau
  equation with critical parameters.
\newblock {\em Mem. Amer. Math. Soc.}, 285, 2023.

\bibitem[DNZ23b]{DNZArxiv2022}
G.~K. Duong, N.~Nouaili, and H.~Zaag.
\newblock Modulation theory for the flat blow-up solutions of nonlinear heat
  equation.
\newblock {\em Commun. Pure Appl. Anal.}, 22(10):2925--2959, 2023.

\bibitem[DNZ24]{DNZ24}
Senhao Duan, Nejla Nouaili, and Hatem Zaag.
\newblock Collapsing-ring blowup solutions for the nonlinear heat equation,
  2024.

\bibitem[FK92]{FKcpam92}
S.~Filippas and R.~V. Kohn.
\newblock Refined asymptotics for the blowup of {$u_t-\Delta u=u^p$}.
\newblock {\em Comm. Pure Appl. Math.}, 45(7):821--869, 1992.

\bibitem[FL93]{FLaihn93}
S.~Filippas and W.~X. Liu.
\newblock On the blowup of multidimensional semilinear heat equations.
\newblock {\em Ann. Inst. H. Poincar{\'e} Anal. Non Lin{\'e}aire},
  10(3):313--344, 1993.

\bibitem[GK85]{GKcpam85}
Y.~Giga and R.~V. Kohn.
\newblock Asymptotically self-similar blow-up of semilinear heat equations.
\newblock {\em Comm. Pure Appl. Math.}, 38(3):297--319, 1985.

\bibitem[GMS04]{GMSmmas04}
Y.~Giga, S.~Matsui, and S.~Sasayama.
\newblock On blow-up rate for sign-changing solutions in a convex domain.
\newblock {\em Math. Methods Appl. Sci.}, 27(15):1771--1782, 2004.

\bibitem[HV93]{HVaihn93}
M.~A. Herrero and J.~J.~L. Vel{\'a}zquez.
\newblock Blow-up behaviour of one-dimensional semilinear parabolic equations.
\newblock {\em Ann. Inst. H. Poincar{\'e} Anal. Non Lin{\'e}aire},
  10(2):131--189, 1993.

\bibitem[Kap80]{KSJAM80}
A.~K. Kapila.
\newblock Reactive-diffusion system with arrhenius kinetics: Dynamic of
  ignition.
\newblock {\em SIAM J. Appl. Math.}, 39:21--36, 1980.

\bibitem[KP80]{KPsiam80}
D.~R. Kassoy and J.~Poland.
\newblock The thermal explosion confined by a constant temperature boundary.
  {I}. {T}he induction period solution.
\newblock {\em SIAM J. Appl. Math.}, 39(3):412--430, 1980.

\bibitem[MNZ16]{MNZNon2016}
F.~Mahmoudi, N.~Nouaili, and H.~Zaag.
\newblock Construction of a stable periodic solution to a semilinear heat
  equation with a prescribed profile.
\newblock {\em Nonlinear Anal.}, 131:300--324, 2016.

\bibitem[MRS20]{MRSIMRNI20}
F.~Merle, P.~Rapha\"{e}l, and J.~Szeftel.
\newblock On strongly anisotropic type {I} blowup.
\newblock {\em Int. Math. Res. Not. IMRN}, (2):541--606, 2020.

\bibitem[MZ97]{MZdm97}
F.~Merle and H.~Zaag.
\newblock Stability of the blow-up profile for equations of the type
  {$u_t=\Delta u+\vert u\vert ^{p-1}u$}.
\newblock {\em Duke Math. J.}, 86(1):143--195, 1997.

\bibitem[MZ98]{MZcpam98}
F.~Merle and H.~Zaag.
\newblock Optimal estimates for blowup rate and behavior for nonlinear heat
  equations.
\newblock {\em Comm. Pure Appl. Math.}, 51(2):139--196, 1998.

\bibitem[MZ00]{MZma00}
F.~Merle and H.~Zaag.
\newblock A {L}iouville theorem for vector-valued nonlinear heat equations and
  applications.
\newblock {\em Math. Ann.}, 316(1):103--137, 2000.

\bibitem[MZ08]{MZjfa08}
N.~Masmoudi and H.~Zaag.
\newblock Blow-up profile for the complex {G}inzburg-{L}andau equation.
\newblock {\em J. Funct. Anal.}, 255(7):1613--1666, 2008.

\bibitem[MZ22]{MZIMRN22}
Frank Merle and Hatem Zaag.
\newblock Behavior rigidity near non-isolated blow-up points for the semilinear
  heat equation.
\newblock {\em Int. Math. Res. Not. IMRN}, (20):16196--16260, 2022.

\bibitem[MZ24]{MZJEMS24}
Frank Merle and Hatem Zaag.
\newblock On degenerate blow-up profiles for the subcritical semilinear heat
  equation.
\newblock {\em Journal of the European Mathematical Society}, 2024.

\bibitem[NZ18]{NZ2017}
N.~Nouaili and H.~Zaag.
\newblock Construction of a blow-up solution for the complex ginzburg-landau
  equation in some critical case.
\newblock {\em Arch. Rat. Mech. Anal}, 228(3), 2018.

\bibitem[QS07]{QSbook07}
P.~Quittner and P.~Souplet.
\newblock {\em Superlinear parabolic problems}.
\newblock Birkh{\"a}user Advanced Texts: Basler Lehrb{\"u}cher. [Birkh{\"a}user
  Advanced Texts: Basel Textbooks]. Birkh{\"a}user Verlag, Basel, 2007.
\newblock Blow-up, global existence and steady states.

\bibitem[Rap06]{Rdm06}
P.~Rapha\"{e}l.
\newblock Existence and stability of a solution blowing up on a sphere for an
  {$L^2$}-supercritical nonlinear {S}chr\"{o}dinger equation.
\newblock {\em Duke Math. J.}, 134(2):199--258, 2006.

\bibitem[RS09]{RSCMP09}
P.~Rapha\"{e}l and J.~Szeftel.
\newblock Standing ring blow up solutions to the {$N$}-dimensional quintic
  nonlinear {S}chr\"{o}dinger equation.
\newblock {\em Comm. Math. Phys.}, 290(3):973--996, 2009.

\bibitem[Vel93]{VELtams93}
J.~J.~L. Vel{\'a}zquez.
\newblock Classification of singularities for blowing up solutions in higher
  dimensions.
\newblock {\em Trans. Amer. Math. Soc.}, 338(1):441--464, 1993.

\end{thebibliography}

\end{document}